%% file: main.tex
\documentclass[a4paper,11pt,bibliography=totoc]{scrartcl}

\input{preamble.tex}

\newcommand{\email}[1]{\href{mailto:#1}{\texttt{#1}}}

\title{Probabilistic Hanna Neumann Conjectures}

\author{%
  Yotam Shomroni\thanks{%
    Weizmann Institute of Science, Rehovot, Israel. \email{yotam.shomroni@weizmann.ac.il}
  }%
  \\[0.4em]
  with an appendix joint with Noam Ta Shma\thanks{%
    Tel Aviv University, Israel.
  }%
}

\date{\today}
\begin{document}
\raggedbottom

\maketitle

\begin{abstract}
    We develop a theory of polymatroids on Stallings core graphs, and prove a
    gap theorem.
    As special cases of it,
    we prove the gap conjecture of Ernst-West, Puder, Seidel and the author on
    the stable $K$-primitivity rank (using a forthcoming result of Ernst-West and Puder);
    prove Reiter's conjecture on word-random subgroups of finite groups;
    improve Bowen's bounds on the $f$-invariant for p.m.p.\ actions of free
    groups in which a generating partition has a \noncyc stabilizer,
    deducing coalescence;
    and extend the proof of Wise's $w$-cycle conjecture by Louder and Wilton
    to not-necessarily-cyclic subgroups, which gives a new proof of this
    special case of the Friedman--Mineyev theorem
    (i.e.\ the Hanna Neumann conjecture).
    Louder and Wilton further improved the bound in Wise's conjecture
    via their \enquote{dependence theorem};
    in an appendix joint with Noam Ta Shma, we extend this improvement too
    to not-necessarily-cyclic subgroups, again as a special case of the main
    gap theorem.
    We also compute the exact decay rate of the expected number of invariant
    subsets, respectively subspaces, for word-random subgroups of $S_n$,
    respectively $\GLnFq$.
\end{abstract}

\setcounter{tocdepth}{2} %
\tableofcontents 

\input{intro.tex}
\input{applications.tex} %

\input{polymatroids.tex}

\input{fixed_points.tex}

\appendix
\input{puder_van_handel.tex}

\input{ta_shma_app.tex}

\input{appendix_jz.tex}
\input{polymatroid_table.tex}
\input{glossary.tex}

\printbibliography

\end{document}

%% file: preamble.tex
\usepackage[utf8]{inputenc}

\usepackage{longtable}

\usepackage{array}

\usepackage{makecell} %
\usepackage{multirow}

\usepackage{tikz}
\usetikzlibrary{positioning,calc,3d}

\usepackage{tikz-cd}
\usepackage{wrapfig}
\usepackage{subcaption}

\usepackage{xcolor}
\newcommand{\red}{\color{red}}
\newcommand{\blue}{\color{blue}}

\usepackage{amsmath,amssymb}

\usepackage{amsthm}

\usepackage{dsfont}

\usepackage{mathtools}

\usepackage[scr=boondox]{mathalpha}

\usepackage[new]{old-arrows}

 \usepackage{comment}

\usepackage[
backend=biber,
style=alphabetic,  %
sorting=nyt,
giveninits=true,   %
doi=false,
isbn=false,
url=false,
backref=true       %
]{biblatex}

\DefineBibliographyStrings{english}{%
  backrefpage = {p.},%
  backrefpages = {pp.},%
}

\DeclareSourcemap{
  \maps[datatype=bibtex]{
    \map[overwrite]{\step[fieldsource=journal,
      match=\regexp{arXiv\s+preprint\s+arXiv:\s*}, replace={arXiv:}]}
    \map[overwrite]{\step[fieldsource=journal,
      match=\regexp{arXiv\s+preprint:\s*}, replace={arXiv:}]}
    \map[overwrite]{\step[fieldsource=journal,
      match=\regexp{Electronic\s+Research\s+Announcements\s+of\s+the\s+American\s+Mathematical\s+Society}, replace={Electron. Res. Announc. Amer. Math. Soc.}]}
    \map[overwrite]{\step[fieldsource=journal,
      match=\regexp{Transactions\s+of\s+the\s+American\s+Mathematical\s+Society}, replace={Trans. Amer. Math. Soc.}]}
    \map[overwrite]{\step[fieldsource=journal,
      match=\regexp{Proceedings\s+of\s+the\s+American\s+Mathematical\s+Society}, replace={Proc. Amer. Math. Soc.}]}
    \map[overwrite]{\step[fieldsource=journal,
      match=\regexp{Journal\s+of\s+the\s+American\s+Mathematical\s+Society}, replace={J. Amer. Math. Soc.}]}
    \map[overwrite]{\step[fieldsource=journal,
      match=\regexp{Journal\s+of\s+Combinatorial\s+Theory,\s+Series[~\s]A}, replace={J. Combin. Theory Ser. A}]}
    \map[overwrite]{\step[fieldsource=journal,
      match=\regexp{Journal\s+of\s+Combinatorial\s+Theory}, replace={J. Combin. Theory}]}
    \map[overwrite]{\step[fieldsource=journal,
      match=\regexp{Journal\s+of\s+the\s+London\s+Mathematical\s+Society}, replace={J. London Math. Soc.}]}
    \map[overwrite]{\step[fieldsource=journal,
      match=\regexp{Bulletin\s+of\s+the\s+London\s+Mathematical\s+Society}, replace={Bull. London Math. Soc.}]}
    \map[overwrite]{\step[fieldsource=journal,
      match=\regexp{International\s+Journal\s+of\s+Algebra\s+and\s+Computation}, replace={Internat. J. Algebra Comput.}]}
    \map[overwrite]{\step[fieldsource=journal,
      match=\regexp{International\s+Mathematics\s+Research\s+Notices}, replace={Int. Math. Res. Not. IMRN}]}
    \map[overwrite]{\step[fieldsource=journal,
      match=\regexp{IEEE\s+Transactions\s+on\s+Information\s+Theory}, replace={IEEE Trans. Inform. Theory}]}
    \map[overwrite]{\step[fieldsource=journal,
      match=\regexp{Rendiconti\s+del\s+Seminario\s+Matematico\s+della\s+Universit.*\s+di\s+Padova}, replace={Rend. Sem. Mat. Univ. Padova}]}
    \map[overwrite]{\step[fieldsource=journal,
      match=\regexp{Journal\s+f.*r\s+die\s+reine\s+und\s+angewandte\s+Mathematik}, replace={J. Reine Angew. Math.}]}
    \map[overwrite]{\step[fieldsource=journal,
      match=\regexp{Advances\s+in\s+Applied\s+Mathematics}, replace={Adv. in Appl. Math.}]}
    \map[overwrite]{\step[fieldsource=journal,
      match=\regexp{Advances\s+in\s+Mathematics}, replace={Adv. Math.}]}
    \map[overwrite]{\step[fieldsource=journal,
      match=\regexp{Inventiones\s+mathematicae}, replace={Invent. Math.}]}
    \map[overwrite]{\step[fieldsource=journal,
      match=\regexp{Mathematische\s+Zeitschrift}, replace={Math. Z.}]}
    \map[overwrite]{\step[fieldsource=journal,
      match=\regexp{Mathematische\s+Annalen}, replace={Math. Ann.}]}
    \map[overwrite]{\step[fieldsource=journal,
      match=\regexp{Journal\s+of\s+Algebra}, replace={J. Algebra}]}
    \map[overwrite]{\step[fieldsource=journal,
      match=\regexp{Journal\s+of\s+Mathematical\s+Cryptology}, replace={J. Math. Cryptol.}]}
    \map[overwrite]{\step[fieldsource=journal,
      match=\regexp{Canadian\s+Mathematical\s+Bulletin}, replace={Canad. Math. Bull.}]}
    \map[overwrite]{\step[fieldsource=journal,
      match=\regexp{Canadian\s+Journal\s+of\s+Mathematics}, replace={Canad. J. Math.}]}
    \map[overwrite]{\step[fieldsource=journal,
      match=\regexp{Israel\s+Journal\s+of\s+Mathematics}, replace={Israel J. Math.}]}
    \map[overwrite]{\step[fieldsource=journal,
      match=\regexp{Forum\s+of\s+Mathematics,\s+Sigma}, replace={Forum Math. Sigma}]}
    \map[overwrite]{\step[fieldsource=journal,
      match=\regexp{Illinois\s+Journal\s+of\s+Mathematics}, replace={Illinois J. Math.}]}
    \map[overwrite]{\step[fieldsource=journal,
      match=\regexp{Pacific\s+Journal\s+of\s+Mathematics}, replace={Pacific J. Math.}]}
    \map[overwrite]{\step[fieldsource=journal,
      match=\regexp{Duke\s+Mathematical\s+Journal}, replace={Duke Math. J.}]}
    \map[overwrite]{\step[fieldsource=journal,
      match=\regexp{Annals\s+of\s+Mathematics}, replace={Ann. of Math.}]}
    \map[overwrite]{\step[fieldsource=journal,
      match=\regexp{Applied\s+Categorical\s+Structures}, replace={Appl. Categ. Structures}]}
    \map[overwrite]{\step[fieldsource=journal,
      match=\regexp{Algebra\s+.*\s+Number\s+Theory}, replace={Algebra Number Theory}]}
    \map[overwrite]{\step[fieldsource=journal,
      match=\regexp{Geometry\s+.*\s+Topology}, replace={Geom. Topol.}]}
    \map[overwrite]{\step[fieldsource=journal,
      match=\regexp{Mathematical\s+Programming}, replace={Math. Program.}]}
    \map[overwrite]{\step[fieldsource=journal,
      match=\regexp{Mathematica\s+Scandinavica}, replace={Math. Scand.}]}
    \map[overwrite]{\step[fieldsource=journal,
      match=\regexp{Geometriae\s+Dedicata}, replace={Geom. Dedicata}]}
    \map[overwrite]{\step[fieldsource=journal,
      match=\regexp{Discrete\s+mathematics}, replace={Discrete Math.}]}
    \map[overwrite]{\step[fieldsource=journal,
      match=\regexp{the\s+electronic\s+journal\s+of\s+combinatorics}, replace={Electron. J. Combin.}]}
    \map[overwrite]{\step[fieldsource=journal,
      match=\regexp{Enseign\.\s+Math\.\(2\)}, replace={Enseign. Math.}]}
    \map[overwrite]{\step[fieldsource=journal,
      match=\regexp{L.Enseignement\s+Math.*}, replace={Enseign. Math.}]}
    \map[overwrite]{\step[fieldsource=journal,
      match=\regexp{Ergodic\s+Theory\s+and\s+Dynamical\s+Systems}, replace={Ergodic Theory Dynam. Systems}]}
    \map[overwrite]{\step[fieldsource=journal,
      match=\regexp{Groups,\s+Geometry,\s+and\s+Dynamics}, replace={Groups Geom. Dyn.}]}
    \map[overwrite]{\step[fieldsource=journal,
      match=\regexp{Journal\s+d.Analyse\s+Math.*}, replace={J. Anal. Math.}]}
  }
}

\usepackage{todonotes}
\usepackage[inline]{enumitem}

 \usepackage {hyperref}
 \hypersetup{
    colorlinks=true,
    linkcolor=blue,
    filecolor=magenta,      
    urlcolor=cyan,
    citecolor=magenta,
    pdftitle={Probabilistic Hanna Neumann Conjectures},
    }

\usepackage{placeins} %

\usepackage{etoolbox} %

\usepackage{xspace} %

\usepackage[legalpaper, a4paper,
margin=1in]{geometry}

\usepackage{graphicx} %
\graphicspath{{figures/}}
\usepackage{quiver}

\usepackage{booktabs}
\usepackage{tabularx}

\DeclareMathAlphabet{\myscr}{U}{BOONDOX-calo}{m}{n}
\SetMathAlphabet{\myscr}{bold}{U}{BOONDOX-calo}{b}{n}

\theoremstyle{plain}
\newtheorem{theorem}{Theorem}[section]
\newtheorem*{theorem*}{Theorem}
\newtheorem{corollary}[theorem]{Corollary}
\newtheorem{lemma}[theorem]{Lemma}
\newtheorem{proposition}[theorem]{Proposition}

\theoremstyle{definition}
\newtheorem{definition}[theorem]{Definition}
\newtheorem{example}[theorem]{Example}
\newtheorem{question}[theorem]{Question}

\theoremstyle{remark}
\newtheorem{remark}[theorem]{Remark}

\newtheorem{observation}[theorem]{Observation}

\numberwithin{equation}{section}

\DeclareMathOperator{\B}{\mathbb{B}}

\DeclareMathOperator{\E}{\mathbb{E}}
\DeclareMathOperator{\EX}{\mathbb{E}} %
\DeclareMathOperator{\F}{\mathbb{F}}

\DeclareMathOperator{\N}{\mathbb{N}}
\DeclareMathOperator{\PR}{\mathbb{P}} %
\DeclareMathOperator{\Q}{\mathbb{Q}}
\DeclareMathOperator{\R}{\mathbb{R}}
\DeclareMathOperator{\Ss}{\mathbb{S}}
\DeclareMathOperator{\T}{\mathbb{T}}

\newcommand{\Z}{\mathbb{Z}} %

\DeclarePairedDelimiter{\ceil}{\lceil}{\rceil}

\DeclarePairedDelimiterX{\inner}[1]{\langle}{\rangle}{#1}
\DeclarePairedDelimiterX{\brackets}[1]{[}{]}{#1}
\DeclarePairedDelimiterX{\braces}[1]{\{}{\}}{#1}
\DeclarePairedDelimiterX{\prn}[1]{(}{)}{#1}
\DeclarePairedDelimiterX{\abs}[1]{|}{|}{#1}

\let\origbigvee\bigvee
\renewcommand{\bigvee}{\origbigvee\nolimits}

\newcommand{\enquote}[1]{``{#1}"} %

\newcommand{\defeq}{\overset{\mathrm{def}}{=\joinrel=}}
\newcommand{\one}{ \mathds{1} }
\robustify{\underset}
\def\wrt{with respect to }
\def\multiset#1#2{\ensuremath{\left(\kern-.3em\left(\genfrac{}{}{0pt}{}{#1}{#2}\right)\kern-.3em\right)}}

\newcommand{\immerse}{\looparrowright}

\newcommand{\pibar}{\overline{\pi}}
\newcommand{\spivarbar}[1]{s\overline{\pi}_{#1}}

\newcommand{\Crit}{\textup{Crit}}

\newcommand{\scl}{\textup{scl}}

\def\DecompB{\textup{Decomp}_{B}}

\newcommand{\spn}{\textup{span}}
\newcommand{\GLvarFq}[1]{\textup{GL}_{#1}(\F_q)}
\newcommand{\GLnFq}{\textup{GL}_n(\F_q)}
\newcommand{\GL}{\textup{GL}}
\newcommand{\Grasm}{\textup{\textbf{Gr}}}
\newcommand{\PSL}{\textup{PSL}}
\newcommand{\Inter}{\textup{Inter}}
\newcommand{\Interinj}{\textup{Inter}^{\textup{inj}}}

\newcommand{\Aut}{\textup{Aut}}

\newcommand{\rk}{\textup{rk}}

\newcommand{\Hmod}{\ensuremath{(H\le\FF)}\textup{-module}}
\newcommand{\Hom}{\textup{Hom}}

\newcommand{\stab}{\textup{stab}}

\newcommand{\acts}{\curvearrowright} %
\newcommand{\catC}{\ensuremath{\mathbf{C}}}
\newcommand{\HomCinj}{\Hom^{\mathrm{inj}}_{\mathbf{C}}}

\newcommand{\alphaSimUHomFSn}{\alpha\sim U(\Hom(\FF, S_n))}
\newcommand{\alphaSimUHomFVar}[1]{\alpha\sim U(\Hom(\FF, {#1}))}
\newcommand{\EHtoFVarActsVar}[2]{\EX_{H\to \FF}\brackets*{{#1}\acts {#2}}}
\newcommand{\EHtoFSnActsND}{\EHtoFVarActsVar{S_n}{\binom{[n]}{d}}}
\newcommand{\EHtoFSnActsNDG}{\EHtoFVarActsVar{S_n}{[n]_d / G}}

\newcommand{\sgn}{\textup{sgn}}

\newcommand{\Img}{\textup{Im}}
\newcommand{\src}{\mathfrak{s}}
\newcommand{\tar}{\mathfrak{t}}
\newcommand{\hp}{\myscr{h}}
\newcommand{\Binv}{$B$-invariant\xspace}
\newcommand{\noncyc}{non-cyclic\xspace}
\newcommand{\Cay}{\operatorname{Cay}}
\makeatletter
\newcommand{\wrapinproof}{\let\par\@@par}
\patchcmd{\WF@putfigmaybe}{\global\WF@floatfalse \pagebreak}
  {\ifWF@float\else\global\WF@floatfalse\pagebreak\fi}{}
  {\PackageWarningNoLine{preamble}{The wrapfig patch failed}}
\makeatother
 
\newcommand{\FF}{\textbf{\textup{F}}}

\newcommand{\calD}{\mathcal{D}}

\newcommand{\two}{{\color{red}2}\xspace}

%% file: intro.tex
\section{Introduction}
\label{section_intro}

Throughout the paper, $\FF$ denotes a fixed free group of finite rank, with
a fixed basis $B=\{x,y,\ldots\}$.
In this paper we develop a theory of
polymatroids (a concept generalizing
both matroids and joint entropy)
which are defined on
Stallings core graphs.
Our definition of a $\Gamma$-polymatroid
(Definition~\ref{def_eta_polymatroid}) is, arguably, the main
contribution of this paper.
We prove a gap theorem
(Theorem~\ref{thm_Gamma_invariant_polymatroid_theorem})
within this new theory, and as special cases, we
\begin{itemize}
\item prove the gap conjecture of
Ernst-West, Puder, Seidel and the author about the
stable $K$-primitivity rank
\cite[Appendix]{puder2023stable},
using a theorem from
a forthcoming paper
\cite{ernstwestPuderCoresInPrep};
\item prove a conjecture of Reiter \cite{reiter19}
about
word-random subgroups of finite groups;
\item improve Bowen's bound
\cite[Theorem 1.7]{Bowen2010FactorBound}
on the $f$-invariant of a factor
of a measure preserving action of a free group
whenever some generating partition has a \noncyc stabilizer,
concluding that such actions are coalescent (that is, every endomorphism of the system is an isomorphism);
\item and extend the proof of Wise's $w$-cycle conjecture
\cite[Conjecture 3.3]{wise2003nonpositive}
by Louder and Wilton
\cite{louder2014stacking}
from cyclic subgroups to all finitely generated subgroups,
which gives a new proof of this special case of
the Friedman--Mineyev theorem (that is, of the Hanna Neumann conjecture).
\end{itemize}

Another special case of our main gap theorem,
which is joint work with Noam Ta Shma
(Appendix~\ref{appendix_ta_shma}),
is a generalization of the \enquote{dependence theorem}
by Louder and Wilton
\cite[Theorem 1.16]{louder2022negative},
which strengthens Wise's conjecture further.

The $\Gamma$-polymatroid theorem is not an abstraction aimed at generalizing
the applications above: it is the simplest solution we could find to Reiter's
conjecture, which requires treating the intricate dependencies that arise when
a random action of $\FF$ is restricted to a subgroup $H\le \FF$.
The other applications were discovered only afterwards.

\medskip

The content of Reiter's conjecture is a bound on the expected number of
fixed points of a finitely generated subgroup $H\le \FF$ under a random
homomorphism to a finite group $G$ acting on a set $X$.
For two special cases, we compute the following exact asymptotics:
Given a random homomorphism $\alpha$
from $\FF$ to a symmetric group $S_n$
or to a finite general linear group $\GLnFq$,
we find, for every $d\in\N$, the exact decay rate
of the expected number of $d$-element subsets, respectively
$d$-dimensional subspaces, invariant under $\alpha(H)$;
for $S_n$ we get
\begin{equation}
    \Pr_{\alphaSimUHomFSn}\prn*{\{1,\ldots, d\}\textup{ is }\alpha(H)\textup{-invariant}}
= \Theta\prn*{n^{-d\cdot \pibar(H)}}\qquad\textup{as }n\to\infty,
\label{eq_intro_prob_Sn_inv_subspace}
\end{equation}
where $\pibar(H)\defeq\min\{\rk(J): H\le J\le\FF\}$, and for $\GLnFq$ we get an exponent
which is a profinite invariant of subgroups in a free
group.

\medskip

In another two appendices, we record yet-unpublished observations by Doron Puder
and by Andrei Jaikin-Zapirain; the first confirming a conjecture of
Alex Lubotzky about generation of alternating groups by non-commuting words,
and the second confirming a conjecture of the author about the stability of
the compressed rank in free groups, which is essential for proving
the decay rate \eqref{eq_intro_prob_Sn_inv_subspace}.
We thank both of them for kindly allowing us to include these results here.

Ideas close to polymatroids on Stallings core graphs, which inspired the
current work, appear in Reiter's thesis \cite[Chapter 3]{reiter19} and in
\cite{louder2022negative}.
Let us now describe our results in slightly more detail.
\subsubsection*{The Hanna Neumann conjecture and Wise's conjecture}
This paper intertwines several stories, the most famous of which
is about the Hanna Neumann conjecture \cite{neumann1956intersection}
from 1956,
which was strengthened by Walter Neumann in \cite{neumann2006intersections}
(and in its \underline{s}trengthened form, denoted \underline{S}HNC)
and was finally proved by Friedman
\cite{friedman2015sheaves} and Mineyev
\cite{mineyev2012submultiplicativity} independently in 2011.
Another, closely related story is about Wise's $w$-cycle conjecture \cite[Conjecture 3.3]{wise2003nonpositive}
(see also \cite{wise2005coherence}), called also the
\enquote{rank-1 Hanna Neumann conjecture},
that was proved
by
Helfer and Wise \cite{helfer2016counting}
and
by Louder and Wilton \cite{louder2014stacking}
independently.

This conjecture can be rephrased as a gap problem
regarding a function $s\pi\colon\FF\to\R_{\ge 0}\cup\{\infty\}$, where
$s\pi$ is the \textbf{stable primitivity rank}
that was defined by
Wilton in
\cite[Definition 10.6]{wilton2022rational}
(see also \cite{puder2023stable}).
We give the definition of $s\pi$ in Definition~\ref{def_stable_primitivity_rank}, but
for now, we just mention that $s\pi(w)=0$ whenever $w=u^k$ is a proper power
in $\FF$ (i.e.\ $u\in\FF$ and $k\ge 2$), and otherwise $s\pi(w)\ge 1$;
this inequality is essentially
Wise's conjecture.
Puder and the author generalized $s\pi$ to be defined on
finitely generated subgroups $H$
of $\FF$, where $s\pi(\inner{w})=s\pi(w)$ for a word $w\in\FF$.
When $\rk(H)>1$, the gap $s\pi(H)\ge 1$ is an important special case of the
Hanna Neumann conjecture.
In Theorem~\ref{thm_spi_gap_unified} we show that the gap $s\pi(H)\ge 1$
is a special case of our main technical theorem (Theorem~\ref{thm_Gamma_invariant_polymatroid_theorem}),
thus giving a unified proof for the two cases of $\rk(H)>1$ and of $H=\inner{w}$
for a non-power $w$.

Denote $[g,h]\defeq ghg^{-1}h^{-1}$.
As another motivation, the
well-known
\textbf{stable commutator length} ($\scl$)
\cite{bavard1991longueur} of a word,
\[\scl(w)\defeq \lim_{n\to\infty}\frac{1}{n}\cdot \min\braces*{g\,\middle|\, \exists u_1,v_1,\ldots, u_g,v_g\in \FF:\,\, \prod_{i=1}^g [u_i, v_i]=w^n}\in\R_{\ge 0}\cup\{\infty\}\]
satisfies $s\pi(w)\le 2\cdot \scl(w)$ for every $w\in\FF$ \cite[Corollary 4.10]{puder2023stable},
so if $w$ is not a proper power then $\scl(w)\ge 1/2$,
recovering the gap theorem from \cite[Corollary 5.2]{duncan1991genus}.

\subsubsection*{The stable $K$-primitivity rank}
The \textbf{stable} $K$\textbf{-primitivity rank}
$s\pi_K$ was defined in
\cite[Definition A.2]{puder2023stable} for words, where it was also
conjectured to admit a gap $s\pi_K(w)\ge 1$ for non-power words $w$,
as opposed to proper power words $w$ for which $s\pi_K(w)=0$.
We generalize $s\pi_K$ to subgroups in Definition~\ref{def_stable_K_primitivity_rank},
and show in Theorem~\ref{thm_spiK_gap} that
the gap $s\pi_K(H)\ge 1$, whenever $\rk(H)>1$ or $H=\inner{w}$
for a non-power $w$,
is another special case of Theorem~\ref{thm_Gamma_invariant_polymatroid_theorem},
thus proving a conjecture of
Ernst-West, Puder, Seidel and the author.
The case $\rk(H)>1$ is a weak $K$-analog of
the Hanna Neumann conjecture.
These four results
are summarized in Table~\ref{tab:conjectures}.

\begin{table}[ht]
    \centering
    \renewcommand{\arraystretch}{1.2}
    \begin{tabular}{lll}
        \toprule
        $\rk(H)$
        & $s\pi(H)\ge 1$ (Theorem~\ref{thm_spi_gap_unified})
        & $s\pi_K(H)\ge 1$ (Theorem~\ref{thm_spiK_gap}) \\
        \midrule
        $=1$
        & Wise's $w$-cycle conjecture \cite{wise2003nonpositive};
        & Conjectured in \cite{puder2023stable}. \\
        & proved by \cite{louder2014stacking}, \cite{helfer2016counting}. & \\
        \addlinespace
        $>1$
        & A special case of
        the SHNC.
        & New. Analogous to the HNC. \\
        \bottomrule
    \end{tabular}
    \caption{Context of our gap theorems for $s\pi$ and $s\pi_K$
    within the literature.}
    \label{tab:conjectures}
\end{table}

\subsubsection*{Reiter's conjecture}
The original motivation for this paper is Reiter's conjecture, about the
probability that a random homomorphism
$\alphaSimUHomFVar{G}$ from a free group to a finite group
maps a specific subgroup $H\le \FF$ to a specific subgroup of $G$.
We prove an upper bound on this probability
(for every finite group $G$ and every subgroup of it),
thus confirming a conjecture of Reiter \cite{reiter19};
this is another special case of
Theorem~\ref{thm_Gamma_invariant_polymatroid_theorem}.
It is fruitful to think of the subgroup of $G$ as the stabilizer of a point
with respect to a transitive action $G\acts X$, and then the bound
(Theorem~\ref{thm_reiter_conjecture}) is
\begin{equation}
    \label{eq_intro_reiter}
    \E_{\alpha}\brackets*{\abs*{X^{\alpha(H)}}}
    \;\le\;\frac{\abs*{X^{t}/G}}{\abs*{X}}
    \qquad\textup{for }1<\rk(H)<\infty,
\end{equation}
where $X^{\alpha(H)}\subseteq X$ is the set of
$\alpha(H)$-fixed points, $t\defeq\abs*{V(\Gamma_B(H))}$ is the number of vertices of the
Stallings core graph of $H$ (Definition~\ref{def_stallings_core_graph}), and
$X^t/G$ is the set of orbits of the diagonal action of $G$ on $X^t$.
The decay rate \eqref{eq_intro_prob_Sn_inv_subspace} and its $\GLnFq$
analog follow from Theorems~\ref{thm_Sn_spibard} and~\ref{thm_spibarqd}, and the opening of
Section~\ref{section_counting_fixed_sub_spaces} explains how one could extend their proofs to the
other classical groups.

These two computations show that Reiter's conjecture is not merely one more
special case of our main theorem, but is closely tied to the Hanna Neumann
conjecture.
Reiter's conjecture is best read as a family of conjectures, one for every
sequence of finite group actions: each is probabilistic, and each is
conjecturally equivalent to a bound on the rank of an intersection of free
objects.
The sequence $S_n\acts \binom{[n]}{d}$ corresponds to the strengthened Hanna
Neumann conjecture (SHNC) by Walter Neumann \cite{neumann2006intersections},
restricted to the case where the (generalized) intersection of the two
subgroups has finite index in one of them.
The $q$-analog of $\binom{[n]}{d}$ is the Grassmannian
$\Grasm_d(\F_q^n)$, the set of $d$-dimensional subspaces of $\F_q^n$.
The sequence $\GLnFq\acts \Grasm_d(\F_q^n)$ corresponds to an upper bound
\eqref{eq_q_HNC}
on the rank of the intersection of a submodule of $\F_q[\FF]^d$ with $\F_q[H]^d$, under a finite
dimension requirement; unlike the SHNC, this bound is false without that requirement.
We conjecture that further sequences of actions give rise, in a similar way,
to further upper bounds on the rank of an intersection of free objects; hence
the name \enquote{Probabilistic Hanna Neumann Conjectures}.

\subsubsection*{Bowen's $f$-invariant}
It is somewhat surprising that
Theorem~\ref{thm_Gamma_invariant_polymatroid_theorem}
is relevant also for measurable dynamics.
Let $\FF$ act on a probability space $(X,\mathcal{X},\mu)$ by measure
preserving transformations.
A partition $\alpha$ of $X$ into countably many measurable pieces is
\textbf{generating} if the pieces of $\alpha$ and of its translates $g\alpha$
(for $g\in \FF$) separate almost all points, and its \textbf{Shannon entropy} is
$\mathcal{H}_\mu(\alpha)\defeq-\sum_{C\in \alpha}\mu(C)\log\mu(C)$.
The \textbf{join} $\alpha\vee\beta$ of two partitions is the partition into
the sets $C\cap C'$ with $C\in \alpha$ and $C'\in \beta$, and
$\bigvee_{i=1}^k\alpha_i$ is defined inductively.
For such a partition $\alpha$ with $\mathcal{H}_\mu(\alpha)<\infty$, Bowen
\cite{Bowen2010Invariant} defined the $f$\textbf{-invariant}
\[ f_{\FF}(X,\mu)\defeq\inf_{n\ge 0}
\Phi_B\prn[\big]{\bigvee_{g\in \B_n}g\alpha},
\qquad
\Phi_B(\beta)\defeq (1-2|B|)\mathcal{H}_\mu(\beta)
+\sum_{b\in B}\mathcal{H}_\mu(\beta\vee b^{-1}\beta), \]
where $\B_n\subseteq \FF$ is the ball
$\braces*{b_1\cdots b_n\mid b_i\in B\cup B^{-1}\cup\{1\}}$;
$f_{\FF}(X,\mu)$ does not depend on the choice of $\alpha$
\cite[Theorem 5.3]{Bowen2009AbramovRohlin}, and with it Bowen distinguished
Bernoulli shifts over $\FF$ \cite{Bowen2010Invariant}, answering
a question of Ornstein and Weiss.
For a \textbf{factor} $\pi\colon (X,\mathcal{X},\mu)\to (Y,\mathcal{Y},\nu)$, that is, a
quotient by an $\FF$-invariant $\sigma$-algebra, Bowen
\cite[Theorem 1.7]{Bowen2010FactorBound} proved that
$f_{\FF}(Y,\nu)\ge f_{\FF}(X,\mu)-\mathcal{H}_\mu(\alpha)$ when $\alpha$ is finite
and $Y$ has a finite generating partition.
We show (Theorem~\ref{thm_entropic_stabilizer_gap}
and Corollary~\ref{thm_f_factor_monotonicity}) that if $f_{\FF}(X,\mu)$ and $f_{\FF}(Y,\nu)$ are
finite real numbers and the stabilizer $\stab_{\FF}(\alpha)\defeq\braces*{g\in \FF: g\alpha=\alpha}$ is
\noncyc,\footnote{
    Here $g\alpha\defeq\braces*{g.C: C\in \alpha}$, so $g\alpha=\alpha$ means
    that $g$ permutes the pieces of $\alpha$.
} then
\[ f_{\FF}(Y,\nu)\ge f_{\FF}(X,\mu)
+\mathcal{H}_\mu\prn*{\alpha\mid \pi^{-1}(\mathcal{Y})}, \]
where $\mathcal{H}_\mu(\cdot\mid\cdot)$ is the conditional Shannon entropy
\eqref{eq_conditional_entropy}, which is strictly positive unless $\pi$ is an
isomorphism modulo null sets.
In particular such an action $\FF\acts(X,\mathcal{X},\mu)$ is
\textbf{coalescent}:
every measure preserving $\FF$-equivariant endomorphism of it is invertible
modulo null sets.
The stabilizer hypothesis holds, for instance, for every stationary process on
$A^{H\backslash \FF}$ with $H\le \FF$ \noncyc and with finite entropy at each coordinate.
There is an exact earlier result for finite-to-one factors: if $X$ is ergodic and
$X\to Y$ is almost everywhere $n$-to-one, then
$f_{\FF}(Y,\nu)=f_{\FF}(X,\mu)+(\rk(\FF)-1)\log n$ whenever both sides are
defined \cite[Corollary 1.4]{Bowen2009AbramovRohlin}. Hence, if $\rk(\FF)>1$, every
almost everywhere $n$-to-one endomorphism of an ergodic action with finite
$f$-invariant satisfies $n=1$; Corollary~\ref{thm_f_factor_monotonicity} is
stated for all endomorphisms, finite-to-one or not.
Austin, Glasner, Thouvenot and Weiss proved that an ergodic $\Z$-system of finite
positive entropy is never coalescent
\cite[Theorem 1.4]{AustinGlasnerThouvenotWeiss2023}.
In contrast, every action of $\FF$ with finite $f$-invariant and with a generating partition
$\alpha$ such that $\mathcal{H}_\mu(\alpha)<\infty$ and $\rk(\stab_{\FF}(\alpha))\ge 2$ is
coalescent.

\subsubsection*{Stackable subgroups}
As part of the proof of
our main theorem, we provide a lemma (Lemma~\ref{lemma_stackable_subgroup})
regarding stackable
subgroups of $\FF$ (Definition~\ref{def_stacking_topological}, following
\cite{louder2014stacking}), which may be of independent interest:
we show that every \noncyc subgroup $H\le \FF$ has a
\noncyc subgroup $J\le H$ which is stackable over $\FF$
(and in fact, $J$ can have arbitrary rank).

\subsection*{The \texorpdfstring{$\Gamma$}{Gamma}-polymatroid theorem (invariant version)}
\label{subsec_Gamma_poly_thm}

We now introduce
a simplified version of our main theorem,
which is still general enough to imply
the applications in Section~\ref{section_Gamma_poly_Thm_and_Apps},
but not the results in Appendix~\ref{appendix_ta_shma}.
During Sections~\ref{section_intro}--\ref{section_Gamma_poly_Thm_and_Apps},
we adopt a combinatorial
approach to graphs, using labelings,
for simplicity; in the rest
of the paper we work in the slightly more general
framework of immersions between Serre graphs \cite{serre2002trees}.

Given a finite connected graph $\Gamma$
with a distinguished vertex $v_0$,
the fundamental group $\pi_1(\Gamma, v_0)$
is free of rank $1 - |V(\Gamma)| + |E(\Gamma)|$.
In order to study subgroups of $\FF$,
Stallings \cite[5.4 and 7.2]{Stallings1983TopologyOF} introduced
a family of
graphs with $B$-labeled edges, called \enquote{Stallings core graphs}.
Stallings' graphs are rooted at a vertex, the only vertex that may have degree less than
$2$; we give a non-rooted version, which corresponds to conjugacy classes of subgroups.

\begin{definition}
    \label{def_stallings_core_graph}
    A $B$\textbf{-graph} is a directed graph $\Gamma$ with $B$-labels
    (letters) on the edges,
    such that
    no two incident edges
    have the same label and direction;
    loops and multiple edges with the same endpoints are allowed.
    The \textbf{degree} of a vertex is the number of edge-ends at it.
    A finite $B$-graph is a \textbf{core graph} if every vertex has degree at
    least $2$, and a \textbf{Stallings core graph} if, moreover, it is non-empty
    and connected.
    The \textbf{core} of a $B$-graph is the union of all its non-trivial cyclically
    non-backtracking closed walks; for a finite $B$-graph, it is the largest subgraph that is a
    core graph, obtained by removing the tree components and the hanging trees.
\end{definition}

We denote by $\Omega_B$ the
Stallings core graph
with one vertex
$v_B$ and a $b$-labeled
(directed) loop for every $b\in B$
(a \enquote{bouquet}).
Every $B$-graph $\Gamma$ admits a unique
label-preserving map $\eta\colon\Gamma\to\Omega_B$.
Given a vertex $v_0\in V(\Gamma)$,
the induced map on fundamental groups
$\eta_*\colon\pi_1(\Gamma,v_0)
\to \pi_1(\Omega_B,v_B)=\FF$
is injective, as every word $w\in\FF$ can be read
from $v_0$ along at most one path, so we
identify $\pi_1(\Gamma,v_0)$ with a subgroup of $\FF$.
For example, in
Example~\ref{example_counting_graph_morphisms}
there are 2 Stallings core graphs
with labels
$B=\{{\red x},{\blue y}\}$;
the graph $\Gamma$
has $\pi_1(\Gamma,v_0)
=\inner{{\red x}{\blue y}{\red x},
{\blue y}{\red x^2}}\le\FF$.
There is a bijective correspondence
$\Gamma\leftrightarrow \eta_*\pi_1(\Gamma)^{\FF}$
between
Stallings core graphs as in Definition~\ref{def_stallings_core_graph}
and conjugacy classes of non-trivial finitely generated subgroups of $\FF$.\footnote{
Stallings' rooted graphs are in bijective correspondence with the finitely
generated subgroups of $\FF$, the trivial subgroup corresponding to a single vertex.
}
We denote by $\Gamma_B(H)$ the Stallings core graph of a non-trivial finitely
generated subgroup $H\le\FF$: the core of the \textbf{Schreier graph}
$H\backslash\textup{Cay}(\FF,B)$, which is the graph with vertices
$\braces*{Hg: g\in \FF}$ and a $b$-labeled edge $Hg\to Hgb$ for every
$b\in B$ and $g\in \FF$.
We define $\Gamma_w\defeq\Gamma_B(\inner{w})$ for
$1\ne w\in \FF$; note that $\Gamma_w$ is topologically a circle $\Ss^1$.
Given a
$B$-graph $\Gamma$,
let us denote by $\src,\tar\colon E(\Gamma)\to V(\Gamma)$
the \textbf{s}ource and \textbf{t}arget maps,
that map an edge to its endpoints.
For every $b\in B$, denote by $E_b(\Gamma)$
the set of $b$-labeled edges;
by definition, the restrictions
of $\src,\tar$ to $E_b(\Gamma)$
are injective.

Polymatroids were defined in \cite{edmonds1970submodular};
see also \cite{schrijver2003combinatorial}
for a comprehensive reference.

\begin{definition}
    \label{def_polymatroid}
    A \textbf{polymatroid}
    on a set $V$ is a function
    $\hp\colon 2^{V}\to \R$ satisfying
    $\hp(\emptyset)=0$, which is increasing
    (if $A\subseteq B$ then $\hp(A)\le \hp(B)$)
    and submodular:
    $\hp(A)
    +\hp(B)
    \ge \hp(A\cup B)
    +\hp(A\cap B)$.
\end{definition}

For a polymatroid $\hp\colon 2^{V}\to \R$ and $V'\subseteq V$, 
we denote by $\hp\restriction_{V'} \colon 2^{V'}\to \R$ the restriction of $\hp$ to $V'$.

\begin{example}
    \label{example_polymatroids_as_unions}
    If $\{S_i\}_{i\in V}$
    is a family of
    subsets of a finite set, and
    $\hp(A) = |\bigcup_{i\in A} S_i|$ for every $A\subseteq V$,
    or $\{S_i\}_{i\in V}$ is
    a family of subspaces of a finite dimensional vector space,
    and $\hp(A) = \dim(\sum_{i\in A} S_i)$,
    then $\hp$ is a polymatroid on $V$.
\end{example}

The following definition appears in the literature
under the terms
\enquote{tight} \cite[Section 3]{matuvs2016entropy},
\enquote{compact} {\cite[Lemma 3.3]{jowett2016connectivity}}
and \enquote{coloopless} (in the context of matroids).

\begin{definition}
    \label{def_tight_polymatroid}
    A polymatroid $\hp$ on $V$ is called
    \textbf{tight}
    if
    $\hp(V\setminus \{v\})
    =\hp(V)$
    for every $v\in V$. %
\end{definition}

\begin{example}
    \label{example_tight_polymatroids}
If $\{S_i\}_{i\in V}$
is a family of
subsets of a finite set, then
$\hp(A) = |\bigcup_{i\in A} S_i|$
is tight if and only if every $S_i$ is contained in $\bigcup_{j\ne i} S_j$,
or equivalently,
every element of the union $\bigcup_{i\in V} S_i$
belongs to at least two of the subsets $S_i$.

As another example, if $\{S_i\}_{i\in V}$
is a family of vectors in a vector space, then
$\hp(A) = \dim\spn\{S_i\}_{i\in A}$
is tight if and only if
every vector is dependent on the others.
More generally, if $\{S_i\}_{i\in V}$
is a family of subspaces of a finite dimensional vector space,
then $\hp(A) = \dim(\sum_{i\in A} S_i)$
is tight if and only if every subspace is contained in the sum of the others.
\end{example}

The following definition is fundamental in this work.
Let $\Gamma$ be a Stallings core graph.
We delay the definition of a
$\Gamma$-polymatroid to
Definition~\ref{def_eta_polymatroid},
as the special case of a \textbf{\Binv $\Gamma$-polymatroid}
is much simpler.

\begin{definition}
    \label{def_B_invariant_polymatroid}
    A \textbf{\Binv $\Gamma$-polymatroid}
    is a polymatroid $\hp$ on $V(\Gamma)$, such that
    for every
    $b\in B$ and every subset
    $E'\subseteq E_b(\Gamma)$ of edges,
    $\hp(\src(E'))=\hp(\tar(E'))$.
    The \textbf{Euler characteristic}\footnote{
        Not to be confused with the characteristic polynomial of a
        polymatroid, which is often similarly denoted, but does not appear in this paper.
    } of $\hp$ is
        \[\chi(\hp)\defeq \hp(V(\Gamma))-\sum_{b\in B} \hp(\src(E_b(\Gamma))).\]
\end{definition}

It is natural to call such a polymatroid $\FF$-invariant, but this would
collide with Bowen's $f$-invariant (Section~\ref{section_f_invariant});
moreover, unless $\Gamma$ is the whole Schreier graph, $\FF$ acts on the
subsets of $V(\Gamma)$ only partially.

The following theorem is the main technical result in the paper.
In Theorem~\ref{thm_Gamma_polymatroid_theorem} we formulate a more general
statement, without the $B$-invariance hypothesis.

\begin{theorem}
    \label{thm_Gamma_invariant_polymatroid_theorem}
    Let $\Gamma$ be a
    Stallings core graph
    with fundamental group $H\le \FF$,
    and $\hp$ a
    \Binv
    $\Gamma$-polymatroid.
    Assume that either
    $\rk(H)>1$ (with no restriction on $\hp$), or
    $H=\inner{w}$ is generated by a non-power $w\in \FF$, and $\hp$ is tight.
    Then $\chi(\hp)\le-\hp(\{v\})$ for some $v\in V(\Gamma)$.
\end{theorem}

Note that since $\Gamma$ is connected, $\hp(\{v\})$ is the same for every
$v\in V(\Gamma)$.
Without the assumptions of Theorem~\ref{thm_Gamma_invariant_polymatroid_theorem}
on $H$ and on $\hp$, we still have $\chi(\hp)\le 0$ for
every \Binv $\Gamma$-polymatroid $\hp$ (Corollary~\ref{cor_chi_nonpositive}).

\subsection*{Overview of the paper}
\label{subsec_overview}

In Section~\ref{section_Gamma_poly_Thm_and_Apps} we derive the applications
from Theorem~\ref{thm_Gamma_invariant_polymatroid_theorem}, one per
subsection:
\begin{itemize}
\item Section~\ref{subsec_HNC}: the special case of the Hanna Neumann
conjecture and Wise's conjecture (Theorem~\ref{thm_spi_gap_unified}), where
we also introduce pullbacks, the compressed rank $\pibar$ and the
stable primitivity rank $s\pi$;
\item Section~\ref{subsec_spiK}: the stable $K$-primitivity rank, with the
$q$-analog \eqref{eq_q_HNC} of the HNC and the proof of
Theorem~\ref{thm_spiK_gap};
\item Section~\ref{subsec_Reiter}: Reiter's conjecture;
\item Section~\ref{section_f_invariant}: Bowen's $f$-invariant.
\end{itemize}
In Section~\ref{section_proof_of_Gamma_poly_thm} we state the general
$\Gamma$-polymatroid theorem (Theorem~\ref{thm_Gamma_polymatroid_theorem}),
without the \Binv hypothesis, and prove it.
In Section~\ref{section_counting_fixed_sub_spaces} we prove the fixed point
theorems, Theorems~\ref{thm_Sn_spibard} and \ref{thm_spibarqd}.
Appendix~\ref{appendix_random_words_gen_A_n} presents an observation by Puder
about generation of alternating groups by random word values.
Appendix~\ref{appendix_ta_shma}, joint with Noam Ta Shma, generalizes the
\enquote{dependence theorem} \cite[Theorem 1.16]{louder2022negative}.
Appendix~\ref{appendix_jaikin} presents a theorem by Jaikin-Zapirain about
the stability of the compressed rank.
From this theorem we derive versions (Corollaries~\ref{cor_stable_primitivity_rank}
and~\ref{cor_iehnc}) of two conjectures of Wilton, one on the stable primitivity rank, disproved by
Ta Shma \cite[Section 6.3]{TaShma26}, and one strengthening the Hanna Neumann conjecture, disproved by
Hunt \cite{hunt2015hanna}.
The paper ends with a
\hyperref[appendix_polymatroid_table]{table} of the $\Gamma$-polymatroids constructed in this paper
and a \hyperref[appendix_glossary]{glossary} of notation.

\subsection*{Acknowledgements}
I wish to thank Doron Puder, Alex Lubotzky, Danielle Ernst-West, and Matan Seidel for helpful
conversations. Special thanks go to Noam Ta Shma for his tremendous help; to Doron Puder and Andrei
Jaikin-Zapirain for the appendices; to Avichai Marmor for many hours of joint work in
2023 on Reiter's conjecture for $S_n$; and to Asael Reiter for his truly enthralling thesis.
This paper was supported by the European Research Council (ERC) under the European Union's
Horizon~2020 research and innovation programme (grant agreement No.~882751).

\paragraph{AI Disclosure.} Unless stated otherwise, all original ideas in this work are mine,
including important corrections to the previous version. Some non-critical corrections were
identified with the assistance of Claude Opus 4 and 5, which I used for LaTeX editing and
formatting. The connection to Bowen's work was found with some assistance from OpenAI's GPT-5.
I take full responsibility for the entire contents of this paper.

%% file: applications.tex
\section{Applications of the \texorpdfstring{$\Gamma$}{Gamma}-Polymatroid Theorem}
\label{section_Gamma_poly_Thm_and_Apps}

Each of the four subsections below derives one application of
Theorem~\ref{thm_Gamma_invariant_polymatroid_theorem}, by producing a
different \Binv $\Gamma$-polymatroid, in different contexts: graphs, modules
over a free group algebra, finite groups, and entropy.
Section~\ref{section_f_invariant} may be read on its own;
Sections~\ref{subsec_spiK} and~\ref{subsec_Reiter} depend on
Section~\ref{subsec_HNC}, and not on each other.

\subsection{The Hanna Neumann conjecture and Wise's conjecture}
\label{subsec_HNC}

Let $H, J \le \FF$ be non-trivial finitely generated subgroups.
By the Nielsen--Schreier theorem, $H$, $J$ and $H\cap J$ are also free.
Howson \cite{howson1954intersection} proved that the
intersection $H\cap J$ is also finitely generated,
and that $\rk(H\cap J)$ is bounded by a function of $\rk(H)$ and $\rk(J)$.
Neumann
\cites[p.~186]{neumann1956intersection}[p.~128]{neumann1957intersection}
improved the bound into
\begin{equation}
    \label{eq_HannaNeumannThm}
    \rk(H\cap J) - 1 \le \two (\rk(H)-1)(\rk(J)-1),
\end{equation}
and conjectured \cite[p.~186]{neumann1956intersection} that in fact, the \two is redundant:
\begin{equation}
    \label{eq:HNC}
    \rk(H\cap J) - 1 \le (\rk(H)-1)(\rk(J)-1).
\end{equation}
This conjecture \eqref{eq:HNC} has become known as the
\textbf{Hanna Neumann conjecture} (HNC).
The conjectured bound is attained:

\begin{example}
    If $H_n\defeq \ker(\inner{x,y}\to \Z/n)$ where
    $x, y\mapsto 1$, then $\rk(H_n) = n + 1$
    (for example, $H_2 = \inner{x^2,xy,y^2}$), and
    if $\textup{gcd}(m,n)=1$ then $H_n\cap H_m = H_{nm}$.
\end{example}

The HNC has a long history of partial results, including work by Burns,
Walter Neumann, Tardos, Dicks, Arzhantseva, Dicks--Formanek,
Khan, Meakin--Weil, Ivanov, Wise, and Dicks--Ivanov; see
\cite[Section 1]{mineyev2012submultiplicativity} and \cite[Section 2.2]{friedman2015sheaves} for
references.
After 55 years, the HNC was finally proved in 2011,
independently by Friedman \cite[Theorem 2.1]{friedman2015sheaves} and
Mineyev \cite[Theorem 25]{mineyev2012submultiplicativity},
whose two proofs were made public within the same week \cite[p.~12]{dicks2012joel}.
In fact, they proved Walter Neumann's \textbf{strengthened}
conjecture \cite[Question 2.2]{neumann2006intersections}.
Both proofs are highly non-trivial:
Friedman used sheaves on graphs (in a 100-page paper),
and Mineyev used Hilbert modules.
Both proofs were simplified by Dicks (see e.g.\ \cite[Theorem~3.4]{dicks2012joel}).

The strengthened conjecture is best stated using core graphs
(Definition~\ref{def_stallings_core_graph}), which need not be connected.
They are called \enquote{$B$-labeled multi core graphs} in \cite[Definition 3.1]{HP22}, the
prefix \enquote{multi} recording that several connected components are
allowed.
Note that every connected component of a core graph has non-positive Euler
characteristic $\chi=|V|-|E|$.
Before stating the SHNC, we give it another motivation:
counting $B$-graph \textbf{morphisms}, which are graph morphisms that
preserve edge labels and directions, or equivalently,
commute with the immersions to $\Omega_B$.

\begin{question}
    \label{question_simplest_d_morphisms}
    Let $d\in \N$ and $\Gamma$ be a connected
    core graph.
    Suppose we want another core graph
    $\Delta$ with $d$
    different
    morphisms
    $\Gamma \to \Delta$.
    How complicated,
    in the sense of Euler characteristic $\chi=|V|-|E|$,
    does $\Delta$ have to be?
\end{question}

    \begin{example}
    \label{example_counting_graph_morphisms}
    In the diagram below,
    the graph $\Gamma$ has 2
    morphisms
    to the graph $\Delta$, sending $v_0\in V(\Gamma)$
    to either $u_1$ or $u_2$, and $\chi(\Delta)=-2$.
    No simpler core graph has 2
    morphisms from $\Gamma$.
\[\begin{tikzcd}[ampersand replacement=\&]
	\Gamma \& {v_0} \& \bullet \& \Delta \& \bullet \& {u_1} \& \bullet \\
	\& \bullet \& \bullet \&\& \bullet \& {u_2}
	\arrow["{=}"{description}, draw=none, from=1-1, to=1-2]
	\arrow[red, "x", from=1-2, to=1-3]
	\arrow[blue, "y"', from=1-2, to=2-2]
	\arrow[blue, "y", from=1-3, to=2-3]
	\arrow[red, "x"', from=2-2, to=2-3]
	\arrow[red, "x", from=2-3, to=1-2]
	\arrow["{=}"{description}, draw=none, from=1-4, to=1-5]
	\arrow[red, "x"', from=1-5, to=2-5]
	\arrow[blue, "y"', from=1-6, to=1-5]
	\arrow[red, "x", from=1-6, to=2-6]
	\arrow[blue, "y"', from=1-7, to=1-6]
	\arrow[red, "x", from=2-5, to=1-6]
	\arrow[red, "x"', from=2-6, to=1-7]
	\arrow[blue, "y", from=2-6, to=2-5]
\end{tikzcd}\]
\end{example}

Question~\ref{question_simplest_d_morphisms}
can be rephrased using the following definition
of pullback.

\begin{definition}[{\cite[Section 4.1]{HP22}}]
\label{def_pullback}
Given two $B$-graphs $\Gamma$ and $\Delta$,
their fiber product over $\Omega_B$ has
vertices $V(\Gamma)\times V(\Delta)$
and $b$-labeled edges $E_b(\Gamma)\times E_b(\Delta)$;
it is a $B$-graph, which may have leaves.
The \textbf{pullback} $\Gamma\times_{\Omega_B}\Delta$
of $\Gamma$ and $\Delta$
is defined as the core of their fiber product.\footnote{
    It is possible that the pullback is the empty $B$-graph,
    with no vertices.
}
\end{definition}
\begin{figure}[t]
\centering
\begin{minipage}[c]{0.5\textwidth}
\centering
\begin{tikzcd}[ampersand replacement=\&]
	\& \bullet \& \bullet \& \bullet \& \bullet \\
	\bullet \& \circ \& \bullet \& \bullet \& \circ \\
	\bullet \& \circ \& \circ \& \bullet \& \bullet \\
	\bullet \& \bullet \& \bullet \& \circ \& \circ
	\arrow["{\color{red}x}"{swap},  red, from=1-2, to=1-3]
	\arrow["{\color{red}x}"{swap},  red, from=1-3, to=1-4]
	\arrow["{\color{blue}y}"',      blue, bend right=12, from=1-4, to=1-2]
	\arrow["{\color{red}x}"{swap},  red, from=1-4, to=1-5]
	\arrow["{\color{blue}y}"',      blue, bend right=12, from=1-5, to=1-3]
	\arrow["{\color{red}x}"',       red, from=2-1, to=3-1]
	\arrow["{\color{red}x}"{description}, red, dotted,   from=2-2, to=3-3]
	\arrow["{\color{red}x}"{description}, red,           from=2-3, to=3-4]
	\arrow["{\color{red}x}"{description}, red,           from=2-4, to=3-5]
	\arrow["{\color{blue}y}"',      blue, from=3-1, to=4-1]
	\arrow["{\color{blue}y}"{description}, blue, from=3-4, to=4-2]
	\arrow["{\color{blue}y}"{description}, blue, from=3-5, to=4-3]
	\arrow["{\color{red}x}",        red, bend left=12, from=4-1, to=2-1]
	\arrow["{\color{red}x}"{description}, red, from=4-2, to=2-3]
	\arrow["{\color{red}x}"{description}, red, from=4-3, to=2-4]
	\arrow["{\color{red}x}"{description}, red, dotted,   from=4-4, to=2-5]
\end{tikzcd}
\end{minipage}\hfill
\begin{minipage}[c]{0.46\textwidth}
\caption{Pullback of $B$-graphs. Removed tree components are denoted by
white vertices and dotted edges.}
\label{fig_pullback_of_labeled_graphs}
\end{minipage}
\end{figure}
See Figure~\ref{fig_pullback_of_labeled_graphs} for an example of
a pullback.
Given two pointed labeled connected graphs
$(\Gamma, u_0)$ and $(\Delta, v_0)$, closed paths
in the pullback $\Gamma\times_{\Omega_B}\Delta$
correspond to pairs of closed paths in $\Gamma$ and $\Delta$
\enquote{reading the same word}, so
    \[ \pi_1(\Gamma\times_{\Omega_B} \Delta, (u_0, v_0))
    = \pi_1(\Gamma, u_0) \cap \pi_1(\Delta, v_0). \]
This holds for the fiber product, and is defined for the pullback
if it contains $(u_0, v_0)$; in this case the intersection is
not trivial.
Given two core graphs $\Gamma$ and $\Delta$,
the number of morphisms
$\Gamma \to \Delta$ is equal to the number of sections of
$\Gamma$ in $\Gamma\times_{\Omega_B}\Delta$,
that is, right inverses
$\Gamma\to \Gamma\times_{\Omega_B}\Delta$
of the projection
$\Gamma\times_{\Omega_B}\Delta \to \Gamma$.
Indeed, any section $\Gamma\to \Gamma\times_{\Omega_B} \Delta$
can be composed with the projection
$\Gamma\times_{\Omega_B}\Delta \to \Delta$.
Conversely, given a morphism $\phi\colon \Gamma \to \Delta$,
the map $u\mapsto (u, \phi(u))$ is a section of $\Gamma$
in $\Gamma\times_{\Omega_B}\Delta$.
We are now ready to state the SHNC:
    \begin{theorem}[Friedman--Mineyev]
        \label{thm_friedman_mineyev}
        For any two core graphs
        $\Gamma$ and $\Delta$,
      \[
          -\chi(\Gamma\times_{\Omega_B}\Delta) \;\le\; \chi(\Gamma)\cdot\chi(\Delta).
      \]
    \end{theorem}

To relate the SHNC to the
counting problem of graph morphisms,
note that there are $d$ different graph morphisms
$\Gamma\to \Delta$ if and only if
$\Gamma\times_{\Omega_B} \Delta$
contains $d$ sections of $\Gamma$, in which case
\[ -\chi(\Gamma)\cdot d
\le -\chi(\Gamma\times_{\Omega_B} \Delta)
\le \chi(\Gamma)\cdot \chi(\Delta).\]
If $\chi(\Gamma)\neq 0$, we get
$d\le -\chi(\Delta)$: the graph $\Delta$
cannot be too simple.
We encode this
by
\begin{equation}
    \label{eq_trivial_stable_compressed_rank}
    \begin{split}
        \spivarbar{d}^{\textup{triv}}(\Gamma)
    &\defeq \min
    \braces*{ \frac{-\chi(\Delta)}{d}\middle|
    \begin{array}{ll}
        & \Delta \textup{ is a core graph,}\\
        & \textup{and } \abs*{\Hom(\Gamma, \Delta)} \ge d
    \end{array}} \\
    &= \min
        \braces*{ \frac{-\chi(\Delta)}{d} \middle|
        \begin{array}{ll}
            & \Delta \textup{ is a core graph, and }
            \Gamma\times_{\Omega_B} \Delta \\
            &
        \textup{contains the trivial }d\textup{-covering of }\Gamma
        \end{array}
    }
    \ \ge\ 1 ,
    \end{split}
\end{equation}
where by a trivial $d$-covering of $\Gamma$ we mean
    $d$ disjoint copies of $\Gamma$.
It is natural to drop the triviality requirement\footnote{
    For example, in Figure~\ref{fig_Sn_pairs}, there is no
    morphism $\Gamma\to\Delta$, but the morphism
    $\tilde{\Gamma}\to\Delta$ from a $2$-covering of $\Gamma$
    can be seen as \enquote{two morphisms
    $\Gamma\rightrightarrows\Delta$ that are glued together}.
},
and define:
\begin{definition}
    \label{def_spibard}
    Let $H = \pi_1(\Gamma) \le \FF$
    be finitely generated free groups.
    The $d$-\textbf{stable compressed rank} of $H$ is
    \begin{equation*}
        \spivarbar{d}(H)\defeq \min
        \braces*{ \frac{-\chi(\Delta)}{d} \middle|
        \begin{array}{ll}
            & \Delta \textup{ is a core graph, and }\\
            & \Gamma\times_{\Omega_B} \Delta
        \textup{ contains a }d\textup{-covering of }\Gamma
        \end{array}
    }.
    \end{equation*}
\end{definition}
By Theorem~\ref{thm_friedman_mineyev},
if $H$ is \noncyc, then $\spivarbar{d}(H)\ge 1$.
In \cite[Corollary 1.5]{jaikin2024free}, Jaikin-Zapirain defined an invariant
of groups that measures how \enquote{compressed} the group is:
\begin{equation}
    \label{eq_compressed_rank}
    \overline{\pi}(H) \defeq \min\braces*{\rk(J):
    \,\, H\le J\le \FF}.
\end{equation}
A subgroup $H\le \FF$ is \textbf{compressed} if
$\overline{\pi}(H) = \rk(H)$.
Keeping in mind the notation $H=\pi_1(\Gamma)$,
Ivanov \cite[Theorem 1.1(b)]{ivanov2018intersection} showed that the
closely related invariant
\begin{equation}
    \label{eq_Ivanov_stable_compressed_rank}
    s\pibar_{\textup{SHNC}}(H)
    \defeq \inf\braces*{ \frac{\chi(\Gamma)\cdot \chi(\Delta)}
    {-\chi(\Gamma\times_{\Omega_B}\Delta)} \middle|
    \Delta \textup{ is a core graph with }\chi(\Gamma\times_{\Omega_B}\Delta)<0}
\end{equation}
is rational,
by showing that the infimum is attained as a minimum,
assuming $1<\rk(H)<\infty$.\footnote{
This invariant also plays a role
in Friedman's proof of the SHNC, where
it is shown to be an invariant of the
commensurability class
of $H=\pi_1(\Gamma)$ in $\FF=\pi_1(\Omega_B)$
\cite{dicks2012joel}.
}
Observe that for every
$d\in \N$
and $H\le \FF$
with $1<\rk(H)<\infty$,
\begin{equation}
    \label{eq_spibar_upper_bound}
    1\le s\pibar_{\textup{SHNC}}(H)
    \le \spivarbar{d}(H)
\le \overline{\pi}(H) - 1
\le \min(\rk(\FF),\rk(H))-1.
\end{equation}
Jaikin-Zapirain \cite[Corollary 3.13]{jaikin2024free} showed that
$s\pibar_{\textup{SHNC}}(H) = \rk(H)-1$
whenever $H$ is compressed
(this is called \enquote{strong inertness}
in \cite[Section 2.8]{jaikin2024free});
the non-constant quantities in the inequalities
\eqref{eq_spibar_upper_bound} then coincide.
For many (necessarily non-compressed) subgroups
$H\le \FF$, the inequality
$s\pibar_{\textup{SHNC}}(H)
\le \overline{\pi}(H) - 1$
is strict; see Example~\ref{example_shnpi_gap}.
It was conjectured in a previous
version \cite{shomroniProbabilisticHNC} of this paper that in fact,
\begin{equation}
    \label{eq_jz_stability_intro}
    \spivarbar{d}(H) = \overline{\pi}(H) - 1
\end{equation}
for all finitely generated subgroups $1\ne H\le\FF$ and $d\in\N$.
This conjecture was recently proved by
Jaikin-Zapirain, using \cite[Theorem 3.12 and Corollary 3.13]{jaikin2024free} on
$L^2$-independence;
he kindly allowed us to include the proof
in Appendix~\ref{appendix_jaikin}.

\begin{remark}
    One can ask for a connected $\Delta$ in
    Question~\ref{question_simplest_d_morphisms}.
    For example, for $H=\inner{xyx,yx^2}$, we have $\spivarbar{d}(H)=1$, and $\Delta$
    from Definition~\ref{def_spibard} may be taken connected of Euler
    characteristic $-d$, with $d$ distinct morphisms $\Gamma\to \Delta$, as
    in the following illustration.
\end{remark}
\begin{center}
\begin{tikzcd}[ampersand replacement=\&]
  \bullet \& \bullet \&\& \bullet \& \bullet \& \bullet \& \cdots \& \bullet \\
  \bullet \& \bullet \&\& \bullet \& \bullet \& \bullet \& \cdots \& \bullet 
  \arrow["y",                                   blue, from=1-2, to=1-1]
  \arrow["x"',                                  red,  from=1-1, to=2-1]
  \arrow[""{name=0, anchor=center, inner sep=0}, "x", red,  from=1-2, to=2-2]
  \arrow["y",                                   blue, from=1-5, to=1-4]
  \arrow[""{name=1, anchor=center, inner sep=0}, "x"{description}, red, from=1-4, to=2-4]
  \arrow["y",                                   blue, from=1-6, to=1-5]
  \arrow["x"{description},                      red,  from=1-5, to=2-5]
  \arrow["y",                                   blue, from=1-7, to=1-6]
  \arrow["x"{description},                      red,  from=1-6, to=2-6]
  \arrow["y",                                   blue, from=1-8, to=1-7]
  \arrow["x"{description},                      red,  from=1-8, to=2-8]
  \arrow["x"{description},                      red,  from=2-1, to=1-2]
  \arrow["y"{description},                      blue, from=2-2, to=2-1]
  \arrow["x"{description},                      red,  from=2-4, to=1-5]
  \arrow["y",                                   blue, from=2-5, to=2-4]
  \arrow["x"{description},                      red,  from=2-5, to=1-6]
  \arrow["y",                                   blue, from=2-6, to=2-5]
  \arrow["x"{description},                      red,  from=2-6, to=1-7]
  \arrow["y",                                   blue, from=2-7, to=2-6]
  \arrow["x"{description},                      red,  from=2-7, to=1-8]
  \arrow["y",                                   blue, from=2-8, to=2-7]
  \arrow["{d\textrm{ morphisms}}", shorten <=13pt, shorten >=13pt,
         Rightarrow, from=0, to=1]
\end{tikzcd} 
\end{center}

What about words?
Clearly $\pibar(\inner{w}) = 1$
for every word $w$.
In \cite[Conjecture 3.3]{wise2003nonpositive},
Wise conjectured that
for every non-power word $w\in \FF$,
if $\Delta$ is a core graph
and $\Gamma_w\times_{\Omega_B}\Delta$
contains the trivial $d$-covering of $\Gamma_w$,
then the first Betti number $\beta_1(\Delta) \ge d$, and dedicated
the paper
\cite[Conjecture 1.1]{wise2005coherence}
to this conjecture.
In \cite[Theorem 1.3(2)]{wise2005coherence}, Wise also proved (assuming the SHNC)
that $\beta_1(\Delta) \ge d/\two$,
similarly to Neumann's theorem in
\cite[p.~128]{neumann1957intersection}.
By removing parts of $\Delta$ that are
covered at most once by the $d$-covering of
$\Gamma_w$, one gets a slightly stronger conjecture,
which was solved in both
\cites[Theorem 1.2]{louder2014stacking}[Theorem 4.1]{helfer2016counting}
independently:
\begin{theorem}[Louder--Wilton, Helfer--Wise]
    \label{thm_Wilton_Louder_Helfer_Wise}
    Let $w\in \FF$ be a non-power word,
    $\Gamma_w$ its Stallings core graph,
    and $\Delta$ a core graph such that
    $\Gamma_w\times_{\Omega_B}\Delta$
    contains a $d$-covering $\tilde{\Gamma}_w$ of $\Gamma_w$,
    that covers every edge of $\Delta$ at least twice
    through the projection $p_{\Delta}\colon\tilde{\Gamma}_w\to\Delta$.
    Then $\chi(\Delta) \le -d.$
\end{theorem}

In Theorem~\ref{thm_spiK_gap}, we prove an analog
of this theorem for modules over free group algebras.
The assumption that $w$ is not a power is
necessary.\footnote{
    Say $w=u^k$ with $k\ge 2$, and take $\Delta\defeq\Gamma_u$.
    Then already for $d=1$, the pullback $\Gamma_w\times_{\Omega_B}\Gamma_u$ contains
    $\Gamma_w$, which covers every edge of $\Delta$ exactly $k\ge 2$ times;
    yet $\chi(\Delta)=0$.
}
By replacing the geometric condition of
covering every edge at least twice
by the stronger\footnote{
    This is stronger only for words, that is, for $\Gamma=\Gamma_w$.
} basis-independent condition
of \textbf{algebraicity},
Wilton \cite[Definition 10.6]{wilton2022rational}
defined the stable primitivity rank $s\pi$ of words.
We give here a generalization to
not-necessarily cyclic subgroups, which was
defined by Puder and the author.
Recall that $H\le J$ is a \textbf{free factor} of $J$ if $J=H*H'$ for some
$H'\le J$. Kapovich and Myasnikov introduced algebraic extensions in
\cite[Definition 11.1]{kapovich2002stallings}.

\begin{definition}[{\cite[Proposition 3.3]{miasnikov2007algebraic}}]
    \label{def_algebraic_extension}
    Let $H\le J$ be free groups. We say that $J$ is \textbf{algebraic over} $H$ if $H$ is not
    contained in any proper free factor of $J$.
\end{definition}

Following \cite[Definitions 4.2 and 4.6]{HP22}, a morphism $\eta\colon \Gamma\to \Delta$ of
core graphs is \textbf{free} if for every connected component $\Delta_0$ of
$\Delta$, with $\Gamma_1,\ldots,\Gamma_k$ the connected components of
$\eta^{-1}(\Delta_0)$, we have
$\pi_1(\Delta_0)=J_1*\cdots*J_k*J_0$ for some subgroup $J_0$ and some conjugates
$J_i$ (for $1\le i\le k$) of the subgroups $\eta_*\pi_1(\Gamma_i)$ (which are defined, without
choosing base points, only up to conjugacy).
It is \textbf{algebraic} if in every decomposition
$\eta=\eta_2\circ\eta_1$ with $\eta_2$ free, the morphism $\eta_2$ is an
isomorphism.\footnote{
    Equivalently \cite[Section 2.3]{puder2023stable}: $\eta$ is algebraic if for every connected
    component $\Delta_0$ of $\Delta$, the preimage $\eta^{-1}(\Delta_0)$ is
    non-empty and there is \textbf{no} non-trivial free
    splitting $\pi_1(\Delta_0)=J_1*J_2$ such that for every connected
    component $C$ of $\eta^{-1}(\Delta_0)$, the subgroup
    $(\eta|_{C})_*(\pi_1(C))$ is conjugate to a subgroup of $J_1$ or of
    $J_2$.
}
For connected $\Gamma$ and $\Delta$, this says that $\pi_1(\Delta)$ is algebraic over
$\eta_*\pi_1(\Gamma)$.
\begin{definition}
    \label{def_stable_primitivity_rank}
    Let $H\le \FF$ be a finitely
    generated subgroup with
    Stallings core graph $\Gamma$.
    The \textbf{stable primitivity rank} of
    $H$ is $s\pi(H)\defeq\inf_{d\in\N} s\pi_d(H)$, where
    \[
        s\pi_d(H) \defeq \min
        \braces*{ \frac{-\chi(\Delta)}{d} \middle|
        \begin{array}{ll}
            & \Delta \textup{ is a connected core graph,}\\
            & \textup{there is a }d\textup{-covering }\tilde{\Gamma}\textup{ of }
            \Gamma \textup{ in } \Gamma\times_{\Omega_B}\Delta, \\
            & \textup{and the projection } p_\Delta\colon
            \tilde{\Gamma} \to \Delta \textup{ is}\\
            & \textup{algebraic and not an isomorphism}.
        \end{array}
    }.
    \]
\end{definition}
By \cite[Remark 7.9]{HP22}, when $H=\inner{w}$ for a non-power $w\in\FF$, the condition on
$p_\Delta$ in Definition~\ref{def_stable_primitivity_rank} implies that every edge of $\Delta$ is
covered at least twice by $\tilde{\Gamma}$. Hence, by Theorem~\ref{thm_Wilton_Louder_Helfer_Wise},
$s\pi(w)\ge 1$.
By the equality \eqref{eq_jz_stability_intro},
$\pibar(H)-1 \le s\pi_d(H)$; in particular
if $H$ is \noncyc, then
$s\pi(H)\ge 1$.
As $s\pi(w^k)=0$ for every $k\ge 2$,
we get a \textbf{gap}
$\Img(s\pi) \cap [0,1] = \{0,1\}$.
The name \enquote{stable primitivity rank} stems from the $d=1$ case:
in \cite[Definition 1.8]{Puder_2014}, Puder defined the \textbf{primitivity rank} of a
subgroup $H\le \FF$ as
\begin{equation}
    \label{eq_primitivity_rank}
    \pi(H)\defeq \min\braces*{
    \rk(J)\,\middle|\, H\le J\le \FF, \textup{ and }H\textup{ is not a free
    factor of }J},
\end{equation}
where $\pi(H)=\infty$ if $H$ is a free factor of $\FF$;
for $1\neq w\in\FF$ we define $\pi(w)\defeq\pi(\inner{w})$.
Since a proper free factor of $J$ containing $H$ has smaller rank and does not have $H$ as a
free factor, a subgroup $J$ attaining the minimum in \eqref{eq_primitivity_rank} is algebraic over
$H$, so $s\pi_1(H)=\pi(H)-1$.

Wilton \cite{wilton2021stable}
(see also \cite[Conjecture 4.7]{puder2023stable})
conjectured that
$s\pi(w)=\pi(w)-1$ for every $1\ne w\in \FF$;
this was disproved by
Noam Ta Shma \cite[Section 6.3]{TaShma26}.
The equality \eqref{eq_jz_stability_intro}, proved by Jaikin-Zapirain, is an
analog of this conjecture for the compressed rank.
In Appendix~\ref{appendix_ta_shma}, joint with Noam Ta Shma, we prove the
lower bound
\[ s\pi_d(H)\;\ge\; 1+\frac{\pi(H)-2}{d} \]
whenever $\pi(H)\ge 2$ (Theorem~\ref{thm_gen_dependence_theorem}), extending
the \enquote{dependence theorem} of Louder and Wilton
\cite[Theorem 1.16]{louder2022negative} from words to arbitrary finitely
generated subgroups; in particular $s\pi(H)=\pi(H)-1$ whenever $\pi(H)=2$.
By Corollary~\ref{cor_stable_primitivity_rank}, $s\pi(H)=\pi(H)-1$ also
whenever $H$ is not \textbf{strictly compressed}, that is, whenever
$\rk(J)\le\rk(H)$ for some $H\lneqq J\le \FF$.

\subsubsection*{A unified proof for the \texorpdfstring{$s\pi$}{spi} gap}
We are ready to present a unified proof for the gap $s\pi(H)\ge 1$ in the two
cases $\rk(H)=1$ and $\rk(H)>1$,
assuming Theorem~\ref{thm_Gamma_invariant_polymatroid_theorem}.
In the case $H=\inner{w}$ this is essentially the proof of
\cite[Sections 2 and 3]{louder2014stacking}, stated in the language of polymatroids; in the
case $\rk(H)>1$ it is a new proof of this special case of the SHNC.

\begin{theorem}
    \label{thm_spi_gap_unified}
    Let $H\le \FF$ with $\pi(H)\ge 2$; then $s\pi(H)\ge 1$.

    More explicitly, and slightly more generally, assume that either
    $1<\rk(H)<\infty$ or $H=\inner{w}$ where $w$ is not a proper power.
    Let $\Gamma$ and $\Delta$ be core graphs, with $\Gamma$ connected and $\pi_1(\Gamma)=H$.
    Let $P\subseteq \Gamma\times_{\Omega_B}\Delta$ be a subgraph,
    such that the projection $p_{\Gamma}\colon P\to \Gamma$ is a covering map of degree $d$;
    if $\rk(H)=1$, assume in addition that the projection $p_{\Delta}\colon P\to\Delta$ covers every vertex 
    of $\Delta$ at least twice. Then $\chi(\Delta)\le -d$.
\end{theorem}

Note that covering the vertices of $\Delta$ twice is a weaker condition than
covering the edges twice.
Indeed, we show in Observation~\ref{obs_tight_on_edges} that more generally,
for a \Binv $\Gamma$-polymatroid $\hp$, if $\hp\restriction_{\src(E_b(\Gamma))}$
is tight for every $b\in B$, then $\hp$ is tight.
In the following proof we construct a $\Gamma$-polymatroid $\hp$ with $\hp(\{v\})=d$ and
$\chi(\hp)\ge\chi(\Delta)$.

\begin{proof}
    [Proof assuming Theorem~\ref{thm_Gamma_invariant_polymatroid_theorem}]
    Define $\hp\colon 2^{V(\Gamma)}\to\R$ by 
    $\hp(U)\defeq |p_{\Delta}(p_{\Gamma}^{-1}(U))|$.
    Since $\hp(U)=\abs*{\bigcup_{u\in U} 
    p_{\Delta}(p_{\Gamma}^{-1}(u))}$,
    by Example~\ref{example_polymatroids_as_unions}, $\hp$ is a polymatroid.
    Let $E_0\subseteq E_b(\Gamma)$ (for some $b\in B$).
    Since $p_{\Gamma}$ is a covering map,
    $\src(p_{\Gamma}^{-1}(E_0))
    =p_{\Gamma}^{-1}(\src(E_0))$, 
    and since $p_{\Delta}$ is a morphism of $B$-graphs, 
    $\src(p_{\Delta}(p_{\Gamma}^{-1}(E_0)))
    =p_{\Delta}(\src(p_{\Gamma}^{-1}(E_0)))
    =p_{\Delta}(p_{\Gamma}^{-1}(\src(E_0)))$,
    and similarly for $\tar$.
    This shows that $\hp$ is a \Binv $\Gamma$-polymatroid.
    Since $P\subseteq \Gamma\times_{\Omega_B}\Delta$,
    the projection $p_{\Delta}$ is injective on 
    $p_{\Gamma}^{-1}(\{v\})$ for every $v\in V(\Gamma)$, so 
    $\hp(\{v\})=|p_{\Gamma}^{-1}(\{v\})|=d$.
    We claim $\chi(\hp) 
    = \chi(\Img(p_{\Delta})) 
    \ge \chi(\Delta)$.
    Indeed, $\hp(V(\Gamma))
    =|p_{\Delta}(V(P))|$ and, for $b\in B$, we have
    $\hp(\src(E_b(\Gamma)))=|\src(p_{\Delta}(E_b(P)))|=|p_{\Delta}(E_b(P))|$.
    Moreover, by Example~\ref{example_tight_polymatroids}, $\hp$ is 
    tight
    if and 
    only if 
    every element of 
    $\bigcup_{v\in V(\Gamma)} p_{\Delta}p_{\Gamma}^{-1}(\{v\})=V(\Img(p_{\Delta}))$
    is contained in at least two sets $p_{\Delta}p_{\Gamma}^{-1}(\{v\})$,
    or equivalently, is covered by at least two vertices of $V(P)$.
    Now by Theorem~\ref{thm_Gamma_invariant_polymatroid_theorem},
    $\chi(\Delta)\le \chi(\hp) \le -\hp(\{v\}) = -d$ for some $v\in V(\Gamma)$.
\end{proof}

\subsection{The stable \texorpdfstring{$K$}{K}-primitivity rank and a
\texorpdfstring{$q$}{q}-analog of the HNC}
\label{subsec_spiK}

We now replace $B$-graphs by modules over the free group algebra $K[\FF]$.
Under this analogy, a subgroup $H\le \FF$ corresponds to the sub-group-algebra
$K[H]\le K[\FF]$, a $d$-covering corresponds to a free module of rank $d$, and
the $d$-stable compressed rank $\spivarbar{d}$ from
Definition~\ref{def_spibard} corresponds to the following invariant.

Following \cite{ernst2024word}, in this paper we work with right modules over group algebras,
except for Appendix~\ref{appendix_jaikin}.
Therefore, for a submodule $M\le N$, the quotient module is denoted by $M\backslash N$.

\begin{definition}
    \label{def_spibarqd}
    Let $K$ be a field,
    let $H\le \FF$ be a finitely
    generated subgroup, and let $d\in \N$.
    The $d$\textbf{-stable}
    $K$\textbf{-compressed rank}
    of $H$ is
    \begin{equation}
    s\pibar_{K,d}(H)\defeq \inf\braces*{
    \frac{\rk(N)}{d}-1\middle|\,\,
    N\le K[\FF]^d,
    \,\,\dim_{K} (
    (N\cap K[H]^d)\backslash K[H]^d)
    = d}
\end{equation}
where $N$ runs over finitely generated
submodules of
the free right $K[\FF]$-module
$K[\FF]^d$.
When $K=\F_q$ is a finite field we abbreviate
$s\pibar_{q,d}\defeq s\pibar_{\F_q,d}$.
\end{definition}

By the following analog of the Nielsen--Schreier theorem, $s\pibar_{K,d}$ is well
defined.

\begin{theorem}[{\cite[Corollary 3]{cohn1964free}, \cite[Theorem 1]{lewin1969free}}]
    \label{thm_cohn_lewin}
    Let $K$ be a field. Every submodule of a free $K[\FF]$-module is free, and has a
    well-defined rank.
\end{theorem}

\begin{observation}
    \label{obs_spibar_K1}
    For every field $K$ and every finitely generated $H\le\FF$, we have
    $s\pibar_{K,1}(H)\le\pibar(H)-1$.
    Indeed, for $J\ge H$ with $\rk(J)=\pibar(H)$, the right ideal $N\defeq I_JK[\FF]$
    generated by the augmentation ideal $I_J$ of $K[J]$ has $\rk(N)=\rk(J)$, and
    $N\cap K[H]$ is the augmentation ideal of $K[H]$.
\end{observation}

A submodule $M\le N$ of a free module is a \textbf{direct
summand} if $N=M\oplus M'$ for some submodule $M'$, and then both $M$ and
$M'$ are free.

Just as $\spivarbar{d}$ determines the number of invariant $d$-subsets of a
word-random subgroup of $S_n$ (Theorem~\ref{thm_Sn_spibard}), its $q$-analog
determines the number of invariant $d$-subspaces of a word-random subgroup of
$\GLnFq$.
We prove it in Section~\ref{section_counting_fixed_sub_spaces}.

\begin{definition}
    \label{def_EHtoFGLnActsGrd}
    Let $H\le \FF$ be a finitely generated subgroup, let $q$ be a prime power,
    and let $d\le n$ be natural numbers.
    We denote by $\EHtoFVarActsVar{\GLnFq}{\Grasm_d(\F_q^n)}$ the expected
    number of $d$-dimensional subspaces of $\F_q^n$ that are invariant under
    $\alpha(H)$, where $\alpha$ is uniformly random in $\Hom(\FF,\GLnFq)$.
\end{definition}

\begin{theorem}
    \label{thm_spibarqd}
    In the setting of
    Definition~\ref{def_spibarqd},
    for fixed $d$ and $q$,
    \[ \EHtoFVarActsVar{\GLnFq}{\Grasm_d(\F_q^n)}
    = \Theta\prn*{q^{-nd\cdot s\overline{\pi}_{q,d}(H)}}\quad \textup{as } n\to\infty. \]
\end{theorem}

Theorem~\ref{thm_spibarqd} shows that, for every $d\in\N$ and every prime
power $q$, the invariant $s\pibar_{q,d}$ is \textbf{profinite} in the sense
of \cite[Definition 1.3]{puder2023stable}: it is an $\Aut(\FF)$-invariant
function of subgroups which is moreover $\Aut(\hat{\FF})$-invariant, where
$\hat{\FF}$ is the profinite completion of $\FF$.\footnote{
    Indeed, if $\Phi\in \Aut(\hat{\FF})$ maps $\hat{H_1}$ to $\hat{H_2}$,
    then precomposing with $\Phi$ permutes
    $\Hom(\FF,\GLnFq)=\Hom_{\textup{cont}}(\hat{\FF},\GLnFq)$, so the two
    expectations agree for every $n$, and by Theorem~\ref{thm_spibarqd},
    $s\pibar_{q,d}(H_1)=s\pibar_{q,d}(H_2)$.
}

We prove in
Proposition~\ref{thm_easier_spiK_gap}
that, for every field $K$ and every $d\in\N$, we have $s\pibar_{K,d}(H)\ge 1$
for $H\le \FF$ with $1<\rk(H)<\infty$.
By \cite[Theorem 4: The Schreier formula]{lewin1969free},
for every
$N$ as in
Definition~\ref{def_spibarqd}
we have
\[\rk(N\cap K[H]^d) =
d\cdot \rk(H),\]
where $N\cap K[H]^d$
is considered as a right
$K[H]$-module.
Hence, the bound
$s\pibar_{K,d}(H)\ge 1$
can be reformulated as
\begin{equation}
    \label{eq_q_HNC}
\rk(N\cap K[H]^d)/d
- 1
\le \prn*{
    \rk(N)/d
     - 1}
 \cdot (\rk(H) - 1),
\end{equation}
which resembles the HNC
\eqref{eq:HNC}.
Unlike the SHNC, this $K$-analog \eqref{eq_q_HNC} is false without the
requirement that $\dim_K((N\cap K[H]^d)\backslash K[H]^d)<\infty$, even when $N$ is
not contained in any proper direct summand of $K[\FF]^d$.\footnote{
    For example, let $d=2$, let $H\defeq\inner{x,y^2,yxy^{-1}}$, which has rank $3$, and
    let $N\defeq vK[\FF]$ for $v\defeq(x-1,y^2-1)\in K[H]^2$. Then $N\cap K[H]^2=vK[H]$,
    so \eqref{eq_q_HNC} would read $-\frac{1}{2}\le-1$. We thank Doron Puder, Matan Seidel and Danielle Ernst-West for pointing
    out that \eqref{eq_q_HNC} fails without this requirement.
}

We turn to the $K$-analog of the stable primitivity rank $s\pi$.
Ernst-West, Puder and Seidel
\cite[Definition 1.5]{ernst2024word}
defined, for a finite field of order $q$, a $q$-analog $\pi_q$
of the primitivity
rank $\pi$ (defined in \eqref{eq_primitivity_rank}), which naturally extends
to an arbitrary field $K$,
and, together with the author, a $K$-analog
$s\pi_K$
\cite[Appendix]{puder2023stable}
of
the stable primitivity rank
$s\pi$.
We slightly extend the domain of $s\pi_K$
given in \cite[Appendix]{puder2023stable}
to every finitely generated subgroup
$H\le \FF$.
The following modules are the analog of $d$-coverings of the Stallings core
graph of $H$.

\begin{definition}
    \label{def_H_module}
    [{cf.\ \cite[Definition A.1]{puder2023stable}}]
    An \textbf{\Hmod} (over $K$; not to be confused with a $K[H]$-module)
    is a right $K[\FF]$-module of the form
    \[ W = W_0 \otimes_{K[H]} K[\FF], \]
    where $W_0$ is a non-zero, finite dimensional right $K[H]$-module.
    The \textbf{degree} of $W$ is
    \[ \deg(W)\defeq \dim_K(W_0)\in \Z_{\ge 1}. \]
    A sub-\Hmod{} of $W$ is $W' = W_0'\otimes_{K[H]}K[\FF]$, where
    $W_0'\le W_0$ is a $K[H]$-submodule. A quotient of an \Hmod{}
    is $W'\backslash W\cong (W_0'\backslash W_0)\otimes_{K[H]}K[\FF]$, where $W'$ is a
    sub-\Hmod{} as above.
\end{definition}
In Definition~\ref{def_H_module}, $W_0'$ may also be $0$ or $W_0$, so $W$ and
$0$ are quotients of $W$.
A $K[\FF]$-module $Q$ is \textbf{finitely presented} if $Q\cong P_2\backslash P_1$ for
free $K[\FF]$-modules $P_1,P_2$ of finite rank; any two such presentations
have the same value of $\rk(P_1)-\rk(P_2)$, called the \textbf{Euler
characteristic} $\chi(Q)\defeq\rk(P_1)-\rk(P_2)$ of $Q$
\cite[the paragraph after Definition A.2]{puder2023stable}.
Definition~\ref{def_spibarqd} can be rewritten in this language.

\begin{proposition}
    \label{prop_spibarK_module_form}
    Let $H\le \FF$ be a finitely generated subgroup and let $d\in\N$. Then
    \[ s\pibar_{K,d}(H)= \inf\braces*{
        \frac{-\chi(Q)}{d}
        \middle| \begin{array}{l}
        W\textup{ is an \Hmod{} of degree }d,\\
        \phi\colon W\twoheadrightarrow Q\textup{ is a surjective }K[\FF]\textup{-linear map,}\\
        Q\textup{ is finitely presented, and }\phi\textup{ does not factor through}\\
        \textup{a proper quotient }W\overset{\not\cong}{\twoheadrightarrow}\overline{W}\textup{ of \Hmod{}s}
    \end{array}}. \]
\end{proposition}

Since $Q$ is finitely presented, $\chi(Q)$ is well-defined.
Since $K[\FF]$ is a free left $K[H]$-module, on a right transversal of $H$ in
$\FF$ containing $1$, the map $f\mapsto f\otimes 1$ embeds $W_0$ in
$W=W_0\otimes_{K[H]}K[\FF]$; we identify $W_0$ with its image, so that
$f\otimes g=fg$ for $f\in W_0$ and $g\in\FF$.

\begin{observation}
    \label{obs_phi_injective_on_W0}
    Let $W=W_0\otimes_{K[H]}K[\FF]$ be an \Hmod{} and let $\phi\colon W\to Q$
    be a homomorphism.
    Then $\phi$ factors through a proper quotient of \Hmod{}s if and only if
    $\phi$ is not injective on $W_0$.
    Indeed, the proper quotients of $W$ are
    $W\twoheadrightarrow (W_0'\otimes_{K[H]}K[\FF])\backslash W$ for $K[H]$-submodules
    $0\ne W_0'\le W_0$, where $W_0'\otimes_{K[H]}K[\FF]\le W$ is non-zero since
    $K[\FF]$ is a free left $K[H]$-module, and $\phi$ factors through such a quotient exactly when
    $\phi(W_0')=0$; the kernel of $\phi$ on $W_0$ is a $K[H]$-submodule.
\end{observation}

\begin{proof}
    [Proof of Proposition~\ref{prop_spibarK_module_form}]
    Denote by $e_1,\ldots,e_d$ the standard basis of $K[\FF]^d$.
    Let $N\le K[\FF]^d$ be a finitely generated submodule with
    $\dim_K((N\cap K[H]^d)\backslash K[H]^d)=d$, as in Definition~\ref{def_spibarqd}.
    Then $W_0\defeq (N\cap K[H]^d)\backslash K[H]^d$ is a right $K[H]$-module of dimension $d$,
    the inclusion $K[H]^d\subseteq K[\FF]^d$ induces an injective
    $K[H]$-homomorphism $W_0\to Q\defeq N\backslash K[\FF]^d$, and it extends to a
    $K[\FF]$-linear map $\phi\colon W_0\otimes_{K[H]}K[\FF]\to Q$, which is surjective
    since its image contains the images of $e_1,\ldots,e_d$.
    The module $Q$ is finitely presented, as $N$ is finitely generated, and
    $\chi(Q)=d-\rk(N)$, as $N$ is free
    \cites[Corollary 3]{cohn1964free}[Theorem 1]{lewin1969free}.
    By Observation~\ref{obs_phi_injective_on_W0}, $\phi$ does not factor
    through a proper quotient of \Hmod{}s.

    Conversely, let $W$ be an \Hmod{} of degree $d$, let
    $\phi\colon W\twoheadrightarrow Q$ be a surjective $K[\FF]$-linear map that
    does not factor through a proper quotient of \Hmod{}s, with $Q$ finitely
    presented, as in Proposition~\ref{prop_spibarK_module_form},
    and let $f_1,\ldots,f_d$ be a $K$-basis of $W_0$.
    Let $N$ be the kernel of the $K[\FF]$-linear map $K[\FF]^d\to Q$ defined by
    $e_i\mapsto\phi(f_i)$, which is surjective since $\phi(W_0)$ generates
    $Q$.
    As $K[\FF]^d$ is finitely generated and $Q$ is finitely presented, $N$ is
    finitely generated, and $\chi(Q)=d-\rk(N)$.
    By Observation~\ref{obs_phi_injective_on_W0}, $\phi$ is injective on
    $W_0$, so $N\cap K[H]^d$ is the kernel of the surjection $K[H]^d\to W_0$
    defined by $e_i\mapsto f_i$, and $\dim_K((N\cap K[H]^d)\backslash K[H]^d)=d$.
\end{proof}

Following \cite[Appendix]{puder2023stable}, call a $K[\FF]$-module
\textbf{algebraic} if it has no non-zero free direct summand, or equivalently,
if it admits no non-zero free quotient.
By \cite[the paragraph after Definition A.2]{puder2023stable}, for $H=\inner{w}$, if $W$ is
an \Hmod{} and $\phi\colon W\twoheadrightarrow Q$ is a surjection whose kernel is algebraic, then $Q$
is finitely presented.
The following definition is a straightforward generalization, from
words to subgroups, of \cite[Definition A.2]{puder2023stable}:

\begin{definition}
\label{def_stable_K_primitivity_rank}
Let $H\le \FF$ be a finitely generated subgroup. Its \textbf{stable
$K$-primitivity rank} is $s\pi_K(H)\defeq\inf_{d\in\N}s\pi_{K,d}(H)$, where
\[ s\pi_{K,d}(H)\defeq \inf\braces*{
    \frac{-\chi(Q)}{d}
    \middle|\,\begin{array}{@{}l@{}}
    W\textup{ is an \Hmod{} of degree }d,\ \phi\colon W\twoheadrightarrow Q\textup{ is }K[\FF]\textup{-linear,}\\
    \ker(\phi)\textup{ is algebraic, }\phi\textup{ does not factor through a proper}\\
    \textup{quotient }W\overset{\not\cong}{\twoheadrightarrow}\overline{W}\textup{ of \Hmod{}s, and no non-zero}\\
    \textup{quotient }W\twoheadrightarrow\overline{W}\textup{ of \Hmod{}s factors through }\phi
\end{array}}. \]
When $H=1$ we define $s\pi_K(H)=-1$.
\end{definition}

Definition~\ref{def_stable_K_primitivity_rank} constrains $(W,\phi)$ further,
so $s\pibar_{K,d}(H)\le s\pi_{K,d}(H)$.
As for $s\pi$, if $H=\inner{w}$ is cyclic and $w=u^k$ is a power, for some $u\in\FF$ and $k\ge2$, then
$s\pi_K(H)=0$ \cite[the paragraph after Remark A.3]{puder2023stable}.
Otherwise, if $H\ne1$, then $\pi(H)\ge 2$ (that is, $H=\inner{w}$ with $w$
not a proper power, or $\rk(H)>1$).
\begin{theorem}
    \label{thm_spiK_gap}
    Let $1\ne H\le \FF$ be finitely generated.
    If $\pi(H)\ge 2$, then $s\pi_K(H)\ge 1$.
\end{theorem}

To prove Theorem~\ref{thm_spiK_gap} we use the following construction.

\begin{definition}
    \label{def_L_of_a_vertex}
    Let $1\ne H\le\FF$ be a finitely generated subgroup, let
    $W=W_0\otimes_{K[H]}K[\FF]$ be an \Hmod{} of degree $d$, and let
    $\phi\colon W\to Q$ be a $K[\FF]$-linear map.
    Recall that $\Gamma_B(H)$ is the core of the Schreier graph
    $H\backslash\textup{Cay}(\FF,B)$, whose vertices are the cosets $Hg$ and
    whose $b$-labeled edges are $Hg\to Hgb$ (for $g\in\FF$ and $b\in B$).
    For a coset $Hg$ define
    \begin{equation}
        \label{eq_L_of_a_vertex}
        \mathcal{L}(Hg)\defeq\phi(W_0)\,g\subseteq Q,
        \qquad
        \mathcal{L}(U)\defeq\sum_{u\in U}\mathcal{L}(u)
        \quad(\forall U\subseteq H\backslash\FF\textup{ finite}),
    \end{equation}
    where $\mathcal{L}(Hg)$ depends only on $Hg$, because
    $\phi(W_0)h=\phi(W_0h)=\phi(W_0)$ for every $h\in H$.
    For a finite set $U$ of cosets, $\mathcal{L}(U)$
    is a finite dimensional subspace of $Q$, and
    $\hp(U)\defeq\dim_K\mathcal{L}(U)$ is a polymatroid by
    Example~\ref{example_polymatroids_as_unions}.
    We stress that $\mathcal{L}$ and $\hp$ are defined on every finite subset
    of $H\backslash\FF$.
\end{definition}

Ernst-West and Puder proved the following theorem, stated here in our notation:

\begin{theorem}[{Ernst-West and Puder \cite{ernstwestPuderCoresInPrep}}]
    \label{thm_no_lonely_edges}
    Let $w\in\FF$ be a non-power, and let $W$ and
    $\phi\colon W\twoheadrightarrow Q$ satisfy the conditions of
    Definition~\ref{def_stable_K_primitivity_rank} for $H=\inner{w}$.
    Then, for every $b\in B$ and every $e\in E_b(\Gamma_w)$, we have
    $\mathcal{L}(\src(e))\subseteq
    \sum_{e'\in E_b(\Gamma_w)\setminus\{e\}}\mathcal{L}(\src(e'))$.
\end{theorem}

Theorem~\ref{thm_no_lonely_edges}
says exactly that the restriction of $\hp$ to $\src(E_b(\Gamma_w))$ is tight
for every $b\in B$ (Example~\ref{example_tight_polymatroids}).
This implies that $\hp\restriction_{\Gamma_w}$
is tight (by Observation~\ref{obs_tight_on_edges}, and since $\hp\restriction_{\Gamma_w}$ is \Binv by Lemma~\ref{lemma_module_polymatroid}).
In the language of Definition~\ref{def_vertex_edge_tight}, Theorem~\ref{thm_no_lonely_edges} says
that the $\Gamma_w$-polymatroid over $\Omega_B$ determined by $\hp\restriction_{\Gamma_w}$
(Remark~\ref{remark_B_invariant_determines}) is edge-tight.

\begin{lemma}
    \label{lemma_module_polymatroid}
    Let $1\ne H\le\FF$ be finitely generated, denote $\Gamma\defeq\Gamma_B(H)$,
    let $W=W_0\otimes_{K[H]}K[\FF]$ be an \Hmod{} of degree $d$, and let
    $\phi\colon W\twoheadrightarrow Q$ be a surjective $K[\FF]$-linear map that
    does not factor through a proper quotient of \Hmod{}s, with $Q$ finitely
    presented.
    Let $\hp(U)=\dim_K\mathcal{L}(U)$ be the polymatroid from
    Definition~\ref{def_L_of_a_vertex}.
    Then the restriction $\hp\restriction_{\Gamma}$ of $\hp$ to $V(\Gamma)$ is a
    \Binv $\Gamma$-polymatroid with $\hp(\{v\})=d$ for every $v\in V(\Gamma)$,
    and $\chi(Q)\le\chi(\hp\restriction_{\Gamma})$.
    Moreover, if $H=\inner{w}$ for a non-power $w\in\FF$, and $\ker(\phi)$ is
    algebraic, and no non-zero quotient $W\twoheadrightarrow\overline{W}$ of
    \Hmod{}s factors through $\phi$, then $\hp\restriction_{\Gamma}$ is tight.
\end{lemma}

In the following proof we use properties of $\Gamma$-polymatroids that we
show in Section~\ref{section_proof_of_Gamma_poly_thm}.

\begin{proof}
    By Observation~\ref{obs_phi_injective_on_W0}, $\phi$ is injective on
    $W_0$, so $\dim_K\phi(W_0)=d$, and $\hp(\{Hg\})=d$ for every coset $Hg$,
    because right multiplication by $g$ is a $K$-linear automorphism of $Q$.
    Denoting $Ub\defeq\{ub : u\in U\}$ for $b\in B$, we have
    $\mathcal{L}(Ub)=\mathcal{L}(U)b$ and hence $\hp(Ub)=\hp(U)$.
    For every finite subgraph $\Delta$ of the Schreier graph,
    $\tar(E_b(\Delta))=\src(E_b(\Delta))b$, so the restriction
    $\hp\restriction_{\Delta}$ of $\hp$ to $V(\Delta)$ is \Binv, with Euler
    characteristic
    \[ \chi(\hp\restriction_{\Delta})\defeq\hp(V(\Delta))
    -\sum_{b\in B}\hp(\src(E_b(\Delta))). \]
    In particular $\hp\restriction_{\Gamma}$ is a \Binv $\Gamma$-polymatroid.

    Let $f_1,\ldots,f_d$ be a $K$-basis of $W_0$, let $e_1,\ldots,e_d$ be the
    standard basis of $K[\FF]^d$, and let $N$ be the kernel of the
    $K[\FF]$-linear map $K[\FF]^d\to Q$ that sends $e_i$ to $\phi(f_i)$.
    As in the proof of Proposition~\ref{prop_spibarK_module_form}, $N$ is
    finitely generated and $\chi(Q)=d-\rk(N)$; identifying $Q$ with
    $N\backslash K[\FF]^d$, we have $\mathcal{L}(Hg)=\spn_K\braces{e_ig+N}_{i=1}^d$ for
    every $g\in\FF$.
    As in \cite[Section 3]{ernst2024word}, identify the monomials $e_ig$
    (for $i\le d$ and $g\in\FF$) with the vertices of $d$ copies of
    $\textup{Cay}(\FF,B)$, and recall that the \textbf{support} of an element
    of $K[\FF]^d$ is the set of monomials appearing in it.
    Fix a finite generating set of $N$, and let
    $\T\subseteq\textup{Cay}(\FF,B)$ be a finite subtree containing $1$, the
    prefixes of a generating set of $H$, and every $g\in\FF$ with $e_ig$ in
    the support of one of these generators.
    Let $\T^d$ be the union of the $d$ copies of $\T$, a forest with $d$
    connected components.
    The image $\Gamma'$ of $\T$ under $g\mapsto Hg$ is a finite connected
    subgraph of the Schreier graph of $H\backslash\FF$, and it contains
    $\Gamma$ (as it contains the closed walks at $H$ reading the generators
    of $H$), so its core is $\Gamma$ and $\pi_1(\Gamma')=\pi_1(\Gamma)$.

    We claim that $\chi(Q)=\chi(\hp\restriction_{\Gamma'})$.
    For a set $S$ of monomials let $N|_S$ be the $K$-space of elements of $N$
    supported on $S$.
    Since $N$ is generated by elements supported on $\T^d$, by
    the proof of \cite[Lemma 3.4]{ernst2024word} and by \cite[Theorem 3.8]{ernst2024word},\footnote{
        Both hold over an arbitrary field
        \cite[the opening of Section 3]{ernst2024word}.
    }
    \[ \rk(N)=\dim_K N|_{V(\T^d)}
    -\sum_{b\in B}\dim_K N|_{\src(E_b(\T^d))}. \]
    For every set $S$ of monomials, the kernel of the $K$-linear map
    $\spn_K(S)\to Q$ that sends $s$ to $s+N$ is $N|_S$, so
    $\dim_K N|_S=|S|-\dim_K\spn_K\braces{s+N}_{s\in S}$.
    Every $b$-labeled edge of $\Gamma'$ is the image of a $b$-labeled edge of
    $\T$, and $\mathcal{L}$ depends only on the coset, so the images in $Q$ of
    the monomials in $V(\T^d)$ and in $\src(E_b(\T^d))$ span
    $\mathcal{L}(V(\Gamma'))$ and $\mathcal{L}(\src(E_b(\Gamma')))$,
    respectively.
    Therefore $\dim_K N|_{V(\T^d)}=\abs{V(\T^d)}-\hp(V(\Gamma'))$ and
    $\dim_K N|_{\src(E_b(\T^d))}=\abs{\src(E_b(\T^d))}-\hp(\src(E_b(\Gamma')))$, so
    \[ \rk(N)=\abs*{V(\T^d)}-\sum_{b\in B}\abs*{\src(E_b(\T^d))}
    -\chi(\hp\restriction_{\Gamma'})
    =d-\chi(\hp\restriction_{\Gamma'}), \]
    where the last equality holds because $\T^d$ is a forest with $d$
    connected components (so it has $\abs{V(\T^d)}-d$ edges) and $\src$ is injective on
    $E_b(\T^d)$. This proves the claim, as $\chi(Q)=d-\rk(N)$.

    Finally, $\Gamma'$ is connected and contains the core graph $\Gamma$, so
    it is not a tree, and the inclusion $\Gamma\subseteq\Gamma'$ is a morphism
    of immersions into $\Omega_B$.
    In Corollary~\ref{cor_reduction_to_subgraphs} we show that the Euler
    characteristic of $\Gamma$-polymatroids is monotone with respect to
    inclusion, so
    ${\chi(\hp\restriction_{\Gamma})\ge\chi(\hp\restriction_{\Gamma'})=\chi(Q)}$.

    For the last part, let $H=\inner{w}$ for a non-power $w\in\FF$, so that
    $\Gamma=\Gamma_w$, and assume that $\ker(\phi)$ is algebraic and that no
    non-zero quotient of \Hmod{}s factors through $\phi$.
    Then $\hp\restriction_{\Gamma}$ is tight by
    Theorem~\ref{thm_no_lonely_edges} and Observation~\ref{obs_tight_on_edges}.
\end{proof}

\begin{proposition}
    \label{thm_easier_spiK_gap}
    Let $K$ be a field, let $d\in\N$, and let $H\le\FF$ with
    $1<\rk(H)<\infty$.
    Then $s\pibar_{K,d}(H)\ge1$; that is, the inequality \eqref{eq_q_HNC}
    holds for every finitely generated $N\le K[\FF]^d$ with
    $\dim_K((N\cap K[H]^d)\backslash K[H]^d)=d$.
\end{proposition}

\begin{proof}
    [Proof assuming Theorem~\ref{thm_Gamma_invariant_polymatroid_theorem}]
    Denote $\Gamma\defeq\Gamma_B(H)$.
    Let $W$ be an \Hmod{} of degree $d$, and let
    $\phi\colon W\twoheadrightarrow Q$ be a surjective $K[\FF]$-linear map that
    does not factor through a proper quotient of \Hmod{}s, with $Q$ finitely
    presented, as in Proposition~\ref{prop_spibarK_module_form}.
    Let $\hp\restriction_{\Gamma}$ be the \Binv $\Gamma$-polymatroid from
    Lemma~\ref{lemma_module_polymatroid}.
    By Theorem~\ref{thm_Gamma_invariant_polymatroid_theorem}, there is a
    vertex $v$ with
    $\chi(Q)\le\chi(\hp\restriction_{\Gamma})\le-\hp(\{v\})=-d$, so
    $-\chi(Q)/d\ge1$.
\end{proof}

\begin{proof}
    [Proof of Theorem~\ref{thm_spiK_gap} assuming
    Theorem~\ref{thm_Gamma_invariant_polymatroid_theorem}]
    Denote $\Gamma\defeq\Gamma_B(H)$.
    Let $d\in\N$, and let $W$ and $\phi\colon W\twoheadrightarrow Q$ satisfy
    the conditions of Definition~\ref{def_stable_K_primitivity_rank}.
    As $\ker(\phi)$ is algebraic, $Q$ is finitely presented; let
    $\hp\restriction_{\Gamma}$ be the \Binv $\Gamma$-polymatroid from
    Lemma~\ref{lemma_module_polymatroid}.
    Since $H\ne1$ and $\pi(H)\ge2$, either $\rk(H)>1$, or $H=\inner{w}$ for a
    non-power $w\in\FF$, and then $\hp\restriction_{\Gamma}$ is tight by
    Lemma~\ref{lemma_module_polymatroid}.
    In both cases Theorem~\ref{thm_Gamma_invariant_polymatroid_theorem} gives a
    vertex $v$ with
    $\chi(Q)\le\chi(\hp\restriction_{\Gamma})\le-\hp(\{v\})=-d$, that is,
    $-\chi(Q)/d\ge1$.
    Taking the infimum over $W$, $\phi$ and $d$ gives $s\pi_K(H)\ge1$.
\end{proof}

Proposition~\ref{thm_easier_spiK_gap} and Theorem~\ref{thm_spiK_gap} can also be proved without
Corollary~\ref{cor_reduction_to_subgraphs}. One first conjugates $H$ so that the coset $H$ is a
vertex of $\Gamma$, and reduces to the case where $\ker(\phi)$ is algebraic, by replacing $Q$ with
$A\backslash W$, where $\ker(\phi)=A\oplus P$ and $P$ is free of maximal rank. Then, by \cite[Corollary 3.10]{ernst2024word} (whose proof holds over any field), $\T$
can be chosen in the proof of Lemma~\ref{lemma_module_polymatroid} so that $\Gamma'=\Gamma$.

Let $\hp$ be the polymatroid from Definition~\ref{def_L_of_a_vertex}, and denote $\Gamma\defeq\Gamma_B(H)$.
The subspaces $\mathcal{L}(V(\Gamma))$ and $\mathcal{L}(\src(E_b(\Gamma)))$ for $b\in B$ form a sheaf
on $\Omega_B$, in the sense of \cite[Definition 1.1]{friedman2015sheaves}: its value at $v_B$ is
$\mathcal{L}(V(\Gamma))$, its value at the loop $b$ is $\mathcal{L}(\src(E_b(\Gamma)))$, and its two
restriction maps $\mathcal{L}(\src(E_b(\Gamma)))\to\mathcal{L}(V(\Gamma))$ are the inclusion and the
right multiplication by $b$. Its Euler characteristic, defined in
\cite[Section 1.2.1]{friedman2015sheaves}, is
${\dim_K\mathcal{L}(V(\Gamma))-\sum_{b\in B}\dim_K\mathcal{L}(\src(E_b(\Gamma)))}=\chi(\hp\restriction_{\Gamma})$.
The proof of Theorem~\ref{thm_Gamma_invariant_polymatroid_theorem}, however, also uses the values
of a \Binv $\Gamma$-polymatroid on other subsets of $V(\Gamma)$.

\subsection{Reiter's conjecture}
\label{subsec_Reiter}

This section is about probability in finite groups, and was the original
motivation for this paper.
Choose two permutations $\sigma,\tau\in S_n$
independently and uniformly at random.
Dixon \cite{Dixon1969} proved
$\PR\prn*{\inner*{\sigma,\tau}\supseteq A_n}\to_{n\to\infty} 1$,
confirming a conjecture of Netto.
Babai \cite{Babai1989} proved the following conjecture of Dixon:
\begin{equation}
    \label{eq_Dixon_conjecture}
    \PR\prn*{\inner*{\sigma,\tau}\supseteq A_n}
    = 1 - n^{-1} + O(n^{-2})\qquad\textup{as }n\to\infty.
\end{equation}
The main obstruction for generating $A_n$
is the event that both permutations
have a common fixed point,
providing the $n^{-1}$ term.\footnote{
Dixon \cite{dixon2005asymptotics} later sharpened \eqref{eq_Dixon_conjecture} into the full
expansion $\PR\prn*{\inner*{\sigma,\tau}\supseteq A_n}
=1-n^{-1}-n^{-2}-4n^{-3}-23n^{-4}+O(n^{-5})$ as $n\to\infty$.
}

Alex Lubotzky conjectured the following generalization of Dixon's result (personal communication):
even if $\sigma$ and $\tau$ are replaced by
$w_1(\sigma, \tau)$ and $w_2(\sigma, \tau)$ for non-commuting words
$w_1, w_2$ in two letters, still
\[ \PR\prn*{\inner*{w_1(\sigma, \tau),w_2(\sigma, \tau)}\supseteq A_n}
\to_{n\to\infty} 1, \]
and more generally,
that if $H\le \FF$ has $1 < \rk(H) < \infty$ and
$\alphaSimUHomFSn$ is a random homomorphism, then
$\PR\prn*{\alpha(H)\supseteq A_n} \to_{n\to\infty} 1$.
For example, if $H = \inner{x,yxy^{-1}}$, then
${\alpha(H) = \inner{\alpha(x), \alpha(yxy^{-1})}}$
is generated by two random conjugate permutations.
Puder (Appendix~\ref{appendix_random_words_gen_A_n}) observed that this generalization
follows from \cite[Theorem 3.14]{chen2024new}.

Asael Reiter, while working on this conjecture of Lubotzky, came up with the following idea:
Since the main obstruction for generating $A_n$ is
having a small invariant set,
it is useful to bound, for every $d\in\N$, the expected number of
common invariant subsets of size $d$ of
$\alpha(H)$ (and thus to bound the
probability of having such a small invariant set);
we denote this expectation by $\EHtoFSnActsND$.
More generally:
\begin{definition}
    \label{def_EHtoFSnActsND}
    Let $H\le \FF$ be a finitely generated subgroup, and let $G$ be a
    finite group acting on a set $X$.
    We denote the expected number of common fixed points of $\alpha(H)$ in
    $X$ by
    \[
        \EHtoFVarActsVar{G}{X} \defeq \E_{\alpha\sim\textup{Unif}(\Hom(\FF, G))}
        \brackets*{\abs*{X^{\alpha(H)}}},
    \]
    where, here and below,
    whenever a group $G$ acts on a set $X$ we denote the set of $G$-fixed points by $X^G$.
    In particular, for $d\in \N$, the expected number of common invariant sets of size $d$ of
    $\alpha(H)$ is $\EHtoFSnActsND$.
\end{definition}

\begin{example}
For $H=\FF$, with $r\defeq\abs*{B}$, the random subgroup
$\alpha(\FF) \le S_n$
is generated by the $r$ independent uniformly random
permutations $\alpha(b)$ for $b\in B$.
Each $d$-subset of $[n]$ has probability
$\binom{n}{d}^{-r}$ to be invariant under
$\alpha(\FF)$, so
$
\EHtoFSnActsND = \binom{n}{d}^{1-r}
$
and in particular
$\PR(\textup{There is an invariant set of size }d) = O(n^{(1-r)d})$.
\end{example}

It follows from \cite[Theorem 1.8]{PP15}
  that
for $d=1$, the expected number of common fixed points of $\alpha(H)$ is
\[ \EHtoFVarActsVar{S_n}{[n]}
= n^{1-\overline{\pi}(H)}\cdot (|\overline{\textup{Crit}}(H)| + O(n^{-1})), \]
where $\overline{\textup{Crit}}(H)\defeq \braces*{
    J\,\middle|\,
H\le J,\,\,
\rk(J) = \overline{\pi}(H)
}$ is finite and was defined in \cite[Corollary 1.5]{jaikin2024free}.
For $d>1$, the problem of computing
$\EHtoFVarActsVar{S_n}{\binom{[n]}{d}}$ is more difficult,
and has more complicated solutions;
see Table~\ref{tab_EHtoFSnActs} for examples.

\begin{table}[ht]
  \centering
  \begin{subtable}[t]{0.46\textwidth}
    \centering
    \begin{tabular}{c|c}
      $H$ & $\EHtoFVarActsVar{S_n}{[n]}$ \\ \hline
      $\FF$ & $n^{-2}$ \\[2pt]
      $\inner{x,y}$ & $n^{-1}$ \\[2pt]
      $\inner{xy^{-1},x^{3},y^{3}}$ & $(n-1)^{-1}$ \\[2pt]
      $\inner{xyx,yx^{2}}$ & $2\,n^{-1}$ \\[2pt]
      $\inner{[x,y],z^{210}}$ & $16\,(n-1)^{-1}$ \\
    \end{tabular}
    \caption{The case $d=1$.}
    \label{tab_EHtoFSnActsN}
  \end{subtable}
  \hfill
  \begin{subtable}[t]{0.46\textwidth}
    \centering
    \begin{tabular}{c|c}
      $H$ & $\EHtoFSnActsND$ \\ \hline
      $\FF$ & $\binom{n}{d}^{-2}$ \\[2pt]
      $\inner{x,y}$ & $\binom{n}{d}^{-1}$ \\[2pt]
      $\inner{x,y^2}$ & $(d+1)\binom{n}{d}^{-1}$ \\[2pt]
      $\inner{[x,y], z}$ & $(1+\varepsilon)\binom{n}{d}^{-1}$ \\
    \end{tabular}
    \caption{General $d$.}
    \label{tab_EHtoFSnActsND}
  \end{subtable}
  \caption{Examples where $B=\{x,y,z\}$; on the right,
  $\varepsilon = \sum_{k=1}^d \prn*{\binom{n}{k} - \binom{n}{k-1}}^{-1}$.}
  \label{tab_EHtoFSnActs}
\end{table}

Still, whenever $d$ is fixed (and $n\to\infty$),
Reiter was able to prove that
for \noncyc subgroups $H\le \FF$,
the probability of having an invariant set of size $d$
decays to $0$: he showed \cite[Corollary 4.6 and Claim 4.7]{reiter19} that
$\EHtoFSnActsND = O\prn*{n^{-d/\two}}$ for every $d\in\N$.
Similarly to Hanna Neumann,
Reiter conjectured \cite[Chapter 7]{reiter19} that the \two factor can be
removed, that is, that
$\EHtoFSnActsND = O\prn*{n^{-d}}$.
We prove this, and moreover identify the exact decay rate.
We show that the similarity is deeper than expected:
this \two is the same \two from Hanna Neumann's theorem
\eqref{eq_HannaNeumannThm}!

    \begin{theorem}
        \label{thm_Sn_spibard}
        Let $H\le \FF$ with $0<\rk(H)<\infty$, and let
        $d\in \N$.
        Then, as $n\to\infty$,
        \[ \EHtoFSnActsND =
        \Theta\prn*{\binom{n}{d}^{1-\overline{\pi}(H)}}. \]
        More generally, for every sequence of transitive actions
        $S_n\acts X_n$ with $\abs*{X_n}=n^{O(1)}$, the same result holds:
        $\EHtoFVarActsVar{S_n}{X_n}=\Theta\prn*{\abs*{X_n}^{1-\pibar(H)}}$.
        In particular, if $\rk(H)>1$ then $\pibar(H)\ge 2$, so
        $\EHtoFVarActsVar{S_n}{X_n} = O\prn*{\abs*{X_n}^{-1}}$, which is
        Reiter's conjecture for these actions.
    \end{theorem}
The proof gives the exponent
$s\pibar_d(H)$, which equals $\pibar(H)-1$ by
Theorem~\ref{thm_jz_stability}.
For example, by Theorem~\ref{thm_Sn_spibard}, for the set $[n]_d$ of $d$-tuples of
distinct elements of $[n]$,
\begin{equation}
    \label{eq_EHtoFSnActsTuples}
    \EHtoFVarActsVar{S_n}{[n]_d} =
    \Theta\prn*{n^{-d\cdot (\overline{\pi}(H)-1)}}.
\end{equation}

For every finite group acting transitively on a set $X$, Reiter \cite[Corollary 5.6]{reiter19}
proved the following theorem with $\abs*{X}^{1/\two}$ in place of $\abs*{X}$, and
conjectured that the \two is redundant.

\begin{theorem}
    [Reiter's conjecture]
    \label{thm_reiter_conjecture}
    Let $H\le \FF$ with $1<\rk(H)<\infty$ and denote
    $t\defeq \abs*{V(\Gamma_B(H))}$.
    Then for every finite group $G$ acting transitively on a set $X$,
    \[ \EHtoFVarActsVar{G}{X}\;\le\;\frac{\abs*{X^{t}/G}}{\abs*{X}}, \]
    where $X^{t}/G$ is the set of orbits of the diagonal action of $G$ on
    $X^{t}$.
\end{theorem}

For cyclic groups $H=\inner{w}$,
we prove the weaker bound
$\EHtoFVarActsVar{G}{X}\le \abs*{X^{t}/G}$
(see the inequality \eqref{eq_Gamma_poly_for_powers}).
It is attained, for example, when $w=b^2$ for some $b\in B$ and the action $G\acts X$ is generously transitive,
that is, for every $p, p'\in X$ there is $g\in G$ switching between them: $g.p=p', g.p'=p$.\footnote{By \cite[Proposition 2]{alejandro2003association}, the action $G\acts X$ is generously transitive if and only if the irreducible constituents of its permutation representation are all real and pairwise non-isomorphic.
One can prove the claim above using the Frobenius--Schur formula for $\E_{a\sim U(G)}[\chi(a^2)]$ for irreducible characters $\chi$.}
This much easier bound is already useful.
Larsen and Shalev \cite[Corollary 4.6]{larsen2018words} proved that for every $w\ne 1$, for almost
simple groups $G$, we have $\PR_{\alphaSimUHomFVar{G}}\prn*{\alpha(w)=1}\to0$ as $\abs*{G}\to\infty$; the
case where the rank of $G$ grows to $\infty$ (including the symmetric groups) also follows from the
inequality \eqref{eq_Gamma_poly_for_powers}.\footnote{We use here the classification of finite
simple groups.}

Theorem~\ref{thm_reiter_conjecture} is most useful when $|X^t/G|$ is small, as in
Examples~\ref{example_Sn_pibar} and~\ref{example_GLn_pibarK}.
Sequences of primitive actions $G_n\acts X_n$ of almost simple groups with
$|X_n|\to_{n\to\infty}\infty$ but $|X_n^t / G_n| = O(1)_{n\to\infty}$ for every $t\in \N$ are rare:
for large $n$, such actions are \textbf{standard},\footnote{
    By the Cameron--Kantor conjecture \cite{cameron1993random}, proved by Liebeck and Shalev
    \cite{liebeck1999simple}, with the constant $7$ found by Burness and his coauthors
    \cite{burness2018base}, every non-standard primitive action of an almost simple group $G$ on a set
    $X$ has a base of at most $7$ points, that is, at most $7$ points with trivial pointwise
    stabilizer. Then $\abs*{G}\le\abs*{X}^{7}$, so
    $\abs*{X^t/G}\ge\abs*{X}^{t}/\abs*{G}\ge\abs*{X}^{t-7}$. For the standard actions, see
    \cite{burness2018base}.}
like the actions of $S_n$ on $\binom{[n]}{d}$ and of $\GLnFq$ on $\Grasm_d(\F_q^n)$ (through
$\textup{PGL}_n(\F_q)$) in
Theorems~\ref{thm_Sn_spibard} and~\ref{thm_spibarqd}.

Despite this rarity, we believe that many more
sequences of group actions give
rise to stable invariants of words and subgroups in $\FF$ of a nature similar
to those above, corresponding to new \enquote{Hanna Neumann type}
conjectures.
For example, let $T_n$ be the rooted binary tree with $2^n$ leaves.
Its automorphism group is the iterated wreath product $G_n\defeq \Z/2\wr\Z/2\wr\ldots \wr\Z/2$ ($n$ times).
Reiter's bound for the action of $G_n$ on the set $X_n$ of leaves of $T_n$ gives
$\EHtoFVarActsVar{G_n}{X_n}
\le\frac{\textup{poly}(n)}{2^{n/\two}}$,
and the \two is redundant by Proposition~\ref{thm_equations_in_group_actions}.
What is the corresponding invariant, and the \enquote{binary tree version} of the HNC?

\subsubsection*{Proof of Reiter's conjecture
(Theorem~\ref{thm_reiter_conjecture}) assuming
Theorem~\ref{thm_Gamma_invariant_polymatroid_theorem}}
\label{subsec_equations}

In the rest of Section~\ref{subsec_Reiter} and in Section~\ref{section_counting_fixed_sub_spaces},
groups act on sets from the right: $p.g$ denotes the image of $p$ under $g$, so that
$p.(g_1g_2)=p.g_1.g_2$. This is less common than $g.p$, but it suits words in $\FF$, which are
read from left to right.

The proof given by Reiter in \cite[the opening of Chapter 5]{reiter19} is
very general: Reiter observed that
for a finitely generated subgroup
$H=\pi_1(\Gamma, v_0)
\le \pi_1(\Omega_B) = \FF$
and a finite group action 
$G\acts X$, 
for every homomorphism
$\alpha\in\Hom(\FF,G)$,
there is a bijective correspondence
between $\alpha(H)$-fixed points
$x_0\in X$ and functions
$f\colon V(\Gamma)\to X$
mapping $f(v_0) = x_0$ that are
$\alpha$\textbf{-valid}, that is,
for every $b\in B$ and $e\in E_b(\Gamma)$,
\begin{equation}
    \label{eq_valid_func}
    f(\src(e)).\alpha(b)=f(\tar(e)).
\end{equation}

The pair $(\Gamma, f)$ can be 
regarded as a system of equations:
The validity constraints 
\eqref{eq_valid_func}
can be seen as equations with 
variables $\{\alpha(b)\}_{b\in B}$ 
and constants $f(V)\subseteq X$;
see Figure~\ref{fig_Sn_pairs}
for an example.
A substitution
$\alpha\colon B\to G$
may or may not satisfy the 
system of equations;
let $\PR_{\alpha}(\Gamma, f)$ denote
the probability over a random 
$\alpha\sim \textup{Unif}(G^B)$
of satisfying $(\Gamma,f)$.
The group $G$ acts diagonally on
$X^{V(\Gamma)}$, and acts diagonally on
$G^B$ by conjugation.
For $g\in G$, note that
$\alpha$ satisfies $(\Gamma,f)$
if and only if
$\alpha'=g^{-1}\alpha g$ satisfies 
$(\Gamma, f.g)$, which is given by
the equations
\[f(\src(e)).g.\alpha'(b)=f(\tar(e)).g.\]
Therefore $\PR_{\alpha}(\Gamma,f)
=\PR_{\alpha}(\Gamma,f.g)$
depends only on the orbit
$\mathcal{O}(f)\subseteq X^{V(\Gamma)}$
under the diagonal action of $G$.
Reiter's observation then gives
\begin{equation}
    \label{eq_EHtoFGActsX_to_sum_O_P_O}
    \EHtoFVarActsVar{G}{X}
= \sum_{f\colon V(\Gamma)\to X}
\PR_{\alphaSimUHomFVar{G}}
(f\textup{ is }
\alpha\textup{-valid})
= \sum_{\mathcal{O}\in 
X^{V(\Gamma)}/G} |\mathcal{O}|
\cdot \PR(\mathcal{O}),
\end{equation}
where 
$\PR(\mathcal{O})\defeq 
\PR_{\alpha}(\Gamma, f)$ for some
arbitrary representative
$f\in \mathcal{O}$.
In the special case where
the action is $S_n\acts [n]$,
these notions played a key role
in \cite[Section 6]{PP15}: 
\begin{example}
    \label{example_Sn_pibar}
    For 
$n\ge k\in \N$, denote by 
$(n)_k\defeq n(n-1)\cdots (n-k+1)$
the falling factorial.
Then for a system of equations 
that is encoded by
a core graph $\Gamma$
and %
$f\colon V(\Gamma)\to [n]$,
let $\Delta$ be the graph obtained
from $\Gamma$ by first gluing 
together preimages of $f$, 
and then gluing together 
$b$-labeled edges with the same
source and target.
If $f$ is valid for some 
$\alpha\in \Hom(\FF,G)$,
then $\Delta$ is a core graph.\footnote{
Indeed, if, say, $\tar_{\Delta}$
were not injective,
there would be $b\in B$ and
$e,e'\in E_b(\Gamma)$ such that
$f(\tar(e))=f(\tar(e'))$ but
$f(\src(e))\neq f(\src(e'))$,
contradicting the validity
constraint
$f(\src(e)).\alpha(b)=f(\tar(e))$.
No vertex of $\Delta$ has degree less than $2$ either, as the quotient map
$\Gamma\to\Delta$ is an immersion, which cannot decrease degrees.
}
Now $f$ factors as 
$f\colon V(\Gamma)\twoheadrightarrow
V(\Delta)\hookrightarrow [n]$,
and we get
$|\mathcal{O}(f)|
=(n)_{|V(\Delta)|}$
and
$\PR_{\alpha}(\Gamma, f)^{-1}
=\prod_{b\in B} (n)_{|E_b(\Delta)|}.$
It follows that 
\begin{equation}
    \label{eq_contribution_Sn_n}
    \PR_{\alpha}(\Gamma, f)
\cdot |\mathcal{O}(f)|
= n^{\chi(\Delta)} \cdot (1+O(1/n)).
\end{equation}

Since, for $n\ge |V(\Gamma)|$, 
the number of orbits in the
diagonal action 
$S_n\acts [n]^{V(\Gamma)}$
is the Bell number\footnote{
    The Bell number $B_k$ is the number of equivalence relations on 
    $\{1,\ldots, k\}$. For $f\colon V(\Gamma)\to [n]$, the orbit $\mathcal{O}(f)$
    is $\braces*{f'\colon V(\Gamma)\to [n]
    \,\,\middle|\,\, 
    \forall u,v\in V(\Gamma):\, 
    f(u)=f(v)\iff f'(u)=f'(v)}$.
} 
$B_{|V(\Gamma)|} = O(1)_{n\to\infty}$,
the formula~\eqref{eq_EHtoFGActsX_to_sum_O_P_O} gives
\begin{equation}
    \label{eq_S_n_fixed_points_pibar}
    \EHtoFVarActsVar{S_n}{[n]}
    =\Theta(n^{1-\pibar(H)}).
\end{equation}
\end{example}

Another well-understood case is
the action of $G=\GLnFq$
on $X=\F_q^n\setminus\{0\}$:
\begin{example}
    \label{example_GLn_pibarK}
    By \cite[Section 2A]{ernst2024word},
for a fixed prime power $q$ and 
$n\to\infty$,
\begin{equation}
    \label{eq_contribution_GLn_Fqn}
    \PR_{\alpha}(\Gamma, f)\cdot 
|\mathcal{O}(f)|
= q^{n(1-\rk(I))}\cdot (1+O(1/q^n))
\end{equation}

where $I$ is the right ideal 
in the free group algebra 
$\F_q[\FF]$ 
generated by the 
linear dependencies between the
vectors $f(V(\Gamma))$, which is a free right $\F_q[\FF]$-module
\cite[Corollary 3]{cohn1964free}, as stated in Theorem~\ref{thm_cohn_lewin}.
Similarly to the case of $S_n$,
for $n\ge |V(\Gamma)|$,
the number of orbits in the 
diagonal action 
$\GLnFq\acts 
(\F_q^{n})^{V(\Gamma)}$ 
is the number of $\F_q$-linear subspaces of $\F_q^{V(\Gamma)}$,\footnote{
To describe the orbit of $f\colon V(\Gamma)\to \F_q^n$, identify $f$ with 
    $f\in \Hom_{\F_q}(\F_q^{V(\Gamma)},\F_q^n)$; then 
    $\mathcal{O}(f)=\{f': \ker(f)=\ker(f')\}$.
}
which is $O(1)_{n\to\infty}$.
We get 
\[\EHtoFVarActsVar{\GLnFq}{
    \F_q^n\setminus\{0\}
} = \Theta\prn*{q^{n(1-\pibar_q^{\textup{triv}}(H))}}\]
where $\pibar_q^{\textup{triv}}(H)$ 
is the minimal rank of a 
proper right ideal 
$I\triangleleft \F_q[\FF]$
containing 
$\{1-h\}_{h\in H}$.
\end{example}

As another example, if 
$G$ is an abelian group acting
transitively on $X$, and $\Gamma$ 
is connected, then for every $f$, we have
$|\mathcal{O}(f)|=|X|$ 
and $\PR_{\alpha}(\Gamma, f)
\in \{0, |X|^{-r}\}$,
where ${r\defeq\abs{\braces{b\in B\mid E_b(\Gamma)\ne\emptyset}}}$.
As a last example, for $H=\inner{x^k}$, we have
$\PR_{\alpha}(\Gamma, f)\cdot
|\mathcal{O}(f)| \in\{0,1\}$ for every $f$.
In the language of systems of equations, the bound Reiter proved in
\cite[Corollary 5.6]{reiter19} is
\[ \PR_{\alpha}(\Gamma, f)\cdot |\mathcal{O}(f)| \le |X|^{-1/\two} \]
for every connected core graph $\Gamma$ with $\chi(\Gamma)<0$ and
every system $(\Gamma,f)$, and
Theorem~\ref{thm_reiter_conjecture} replaces the exponent $1/\two$ by $1$.

If $\chi(\Gamma)=0$, that is,
$\pi_1(\Gamma)=\Z$ corresponds to a
word $w\in \FF$, there are
many systems $f$ of equations on $\Gamma$
where
$ \PR_{\alpha}(\Gamma, f)\cdot
|\mathcal{O}(f)| $
does not decay as $|X|$ grows,
as in the equation
\eqref{eq_contribution_Sn_n} of Example~\ref{example_Sn_pibar}.

Theorem~\ref{thm_reiter_conjecture} follows from Reiter's observation
\eqref{eq_EHtoFGActsX_to_sum_O_P_O} and the next proposition.
Let $\Gamma$ be a connected core graph, and let $H\defeq \pi_1(\Gamma, v_0)$
for some $v_0\in V(\Gamma)$.

\begin{proposition}
    \label{thm_equations_in_group_actions}
    Assume $\rk(H)>1$.
    Then for every finite transitive group action $G\acts X$ and every
    $f\colon V(\Gamma)\to X$,
    \[ |\mathcal{O}(f)|\cdot \PR_{\alpha}(\Gamma,f)\le |X|^{-1}. \]
\end{proposition}

Now we aim to conclude
Proposition~\ref{thm_equations_in_group_actions} from
Theorem~\ref{thm_Gamma_invariant_polymatroid_theorem}.
Let $G$ be a finite group, $V$ a finite set,
and for each $v\in V$ let $G_v\le G$ be a subgroup.
In \cite[Theorem 3.1]{chan2002relation}, 
Chan and Yeung showed that the 
function $\hp\colon 2^V\to \R$ defined by
\[\forall U\subseteq V:\,\,\, \hp(U)\defeq
\log\prn[\big]{\brackets[\big]{G:\bigcap\nolimits_{v\in U} G_v}}\]
is an entropy function, and hence a polymatroid.
We are interested in the special case where 
all the subgroups $G_v$ are conjugate;
equivalently, there is a transitive group action $G\acts X$
and $f\colon V\to X$ such that $G_v$ is the stabilizer of $f(v)$.
Recall that in Proposition~\ref{thm_equations_in_group_actions}, 
we are also given a connected
core graph $\Gamma$ with $V(\Gamma)=V$ and fundamental group $H\le \FF$,
and a random $\alphaSimUHomFVar{G}$,
and we wish to bound  $|\mathcal{O}(f)|
    \cdot \PR_{\alpha}(\Gamma,f)
    \le |X|^{-1}$.

\begin{proof}
    [Proof of Proposition~\ref{thm_equations_in_group_actions}
    assuming Theorem~\ref{thm_Gamma_invariant_polymatroid_theorem}]
    Note that $\hp(V)=\log|\mathcal{O}(f)|$ is the logarithm of the orbit
    size of $f$ under the diagonal action of $G$ on $X^V$.
    If $\PR_{\alpha}(\Gamma,f)=0$ the claimed bound is vacuous.
    Otherwise, there is some 
    $\alpha_0\in G^B$ for which all the equations 
    $f(\src(e)).\alpha_0(b)=f(\tar(e))\,\, (b\in B,e\in E_b(\Gamma))$
    hold, so $\hp$ is invariant:
    \[\forall b\in B:\, \forall U\subseteq E_b(\Gamma):\,\,\,\,
    \hp(\tar(U)) 
    = \log\abs*{\mathcal{O}\prn*{f\restriction_{\tar(U)}}}
    = \log\abs*{\mathcal{O}\prn*{f\restriction_{\src(U)}.\alpha_0(b)}}
    = \hp(\src(U)),\]
    and thus $\hp$ is a \Binv $\Gamma$-polymatroid.
    By Theorem~\ref{thm_Gamma_invariant_polymatroid_theorem},
    $\chi(\hp)\le-\hp(\{v\})$ for some $v\in V(\Gamma)$.\footnote{
        For cyclic $H$ the same argument needs $\hp$ to be tight, which here
        says that for every $v\in V$ the restriction
        $\mathcal{O}(f)\to\mathcal{O}\prn*{f\restriction_{V\setminus\{v\}}}$ is
        bijective, that is, that each value $f(v)$ can be recovered from the
        others.
    }
    Here $\hp(\{v\})=\log\brackets*{G:G_v}=\log|X|$, since $G\acts X$ is
    transitive.
    Since $\braces*{\alpha(b)}$ are independent, uniform $G$-elements,
    $\log\PR_{\alpha}(\Gamma,f)=-\sum_{b\in B} \hp(\src(E_b(\Gamma)))$, so
    \[ \chi(\hp)=\hp(V)-\sum_{b\in B}\hp(\src(E_b(\Gamma)))
    =\log\abs*{\mathcal{O}(f)}+\log\PR_{\alpha}(\Gamma,f), \]
    so
    $\abs*{\mathcal{O}(f)}\cdot\PR_{\alpha}(\Gamma,f)\le\abs*{X}^{-1}$.
\end{proof}

In this proof, $\hp(\{v\})=\log|X|$ for every $v\in V(\Gamma)$ and
$\chi(\hp)=\log\prn[\big]{\abs*{\mathcal{O}(f)}\cdot\PR_{\alpha}(\Gamma,f)}$,
so any upper bound on the Euler characteristic of such a polymatroid bounds
$\abs*{\mathcal{O}(f)}\cdot\PR_{\alpha}(\Gamma,f)$ as well.
As $\chi(\hp)\le 0$ by Corollary~\ref{cor_chi_nonpositive}, when $\chi(\Gamma)=0$ we
have
\begin{equation}
    \label{eq_Gamma_poly_for_powers}
    |\mathcal{O}(f)|\cdot \PR_{\alpha}(\Gamma,f)\le 1.
\end{equation}

\input{Bowen_dynamical.tex}

%% file: Bowen_dynamical.tex
\subsection{Bowen's \texorpdfstring{$f$}{f}-invariant and conditional entropy}
\label{section_f_invariant}

This subsection is about measurable dynamics. The rest of the paper does not depend
on it.
Here Theorem~\ref{thm_Gamma_invariant_polymatroid_theorem} is applied for a
polymatroid whose values are conditional Shannon entropies
\eqref{eq_conditional_entropy} and the subgroup $H\le \FF$ fixing a generating
partition.

We follow the definitions and notation of Bowen
\cite[Section 1.2]{Bowen2009AbramovRohlin}.
Let $\FF$ act by measure preserving transformations on a
standard probability space $(X,\mathcal{X},\mu)$,\footnote{
    That is, a space isomorphic modulo null sets to an interval with Lebesgue
    measure together with countably many atoms.
} and let
$\mathcal{A}\subseteq \mathcal{X}$ be an $\FF$-invariant sub-$\sigma$-algebra;
all $\sigma$-algebras and partitions here are understood modulo null sets.
The \textbf{join} $\alpha\vee\beta$ of two partitions is the partition into the
sets $C\cap C'$ with $C\in\alpha$ and $C'\in\beta$, and
$\bigvee_{i=1}^k \alpha_i$ is defined inductively.
The Shannon entropy of a countable partition $\alpha$, and its conditional
entropy given $\mathcal{A}$, are
\begin{equation}
    \label{eq_conditional_entropy}
    \mathcal{H}_\mu(\alpha)\defeq-\sum_{C\in\alpha}\mu(C)\log\mu(C),
\qquad
\mathcal{H}_\mu(\alpha\mid \mathcal{A})\defeq
-\sum_{C\in\alpha}\int_X \E_\mu\brackets*{\one_C\mid \mathcal{A}}
\log \E_\mu\brackets*{\one_C\mid \mathcal{A}}\,d\mu,
\end{equation}
with $0\log 0=0$.
Finally, $\alpha$ is \textbf{generating} if the smallest $\FF$-invariant
$\sigma$-algebra containing it is $\mathcal{X}$.
For a partition $\beta$ with $\mathcal{H}_\mu(\beta)<\infty$, Bowen defined
\cite[Section 1.2]{Bowen2009AbramovRohlin}
\begin{equation}
    \label{eq_Phi}
    \Phi_B(\beta\mid \mathcal{A})
    \defeq (1-2|B|)\mathcal{H}_\mu(\beta\mid \mathcal{A})
    +\sum_{b\in B}\mathcal{H}_\mu(\beta\vee b^{-1}\beta\mid \mathcal{A}),
\end{equation}
and the \textbf{relative $f$-invariant} of the action over
$\mathcal{A}$ \cite[Definition 14]{Bowen2009AbramovRohlin},
\begin{equation}
    \label{eq_relative_f}
    f_{\FF}(X,\mu/\mathcal{A})\defeq\inf\nolimits_{n\ge 0}
    \Phi_B\prn[\big]{\bigvee_{g\in \B_n}g\beta\mid \mathcal{A}},
\end{equation}
where $\B_n\subseteq \FF$ is the ball of radius $n$ in
$\textup{Cay}(\FF,B)$.\footnote{
    Shriver gave a random model interpretation of the relative $f$-invariant
    \eqref{eq_relative_f} through non-uniform random sofic approximations
    \cite{Shriver2023RelativeF}.
}
Bowen proved \cite[Theorem 5.3]{Bowen2009AbramovRohlin} that this definition
does not depend on which generating $\beta$ with $\mathcal{H}_\mu(\beta)<\infty$ is
taken.
We abbreviate $f_{\FF}(X,\mu)\defeq f_{\FF}(X,\mu/\{\emptyset,X\})$, as in
\cite{Bowen2010Invariant}.
The $f$-invariant can be negative, and, unlike the Kolmogorov--Sinai entropy
of a single transformation, it can increase under a factor map.

We improve two bounds on the $f$-invariant, in the case where some generating
partition $\alpha$ has a large stabilizer
$\stab_{\FF}(\alpha)\defeq\braces*{g\in \FF: g\alpha=\alpha}$.\footnote{
    Recall from Section~\ref{section_intro} that $g\in\stab_{\FF}(\alpha)$ is
    allowed to permute the pieces of $\alpha$.
}
Throughout, $\alpha$ is a countable generating partition with
$\mathcal{H}_\mu(\alpha)<\infty$.
The stabilizer hypothesis depends on which generating partition is taken; the
$f$-invariant $f_{\FF}(X,\mu/\mathcal{A})$ in Theorem~\ref{thm_entropic_stabilizer_gap} does not
depend on $\alpha$.

\begin{theorem}
    \label{thm_entropic_stabilizer_gap}
    Let $\mathcal{A}$ be an $\FF$-invariant sub-$\sigma$-algebra.
    \begin{enumerate}[label=\textup{(\alph*)}]
        \item \label{item_relative_bowen}
        $f_{\FF}(X,\mu/\mathcal{A})\le\mathcal{H}_\mu(\alpha\mid\mathcal{A})$.
        \item \label{item_zero_gap}
        $f_{\FF}(X,\mu/\mathcal{A})\le\mathrlap{0}\hphantom{-\mathcal{H}_\mu(\alpha\mid\mathcal{A})}$\quad
        if $\stab_{\FF}(\alpha)\ne 1$.
        \item \label{item_entropic_gap}
        $f_{\FF}(X,\mu/\mathcal{A})\le-\mathcal{H}_\mu(\alpha\mid\mathcal{A})$\quad
        if $\stab_{\FF}(\alpha)$ is \noncyc.
    \end{enumerate}
\end{theorem}

Part~\ref{item_relative_bowen} is not new.
It is an immediate corollary of the definitions \eqref{eq_Phi} and
\eqref{eq_relative_f}, since
$\mathcal{H}_\mu(\beta\vee b^{-1}\beta\mid \mathcal{A})\le 2\mathcal{H}_\mu(\beta\mid \mathcal{A})$.
Part~\ref{item_zero_gap} also follows from the formula for the relative $f$-invariant in
\cite[Theorem 3.5]{Seward2016FiniteEntropy} together with \cite[Lemma 3.6]{Seward2016FiniteEntropy}.\footnote{
    Part~\ref{item_entropic_gap} also follows from \cite[Theorem 3.5 and Lemma 3.6]{Seward2016FiniteEntropy}
    when $a^k,b^m\in\stab_{\FF}(\alpha)$ for some distinct $a,b\in B$ and some $k,m\in\Z\setminus\{0\}$.
}

Let $\pi\colon X\to Y$ be a factor and denote $\mathcal{A}=\pi^{-1}(\mathcal{Y})$. 
Bowen proved the Abramov--Rohlin formula
\cites[Theorem 1.3]{Bowen2009AbramovRohlin}[Section 2]{BowenGutman2013Corrigendum}
\begin{equation}
    \label{eq_abramov_rohlin}
    f_{\FF}(X,\mu)=f_{\FF}(Y,\nu)+f_{\FF}(X,\mu/\mathcal{A})
\end{equation}
whenever $X$ and $Y$ have generating partitions of finite Shannon entropy, and
the inequality
\begin{equation}
    \label{eq_bowen_absolute}
    f_{\FF}(Y,\nu)\ge f_{\FF}(X,\mu)-\mathcal{H}_\mu(\alpha)
\end{equation}
when $\alpha$ is finite and $Y$ has a finite generating partition
\cite[Theorem 1.7]{Bowen2010FactorBound}.
We improve the inequality \eqref{eq_bowen_absolute}:

\begin{corollary}
    [Factor monotonicity and strictness]
    \label{thm_f_factor_monotonicity}
    Let $\pi\colon (X,\mathcal{X},\mu)\to (Y,\mathcal{Y},\nu)$ be a factor map,
    and assume that $X$ and $Y$ have finite real $f$-invariants.
    \begin{enumerate}[label=\textup{(\alph*)}]
        \item \label{item_factor_monotone}
        If $\stab_{\FF}(\alpha)$ is non-trivial, then
        $f_{\FF}(Y,\nu)\ge f_{\FF}(X,\mu)$.
        \item \label{item_factor_gap}
        If $\stab_{\FF}(\alpha)$ is \noncyc, then
        \[
            f_{\FF}(Y,\nu)\ge f_{\FF}(X,\mu)
            +\mathcal{H}_\mu\prn*{\alpha\mid \pi^{-1}(\mathcal{Y})}.
        \]
        Consequently every proper factor has a strictly larger $f$-invariant, and
        equality of the two $f$-invariants holds only if $\pi$ is an isomorphism
        modulo null sets.
    \end{enumerate}
    In particular, under the hypothesis of \ref{item_factor_gap} the action is
    \textbf{coalescent}: every measure preserving $\FF$-equivariant
    endomorphism of $X$ is an isomorphism.
\end{corollary}

\begin{proof}
    [Proof of Corollary~\ref{thm_f_factor_monotonicity} from
    Theorem~\ref{thm_entropic_stabilizer_gap}]
    Let $\mathcal{A}=\pi^{-1}(\mathcal{Y})$. The inequalities in
    \ref{item_factor_monotone} and \ref{item_factor_gap} follow from the
    Abramov--Rohlin formula \eqref{eq_abramov_rohlin} with
    Theorem~\ref{thm_entropic_stabilizer_gap}\ref{item_zero_gap} and
    \ref{item_entropic_gap}, respectively.

    If $f_{\FF}(Y,\nu)=f_{\FF}(X,\mu)$, for example if $Y=X$, then
    $0\ge\mathcal{H}_\mu(\alpha\mid\mathcal{A})$ by \ref{item_factor_gap}, so
    $\mathcal{H}_\mu(\alpha\mid\mathcal{A})=0$. Then $\alpha$ is $\mathcal{A}$-measurable;
    as $\mathcal{A}$ is invariant and $\alpha$ is generating, $\mathcal{A}=\mathcal{X}$, so
    $\pi$ is an isomorphism modulo null sets, because the spaces are standard.
\end{proof}

\begin{example}
    \label{example_homogeneous_processes}
    Let $I$ be a transitive right $\FF$-set, fix $i_0\in I$, and let
    $J\defeq\stab_{\FF}(i_0)$, so that $I$ is identified with
    $J\backslash\FF$.
    For a finite or countable alphabet $\Sigma$, the group $\FF$ acts on $\Sigma^I$ by
    $(g.x)_i=x_{ig}$.
    Let $\mu$ be any invariant probability measure on $\Sigma^I$, and let
    $\xi\defeq\braces*{\{x\in \Sigma^I: x_{i_0}=a\}: a\in \Sigma}$ be the partition
    according to the coordinate at $i_0$.
    Then $\xi$ is generating and $J\le \stab_{\FF}(\xi)$.
    If $\mathcal{H}_\mu(\xi)<\infty$, then by
    Theorem~\ref{thm_entropic_stabilizer_gap} with $\alpha=\xi$, we have
    $f_{\FF}(\Sigma^I,\mu/\mathcal{A})\le 0$ for every invariant $\mathcal{A}$ if
    $J$ is non-trivial, and
    $f_{\FF}(\Sigma^I,\mu/\mathcal{A})\le -\mathcal{H}_\mu(\xi\mid \mathcal{A})$
    if $J$ is \noncyc. Hence, by
    Corollary~\ref{thm_f_factor_monotonicity}, the process is coalescent if
    $J$ is \noncyc and its $f$-invariant is a finite real number.
\end{example}

\subsubsection*{Entropy on a Stallings core graph}
\label{subsec_entropy_on_a_core}

Let $H\le \stab_{\FF}(\alpha)$ be a non-trivial finitely generated subgroup,
and let $\Gamma\defeq\Gamma_B(H)$ be its Stallings core graph, which is the
core of the Schreier graph $H\backslash\textup{Cay}(\FF,B)$, so that
$V(\Gamma)\subseteq H\backslash\FF$.
For a coset $v=Hg$ let $\alpha_v\defeq g^{-1}\alpha$, which is well defined
because every element of $H$ fixes $\alpha$, and for a finite set $U$ of cosets
define
\begin{equation}
    \label{eq_core_rank}
    \hp(U)\defeq \mathcal{H}_\mu\prn[\big]{\bigvee_{v\in U}\alpha_v\mid \mathcal{A}},
    \qquad \hp(\emptyset)\defeq 0 .
\end{equation}
For finite sets $U\subseteq U'$ and $W$ of cosets, by the chain rule for conditional
entropy,
\[
    \hp(U\cup W)-\hp(U)
    =\mathcal{H}_\mu\prn[\big]{\bigvee_{v\in W}\alpha_v\mid
    \mathcal{A}\vee\bigvee_{v\in U}\alpha_v}\ge 0,
\]
and, as conditioning does not increase entropy,
$\hp(U'\cup W)-\hp(U')\le\hp(U\cup W)-\hp(U)$. Hence $\hp$ is monotone and
submodular, that is, a polymatroid.
We stress that $\hp$ is defined on every finite subset of $H\backslash\FF$.

\begin{lemma}
    \label{lemma_core_upper_bound}
    The restriction $\hp\restriction_{\Gamma}$ of $\hp$ to $V(\Gamma)$ is a \Binv
    $\Gamma$-polymatroid, with $\hp(\{v\})=\mathcal{H}_\mu(\alpha\mid \mathcal{A})$ for
    every $v\in V(\Gamma)$, and
    $f_{\FF}(X,\mu/\mathcal{A})\le \chi(\hp\restriction_{\Gamma})$.
\end{lemma}

\begin{proof}
    For every finite $U\subseteq H\backslash\FF$ and $b\in B$ denote
    $Ub\defeq\{Hgb: Hg\in U\}$.
    For $g\in\FF$ and $b\in B$, by definition,
    $\alpha_{Hgb}=b^{-1}\alpha_{Hg}$, so
    $\bigvee_{v\in Ub}\alpha_v=b^{-1}\bigvee_{v\in U}\alpha_v$, and
    as $\mu$ and $\mathcal{A}$ are invariant, $\hp(Ub)=\hp(U)$; in particular
    $\hp(\tar(E'))=\hp(\src(E'))$ for every $E'\subseteq E_b(\Gamma)$.
    Similarly
    $\hp(\{Hg\})=\mathcal{H}_\mu(g^{-1}\alpha\mid\mathcal{A})
    =\mathcal{H}_\mu(\alpha\mid\mathcal{A})$ for every $g\in\FF$.
    Then $\src(E_b(\Gamma))\cdot b=V(\Gamma)\cap V(\Gamma)\cdot b$, and by
    submodularity
    \begin{equation}
        \label{eq_union_bound}
        \begin{split}
            \hp(V(\Gamma)\cup V(\Gamma)b)
            &\le \hp(V(\Gamma))+\hp(V(\Gamma)b)-\hp(V(\Gamma)\cap V(\Gamma)b)\\
            &= 2\hp(V(\Gamma))-\hp(\src(E_b(\Gamma))).
        \end{split}
    \end{equation}
    Let $\beta\defeq\bigvee_{v\in V(\Gamma)}\alpha_v$.
    Every $\alpha_v$ is a translate of the generating partition $\alpha$, so
    $\beta$ is generating, $\mathcal{H}_\mu(\beta)\le\abs{V(\Gamma)}\,\mathcal{H}_\mu(\alpha)<\infty$, and
    $\mathcal{H}_\mu(\beta\mid \mathcal{A})=\hp(V(\Gamma))$, while
    $\mathcal{H}_\mu(\beta\vee b^{-1}\beta\mid \mathcal{A})
    =\hp(V(\Gamma)\cup V(\Gamma)b)$.
    By the definition \eqref{eq_Phi} and the bound \eqref{eq_union_bound},
    \[ \Phi_B(\beta\mid \mathcal{A})
    =(1-2\abs{B})\hp(V(\Gamma))+\sum_{b\in B}\hp(V(\Gamma)\cup V(\Gamma)b)
    \le\hp(V(\Gamma))-\sum_{b\in B}\hp(\src(E_b(\Gamma)))
    =\chi(\hp\restriction_{\Gamma}), \]
    and $f_{\FF}(X,\mu/\mathcal{A})\le\Phi_B(\beta\mid \mathcal{A})$ by the
    definition \eqref{eq_relative_f}, with $n=0$.
\end{proof}

\begin{proof}
    [Proof of Theorem~\ref{thm_entropic_stabilizer_gap}]
    For part~\ref{item_zero_gap}, let $H$ be any non-trivial cyclic
    subgroup of $\stab_{\FF}(\alpha)$. By Lemma~\ref{lemma_core_upper_bound} and
    Corollary~\ref{cor_chi_nonpositive},
    $f_{\FF}(X,\mu/\mathcal{A})\le\chi(\hp\restriction_{\Gamma})\le 0$.
    For part~\ref{item_entropic_gap}, let $H\le\stab_{\FF}(\alpha)$ be a
    subgroup of rank $2$. By Lemma~\ref{lemma_core_upper_bound} and
    Theorem~\ref{thm_Gamma_invariant_polymatroid_theorem},
    $f_{\FF}(X,\mu/\mathcal{A})\le\chi(\hp\restriction_{\Gamma})\le -\hp(\{v\})
    =-\mathcal{H}_\mu(\alpha\mid \mathcal{A})$ for some $v\in V(\Gamma)$.
\end{proof}

%% file: polymatroids.tex
\section{The \texorpdfstring{$\Gamma$}{Gamma}-Polymatroid Theorem and its Proof}
\label{section_proof_of_Gamma_poly_thm}

In this section we develop the theory of
$\Gamma$-polymatroids, and state and prove the $\Gamma$-polymatroid theorem
in its general form (Theorem~\ref{thm_Gamma_polymatroid_theorem}), that is,
without the $B$-invariance hypothesis of
Theorem~\ref{thm_Gamma_invariant_polymatroid_theorem}.
Thanks to the existence of stackings 
for non-power words \cite[Lemma 3.4]{louder2014stacking},
the second part
of the $\Gamma$-polymatroid theorem 
(Theorem~\ref{thm_Gamma_polymatroid_theorem}),
about $H=\inner{w}$,
is much easier than the first part,
and is proved in
Corollary~\ref{corollary_Gamma_poly_thm_for_words}.
For \noncyc groups $H$, we show in 
Proposition~\ref{prop_ec_h_is_monotone} how to reduce
Theorem~\ref{thm_Gamma_polymatroid_theorem}
to polymatroids on graphs of subgroups of $H$,
prove the $\Gamma$-polymatroid theorem for stackable
graphs (Theorem~\ref{thm_Gamma_poly_subgrp}),
and finally prove 
Lemma~\ref{lemma_stackable_subgroup} about the existence 
of a \noncyc stackable subgroup.

\subsection*{Some polymatroid theory}

The following characterization of polymatroids is used only in
Appendix~\ref{appendix_ta_shma}.

\begin{observation}
    \label{obs_diminishing_returns}
    Let $V$ be a finite set and let $\hp\colon 2^{V}\to\R$ satisfy $\hp(\emptyset)=0$.
    Then $\hp$ is a polymatroid if and only if
    \begin{equation}
        \label{eqn_diminishing_returns}
        \hp\prn*{U\cup\{u\}}-\hp(U)\ge 0
        \qquad\textup{and}\qquad
        \hp\prn*{U\cup\{u\}}-\hp(U)\ge \hp\prn*{U'\cup\{u\}}-\hp(U')
    \end{equation}
    for every $u\in V$ and $U\subseteq U'\subseteq V$. Indeed, by
    \cite[Theorem 44.1]{schrijver2003combinatorial}, submodularity is equivalent to the second
    inequality for $u\notin U'$.
\end{observation}

The following definition of a morphism of polymatroids is not entirely
standard: for matroids, it is described in
\cite[Definition 2.7 and Lemma 2.8]{heunen2018category}, under the name
\enquote{strong maps}, and in \cite[Definition 1.1]{eur2020logarithmic}, under the name \enquote{morphisms}.

\begin{definition}
    \label{def_polymatroid_morphism}
    Let $\hp_1, \hp_2$ be polymatroids on sets $V_1, V_2$.
    A \textbf{morphism} $\phi\colon \hp_1\to \hp_2$ is a function
    $\phi\colon V_1\to V_2$ such that for every $U\subseteq U'\subseteq V_1$,
    \[ \hp_1(U')-\hp_1(U)\ \ge\ \hp_2(\phi(U'))-\hp_2(\phi(U)). \]
    If moreover $\hp_2(\phi(U))=\hp_1(U)$ for every $U\subseteq V_1$,
    we say that $\phi$ is \textbf{invariant}.
\end{definition}

\begin{definition}
    Let $V_1, V_2$ be sets and $\eta\colon V_1\to V_2$ any function. 
    Given a polymatroid 
    $\hp\colon 2^{V_2}\to\R$, we define 
    $\eta^*\hp\colon 2^{V_1}\to\R$ by 
    $\eta^*\hp(U)\defeq \hp(\eta(U))$.
\end{definition}

Note that if $V_1\subseteq V_2$ and $\hp_2$
is a polymatroid on $V_2$, then
$\hp_1 = \hp_2\restriction_{V_1}$ 
is a polymatroid on $V_1$ and the inclusion map 
$V_1\hookrightarrow V_2$ is an invariant morphism $\hp_1\to\hp_2$.
This can be generalized:

\begin{proposition}
    $\eta^*\hp$ is a polymatroid, and
    $\eta\colon \eta^*\hp\to \hp$ an invariant morphism.
\end{proposition}

\begin{proof}
    Clearly $\eta^*\hp(\emptyset)
    =\hp(\emptyset)=0$.
    If $U\subseteq U'\subseteq V_1$, clearly 
    $\eta(U)\subseteq \eta(U')$, so $\eta^*\hp$ 
    is monotone. Finally, for $A, B\subseteq V_1$, we have
    $\eta(A\cap B)\subseteq \eta(A)\cap \eta(B)$ 
    and $\eta(A\cup B)= \eta(A)\cup \eta(B)$, so
    \begin{equation*}
    \begin{split}
        \eta^*\hp(A)
        +\eta^*\hp(B)
        &=\hp(\eta(A))+\hp(\eta(B))
        \\&\ge \hp(\eta(A)\cup\eta(B))
        +\hp(\eta(A)\cap\eta(B))
        \\&\ge \hp(\eta(A\cup B))+
        \hp(\eta(A\cap B))
        \\&=\eta^*\hp(A\cup B)+\eta^*\hp(A\cap B). 
    \end{split}
    \end{equation*}
Now $\eta\colon \eta^*\hp\to \hp$ is an invariant morphism by definition.
\end{proof}

\begin{definition}
    Given a (total) order $\prec$ on a finite set $V$ 
    and a polymatroid $\hp$ on $V$, denote 
    \[\partial\hp(v)\defeq 
    \partial_{\prec}\hp(v)\defeq \hp(\{u: u\preceq v\}) - \hp(\{u: u\prec v\}) \ge 0\]
    for every $v\in V$.
Note that $\sum_{v\in V} \partial\hp(v)
=\hp(V)$
and $\partial\hp(v)=\hp(\{v\})$ for the $\prec$-minimal $v$.
\end{definition}

\begin{proposition}
\label{prop_injective_morphism_ineq_delta}
For $i=1,2$, let $V_i$ be a finite set with 
a total order $\prec_i$ and a polymatroid $\hp_i$.
Let $\phi\colon (V_1,\hp_1)\to (V_2,\hp_2)$ 
be an injective morphism of polymatroids, 
which is monotonically increasing.
Then $\partial\hp_2(\phi(v))\le \partial\hp_1(v)$ for every $v\in V_1$.
\end{proposition}

\begin{proof}
    Let $v\in V_1$.
    \begin{equation*}
        \begin{array}{ll}
            \qquad\qquad\qquad\qquad\qquad
            \partial\hp_1(v)
            &=\hp_1(\{u\in V_1 :u\preceq_1 v\})
            -\hp_1(\{u\in V_1 :u\prec_1 v\})
            \\(\phi\textup{ is a morphism})
            &\ge \hp_2(\{\phi(u) :u\preceq_1 v\})
            -\hp_2(\{\phi(u): u\prec_1 v\})
            \\(\phi\textup{ is increasing})
            &= \hp_2(u'\in\Img(\phi): u'\preceq_2 \phi(v))
            -\hp_2(u'\in\Img(\phi): u'\prec_2 \phi(v))
            \\(\hp_2\textup{ is submodular})
            &\ge \hp_2(u'\in V_2:u'\preceq_2 \phi(v))
            -\hp_2(u'\in V_2:u'\prec_2 \phi(v))
            \\&=\partial\hp_2(\phi(v)).
        \end{array}
    \end{equation*}
\end{proof}

\subsection*{$\Gamma$-polymatroids}
A $B$-graph (Definition~\ref{def_stallings_core_graph}) is the same as a graph
$\Gamma$ with an \emph{immersion} $\eta\colon\Gamma\immerse\Omega_B$:
the label of an edge is its image, and $\src,\tar$ are injective on
each $E_b(\Gamma)=\eta^{-1}(b)$ exactly when $\eta$ is locally
injective.
From here on we work with morphisms between arbitrary graphs.

We follow Serre's definition of graphs \cite{serre2002trees}:
\begin{definition}
    [{\cite[Section 2.1, Definition 1]{serre2002trees}}]
    \label{def_serre_graph}
    A graph $\Gamma$ consists of a set $V=V(\Gamma)$ of \textbf{vertices},
    a set of (oriented) \textbf{edges},
    two functions $\src,\tar$ from the edges to $V$, called
    \textbf{source} and \textbf{target},
    and a fixed-point-free involution $e\mapsto\overline{e}$ on the
    edges, called \textbf{inverse}, satisfying
    $\src(\overline{e})=\tar(e)$.
\end{definition}
A \textbf{geometric edge} is a pair $\{e,\overline{e}\}$.
We fix an \textbf{orientation} of $\Omega$, one edge out of each
geometric edge, and pull it back to $\Gamma$; throughout, $E(\Gamma)$
denotes the resulting set of oriented edges, one per geometric edge.
Orientations only serve to read paths as words (see the discussion
following Definition~\ref{def_stallings_core_graph}).

\begin{definition}
    [{\cite[7.2 and 7.3(a)]{Stallings1983TopologyOF}}]
    The \textbf{degree} of a vertex $v\in V(\Gamma)$
    is $|\src^{-1}(v)|+|\tar^{-1}(v)|$.
    A \textbf{core graph} is a finite graph
    where all degrees are at least $2$, and the \textbf{core} of a finite graph is its largest
    subgraph that is a core graph.
\end{definition}

\begin{definition}
    A \textbf{morphism} between graphs 
    $\eta\colon \Gamma\to \Omega$ 
    is a function
    $\eta\colon V(\Gamma)\cup E(\Gamma)
    \to V(\Omega)\cup E(\Omega)$ 
    mapping $V(\Gamma)\to V(\Omega)$ and 
    $E(\Gamma)\to E(\Omega)$, which commutes with 
    $\src$, $\tar$, and inverse: for every $e\in E(\Gamma)$, we have
    $\src(\eta(e))=\eta(\src(e))$,
    $\tar(\eta(e))=\eta(\tar(e))$ and 
    $\overline{\eta(e)}=\eta(\overline{e})$.
    A morphism $\eta$ is an \textbf{immersion},
    denoted by %
    $\immerse$,
    if it is locally injective, that is, 
    for every edge $e\in E(\Omega)$, the 
    restrictions 
    $\src\restriction_{\eta^{-1}(e)},
    \tar\restriction_{\eta^{-1}(e)}\colon 
    \eta^{-1}(e)\to V(\Gamma)$ 
     are 
    injective.
\end{definition}

\begin{definition}
    \label{def_eta_polymatroid}
    Let $\eta\colon\Gamma\immerse\Omega$
    be an immersion of graphs.
    An $\eta$\textbf{-polymatroid} $\hp$ is a collection 
    of polymatroids: for every 
    $b\in V(\Omega)\cup E(\Omega)$, a 
    polymatroid $\hp^b$ on $\eta^{-1}(b)$,
    such that for every $e\in E(\Omega)$,
    the injection
    $\src\restriction_{\eta^{-1}(e)}
    \colon (\eta^{-1}(e), \hp^e)\to
    (\eta^{-1}(\src(e)), \hp^{\src(e)})$
    is a morphism of polymatroids,
    and similarly for $\tar$.
\end{definition}

When the immersion $\eta$ is understood from the context, as it always
is below, we also call an $\eta$-polymatroid a $\Gamma$\textbf{-polymatroid
over} $\Omega$, or simply a $\Gamma$-polymatroid; the two terminologies mean
the same thing.

\begin{definition}
    \label{def_vertex_edge_tight}
    Let $\eta\colon\Gamma\immerse\Omega$ be an immersion of graphs. An $\eta$-polymatroid $\hp$ is
    \textbf{vertex-tight} if $\hp^v$ is tight for every $v\in V(\Omega)$, and \textbf{edge-tight}
    if $\hp^e$ is tight for every $e\in E(\Omega)$.
\end{definition}

\begin{observation}
    \label{obs_tight_implies_vertex_tight}
    Let $\eta\colon\Gamma\immerse\Omega$ be an immersion from a core graph $\Gamma$. Then every
    edge-tight $\eta$-polymatroid $\hp$ is vertex-tight.
    Indeed, let $v\in V(\Omega)$ and $u\in\eta^{-1}(v)$. As $\Gamma$ is a core graph, $u=\src(f)$ or
    $u=\tar(f)$ for some $f\in E(\Gamma)$; say $u=\src(f)$ (the other case is the same). Let
    $e\defeq\eta(f)$, so that $\src(e)=v$, and let
    $U\defeq\src(\eta^{-1}(e))$. As, by Definition~\ref{def_eta_polymatroid}, $\src\restriction_{\eta^{-1}(e)}$ is an injective
    morphism $\hp^e\to\hp^v$, and $\hp^e$ is tight,
    $\hp^v(U)-\hp^v(U\setminus\{u\})\le\hp^e(\eta^{-1}(e))-\hp^e(\eta^{-1}(e)\setminus\{f\})=0$.
    Hence, by submodularity,
    $\hp^v(\eta^{-1}(v))-\hp^v(\eta^{-1}(v)\setminus\{u\})\le\hp^v(U)-\hp^v(U\setminus\{u\})\le0$, so
    $\hp^v(\eta^{-1}(v)\setminus\{u\})=\hp^v(\eta^{-1}(v))$ by monotonicity.
\end{observation}

\begin{definition}
    \label{def_morphism_of_immersions}
    Let $\iota_{\Gamma}\colon \Gamma\immerse\Omega,
    \iota_{\Delta}\colon\Delta\immerse\Omega$
    be immersions of graphs. A \textbf{morphism}
    between immersions 
    $\eta\colon\iota_{\Gamma}\to\iota_{\Delta}$
    is an immersion $\eta\colon\Gamma\immerse\Delta$
    such that $\iota_{\Delta}\circ\eta=\iota_{\Gamma}$.
    Given an $\iota_{\Delta}$-polymatroid $\hp$,
    its \textbf{pullback} $\eta^*\hp$ 
    through $\eta$ is 
    the collection of polymatroids 
    $\prn*{\eta\restriction_{
        \iota_{\Gamma}^{-1}(b)
        }}^*\hp^b$
    on $\iota_{\Gamma}^{-1}(b)$,
    for every $b\in V(\Omega)\cup E(\Omega)$.
\end{definition}

This construction is clearly functorial:
$(\eta_2\circ\eta_1)^*\hp=\eta_1^*\eta_2^*\hp$.
For a subgraph $\Delta'\subseteq\Delta$, we denote by $\hp\restriction_{\Delta'}$ the
pullback of $\hp$ through the inclusion $\Delta'\hookrightarrow\Delta$, that is, the restriction
of each $\hp^b$ to $\Delta'\cap\iota_{\Delta}^{-1}(b)$.

\begin{proposition}
    $\eta^*\hp$ is an $\iota_{\Gamma}$-polymatroid.
\end{proposition}

\begin{proof}
    We verify that $\src$ is a morphism;
    the proof for $\tar$ is the same.
    For $e\in E(\Omega)$ and 
    $E_0\subseteq E_1\subseteq \eta^{-1}(e)$,
    \begin{equation*}
        \begin{split}
            \eta^*\hp^e(E_1)-
            \eta^*\hp^e(E_0)
            &= \hp^e(\eta(E_1))
            -  \hp^e(\eta(E_0))
            \\&\ge \hp^{\src(e)}(\src(\eta(E_1)))
            -      \hp^{\src(e)}(\src(\eta(E_0)))
            \\&= \hp^{\src(e)}(\eta(\src(E_1)))
            -    \hp^{\src(e)}(\eta(\src(E_0)))
            \\&= \eta^*\hp^{\src(e)}(\src(E_1))
            -    \eta^*\hp^{\src(e)}(\src(E_0)).
        \end{split}
    \end{equation*}
\end{proof}

An important ingredient in the proof, which we borrow from the proof of
Wise's conjecture by Louder and Wilton and slightly generalize, is the notion of
\textbf{stackings}
of graphs \cite[Definition 2.1]{louder2014stacking}.\footnote{
    In \cite[Definition 2.1]{louder2014stacking} the stacked graph
    is required to be a disjoint union of circles.
    We omit this restriction.
}

The \textbf{topological realization}
of a graph $\Gamma$ is the quotient of
the topological space
$V(\Gamma)\cup(E(\Gamma)\times [0,1])$
by the equivalence relation generated by
$\src(e)\sim (e,0)$ and $\tar(e)\sim (e,1)$
for every $e\in E(\Gamma)$.
When we discuss topological properties of a graph
we implicitly regard its realization.
Following \cite{louder2014stacking},
in this definition we regard graphs as topological spaces; the combinatorial
version we work with is Definition~\ref{def_stacking_combinatorial} below.

\begin{definition}
    \label{def_stacking_topological}
    Let $\eta\colon \Gamma\to\Omega$ be a continuous map between graphs.
    A \textbf{stacking} of $\eta$ is an embedding
    $\hat{\eta}\colon \Gamma\hookrightarrow \Omega\times \R$
    into the trivial $\R$-bundle
    $p_{\Omega}\colon\Omega\times \R\twoheadrightarrow\Omega$,
    such that $p_{\Omega}\circ \hat{\eta}=\eta$.
\end{definition}

If a map $\eta\colon \Gamma\to\Omega$ admits some
stacking, we say it is \textbf{stackable}.
The stackable subgroup lemma announced in the introduction is as follows.

\begin{lemma}
    \label{lemma_stackable_subgroup}
    Let $\eta\colon \Gamma\immerse\Omega$ be
    an immersion of connected graphs with
    negative Euler characteristics. Then
    there exists another connected graph
    $\Sigma$ with negative Euler characteristic,
    and an immersion $\nu\colon\Sigma\immerse\Gamma$,
    such that $\eta\circ\nu\colon\Sigma\immerse\Omega$
    is stackable.
\end{lemma}

The following diagram illustrates
Definition~\ref{def_stacking_topological}
(on the left) and Lemma~\ref{lemma_stackable_subgroup}
(on the right).
\[\begin{tikzcd}[ampersand replacement=\&]
	\& {\Omega\times\R} \&\&\&\&\& {\Omega\times\R} \\
	\Gamma \& \Omega \&\&\& \Sigma \& \Gamma \& \Omega
	\arrow["{p_{\Omega}}", two heads, from=1-2, to=2-2]
	\arrow["{p_{\Omega}}", two heads, from=1-7, to=2-7]
	\arrow["{\widehat{\eta}}", dashed, hook, from=2-1, to=1-2]
	\arrow["\eta"', from=2-1, to=2-2]
	\arrow["{\exists\widehat{\eta\circ\nu}}", dashed, hook, from=2-5, to=1-7]
	\arrow["{\exists\nu}"', dashed, from=2-5, to=2-6]
	\arrow["\eta"', from=2-6, to=2-7]
\end{tikzcd}\]

By \cite[Lemma 3.4]{louder2014stacking},
an immersed cycle $\eta\colon \Ss^1\immerse \Omega_B$ is stackable
if and only if it does not factor through a non-trivial covering map $\Ss^1\to \Ss^1$,
or equivalently, if the corresponding word $w\in \FF$ is not a proper power.
It is evident that if $w$ is stackable,
it is not a power;
we provide parts of the proof from
\cite[Lemma 3.4]{louder2014stacking}
in Proposition~\ref{prop_ff_strictly_compressed_Z_tower_stackable},
showing that non-power words,
and more generally strictly compressed subgroups,
are stackable.

The following definition is a combinatorial
version of Definition~\ref{def_stacking_topological}
(\cite[Definition 2.1]{louder2014stacking}).
\begin{definition}
    \label{def_stacking_combinatorial}
    Let $\eta\colon \Gamma\immerse\Omega$
    be an immersion of graphs.
    A \textbf{stacking} of $\eta$ 
    is a collection of (total) orders $\prec_b$
    on $\eta^{-1}(b)$ for every 
    $b\in V(\Omega)\cup E(\Omega)$,
    such that for every 
    $e\in E(\Omega)$,
    the injection
    $\src\restriction_{\eta^{-1}(e)}
     \colon (\eta^{-1}(e),\prec_e)
     \hookrightarrow 
     (\eta^{-1}(\src(e)), \prec_{\src(e)})$
    is monotonically increasing, and similarly 
    for $\tar$.
\end{definition}

For an immersion $\eta\colon\Gamma\immerse\Omega$, a stacking
$\hat{\eta}=(\eta,\psi)\colon\Gamma\hookrightarrow\Omega\times\R$ in the sense of
Definition~\ref{def_stacking_topological} gives a stacking in the sense of
Definition~\ref{def_stacking_combinatorial}: for every $x\in V(\Omega)\cup E(\Omega)$, order
$\eta^{-1}(x)$ by the height $\psi$.
For edges this is well defined, since $\hat{\eta}$ is injective: if $e\ne e'$ have the same image,
then either $\psi(e,t)<\psi(e',t)$ for all $t\in[0,1]$, or vice versa, where $(e,t)$ is a point of
the topological realization. Conversely, given a stacking in the sense of
Definition~\ref{def_stacking_combinatorial}, choose $\psi\colon V(\Gamma)\to\R$ monotonically increasing
on $(\eta^{-1}(x),\prec_x)$ for every $x\in V(\Omega)$;
then, for example, the linear interpolation
$\hat{\eta}(e,t)\defeq\prn[\big]{(\eta(e),t),\,(1-t)\psi(\src(e))+t\psi(\tar(e))}$
is a stacking in the sense of Definition~\ref{def_stacking_topological}.

\begin{definition}
    \label{def_eta_polymatroid_euler_char}
    Let $\eta\colon\Gamma\immerse\Omega$.
    The \textbf{Euler characteristic} of an
    $\eta$-polymatroid $\hp$ is
    \[\chi(\hp)\defeq
    \sum_{v\in V(\Omega)} \hp^v(\eta^{-1}(v)) -
    \sum_{e\in E(\Omega)} \hp^e(\eta^{-1}(e)). \]
\end{definition}

\begin{remark}
    \label{remark_B_invariant_determines}
    For a Stallings core graph $\Gamma$, a \Binv $\Gamma$-polymatroid $\hp$ determines a
    $\Gamma$-polymatroid over $\Omega_B$, by $\hp^{v_B}\defeq\hp$ and $\hp^b\defeq\hp\circ\src=\hp\circ\tar$
    on $E_b(\Gamma)$ for $b\in B$, with the same Euler characteristic. This
    $\Gamma$-polymatroid is vertex-tight if and only if $\hp$ is tight in the sense of
    Definition~\ref{def_tight_polymatroid}, and it is edge-tight if and only if
    $\hp\restriction_{\src(E_b(\Gamma))}$ is tight for every $b\in B$.
\end{remark}

By Remark~\ref{remark_B_invariant_determines}, the following theorem generalizes
Theorem~\ref{thm_Gamma_invariant_polymatroid_theorem}.

\begin{theorem}
    \label{thm_Gamma_polymatroid_theorem}
    Let $\Gamma$ be a Stallings core graph
    with fundamental group $H\le \FF$,
    and $\hp$ be a $\Gamma$-polymatroid.
    Assume that either
    \begin{itemize}
        \item $\rk(H)>1$, or
        \item $H=\inner{w}$ is generated by a non-power $w\in \FF$, and $\hp$ is vertex-tight.
    \end{itemize}
    Then there is some $b\in B$ and $e\in E_b(\Gamma)$ such that $\chi(\hp)\le -\hp^b(\{e\})$.
\end{theorem}
\begin{observation}
    \label{obs_tight_on_edges}
    Let $\Gamma$ be a Stallings core graph and let $\hp$ be a \Binv
    $\Gamma$-polymatroid whose restriction to $\src(E_b(\Gamma))$ is tight for
    every $b\in B$.
    Then $\hp$ is tight in the sense of Definition~\ref{def_tight_polymatroid}.
    Indeed, by Remark~\ref{remark_B_invariant_determines}, the $\Gamma$-polymatroid over $\Omega_B$
    determined by $\hp$ is edge-tight, so, by Observation~\ref{obs_tight_implies_vertex_tight}, it is
    vertex-tight, that is, $\hp$ is tight.
\end{observation}

A \textbf{spanning forest} is an acyclic
subgraph $\mathcal{F}\subseteq \Gamma$
on the same vertex set as $\Gamma$.
For a stacking $\prec$ of an immersion $\eta\colon\Gamma\immerse\Omega$, with orders $\prec_b$ for
$b\in V(\Omega)\cup E(\Omega)$, and an $\eta$-polymatroid $\hp$, we define each $\partial\hp^b$ with
respect to $\prec_b$, and we say that a vertex or an edge $c$ of $\Gamma$ is $\prec$\textbf{-minimal}
($\prec$\textbf{-maximal}) if it is minimal (maximal) in $\eta^{-1}(\eta(c))$ with respect to
$\prec_{\eta(c)}$. We often omit the superscript of $\hp$
where it can be understood from the context.

\begin{lemma}
    \label{lemma_forest_bound}
    Let $\eta\colon\Gamma\immerse\Omega$ 
    with a stacking $\prec$ and an 
    $\eta$-polymatroid $\hp$.
    Let $\mathcal{F}$ be a 
    spanning forest of $\Gamma$, 
    with connected components 
    $\{\mathcal{T}_i\}_{i\in\mathcal{I}}$.
    For every $i\in \mathcal{I}$ choose 
    $v_i\in V(\mathcal{T}_i)$.
    Then
    \[\chi(\hp)\le 
    \sum_{i\in\mathcal{I}} 
    \partial\hp^{\eta(v_i)}(v_i) - 
    \sum_{e\in E(\Gamma)\setminus E(\mathcal{F})}
    \partial\hp^{\eta(e)}(e).\]
\end{lemma}

\begin{proof}
    For every $i\in\mathcal{I}$, 
    define a bijection 
    $\zeta_i\colon E(\mathcal{T}_i)
    \to V(\mathcal{T}_i)\setminus\{v_i\}$ 
    by choosing $\zeta_i(e)$,
    for every $e\in E(\mathcal{T}_i)$,
    to be the endpoint of $e$ which is 
    farther from $v_i$ on $\mathcal{T}_i$ 
    (where the distance is the length of the 
    unique path in $\mathcal{T}_i$).
    By Proposition~\ref{prop_injective_morphism_ineq_delta},
    $\partial\hp^{\eta(e)}(e)
    \ge \partial\hp^{\eta(\zeta_i(e))}(\zeta_i(e))$.
    Therefore 
    \[
    \sum_{v\in V(\mathcal{T}_i)}
    \partial\hp^{\eta(v)}(v) - 
    \sum_{e\in E(\mathcal{T}_i)}
    \partial\hp^{\eta(e)}(e) \le 
    \partial\hp^{\eta(v_i)}(v_i).\]
    Summing over $i\in\mathcal{I}$, we get
    \[    \chi(\hp)
            = \sum_{v\in V(\Gamma)} 
            \partial\hp^{\eta(v)}(v)
            - \sum_{e\in E(\Gamma)}
            \partial\hp^{\eta(e)}(e)
            \le \sum_{i\in\mathcal{I}} 
    \partial\hp^{\eta(v_i)}(v_i) - 
    \sum_{e\in E(\Gamma)\setminus E(\mathcal{F})}
    \partial\hp^{\eta(e)}(e).\qedhere\]
\end{proof}

To use this lemma optimally,
for every $e\in E(\Omega)$ let 
$m_e$ be the $\prec$-minimal
edge in $\eta^{-1}(e)$, so that
$\partial\hp^{e}(m_e)=\hp^{e}(\{m_e\})$.
If $E(\Gamma)\setminus \{m_e\}_{e\in E(\Omega)}$ is the edge set of a spanning
forest $\mathcal{F}$, apply the lemma to this $\mathcal{F}$,
and for every $i\in\mathcal{I}$ choose
$v_i\in V(\mathcal{T}_i)$ with minimal $\partial\hp(v_i)$.
Our use of this lemma is not always optimal,
though,
as we wish to get bounds that do not depend 
on $\prec$.
Specifically, we get
\begin{corollary}
    \label{corollary_bound_no_tree_components}
    Let $\eta\colon\Gamma\immerse\Omega$ be an immersion with a stacking $\prec$, let $\hp$ be an
    $\eta$-polymatroid, and let $E_0\subseteq E(\Gamma)$ be a set of $\prec$-minimal edges such
    that every connected component of $\Gamma\setminus E_0$ is either not a tree, or contains a
    vertex $v$ with $\partial\hp(v)=0$. Then $\chi(\hp)\le -\sum_{e\in E_0}\hp(\{e\})$.
\end{corollary}

\begin{proof}
    Let $C_1,\ldots,C_m$ be the connected components of $\Gamma\setminus E_0$, let $J$ be the set
    of indices $j$ for which $C_j$ has no vertex $v$ with $\partial\hp(v)=0$, and let $\mathcal{F}$ be
    a spanning forest of $\Gamma\setminus E_0$ with one tree in every $C_j$. For $j\notin J$, choose
    $v_j\in V(C_j)$ with $\partial\hp(v_j)=0$. For $j\in J$, the component $C_j$ is not a tree, so it has an edge
    $f_j\notin E(\mathcal{F})$; let $v_j$ be an endpoint of $f_j$, so that
    $\partial\hp(v_j)\le\partial\hp(f_j)$ by Proposition~\ref{prop_injective_morphism_ineq_delta}.
    The edges $f_j$ (for $j\in J$) and the edges in $E_0$ are distinct and outside $\mathcal{F}$.
    Moreover, $\partial\hp(e)\ge0$ for every $e\in E(\Gamma)$, and $\partial\hp(e)=\hp(\{e\})$ for $e\in E_0$, as $e$ is
    $\prec$-minimal. Hence, by Lemma~\ref{lemma_forest_bound},
    \begin{align*}
        \chi(\hp)
        &\le\sum_{j=1}^m\partial\hp(v_j)-\sum_{j\in J}\partial\hp(f_j)-\sum_{e\in E_0}\hp(\{e\})\\
        &=\sum_{j\in J}\prn[\big]{\partial\hp(v_j)-\partial\hp(f_j)}-\sum_{e\in E_0}\hp(\{e\})
        \le-\sum_{e\in E_0}\hp(\{e\}).
        \qedhere
    \end{align*}
\end{proof}

Another way to get bounds independent of $\prec$ is to restrict to vertex-tight
$\eta$-polymatroids.\footnote{
    Tightness is a natural condition: every polymatroid $\hp$ on $V$ decomposes
    uniquely as $\hp=\hp^{\textup{modular}}+\hp^{\textup{tight}}$, where
    $\hp^{\textup{tight}}$ is tight and $\hp^{\textup{modular}}$ is \textbf{modular}, that is,
    $\hp^{\textup{modular}}(U)=\sum_{u\in U}\hp^{\textup{modular}}(\{u\})$ for every $U$; the
    modular part is determined by
    $\hp^{\textup{modular}}(\{v\})=\hp(V)-\hp(V\setminus\{v\})$.
    The decomposition was proved in \cite{matuvs2016entropy}.
}

\begin{corollary}
    \label{corollary_Gamma_poly_thm_for_words}
    Let $w\in\FF$ be a non-power, let
    $\iota_w\colon \Gamma_w\immerse\Omega_B$
    be an immersion reading $w$, and let
    $\hp$ be a vertex-tight $\iota_w$-polymatroid.
    Then $\chi(\hp)\le -\hp^{\iota_w(e)}(\{e\})$
    for some $e\in E(\Gamma_w)$.
\end{corollary}

\begin{proof}
    By \cite[Lemma 3.4]{louder2014stacking},
    there is a stacking $\prec$ of $\iota_w$.
    Let $e\in E(\Gamma_w)$ be some $\prec$-minimal
    edge, and $v$ some $\prec$-maximal
    vertex. Now $\partial\hp(e)=\hp(\{e\})$,
    and since $\hp$ is vertex-tight,
    $\partial\hp(v)=0$.
    Removing $e$ from $\Gamma_w$ gives a tree 
    $\mathcal{T}$ 
    (a path, in fact)
    which we can root at $v$, so by 
    Lemma~\ref{lemma_forest_bound},
    $\chi(\hp)\le-\hp(\{e\})$.
\end{proof}

We do not use Remark~\ref{remark_k_alternating}, which is included for its own interest.

\begin{remark}
\label{remark_k_alternating}
Following \cite[Definition 1.14]{TaShma26}, for a non-power $w\in\FF$ and $k\in\Z_{\ge1}$, we say
that a stacking $\prec$ of the immersion $\iota_w\colon\Gamma_w\immerse\Omega_B$ reading $w$
is $k$\textbf{-alternating} if some $k$ $\prec$-maximal edges and $k$
$\prec$-minimal edges of $\Gamma_w$ appear in alternating order around the
cycle $\Gamma_w$.
Ta Shma conjectured that almost (in some natural probabilistic sense) all words in $\FF$
have an $(\rk(\FF)-1)$-alternating stacking (private communication, June 2026), and proved a
version of this conjecture, with stackings replaced by \enquote{partial} stackings
\cite[Definition 1.7 and Theorem 1.18]{TaShma26}.
For a word $w$ whose immersion $\iota_w$ has a $k$-alternating stacking, and
an edge-tight $\iota_w$-polymatroid $\hp$, we
have $\chi(\hp)\le-\sum_{i=1}^k \hp(\{e_i\})$
for some edges $e_i\in E(\Gamma_w)$, which improves
Corollary~\ref{corollary_Gamma_poly_thm_for_words} when $k\ge2$.
More generally, given an
immersion $\eta\colon\Gamma\immerse\Omega$ of graphs,
we say that a stacking $\prec$ of $\eta$ is
$k$\textbf{-alternating} if there are $k$
$\prec$-minimal edges
$e_1,\ldots,e_k\in E(\Gamma)$
such that every connected component of
$\Gamma\setminus\{e_1,\ldots,e_k\}$
contains a $\prec$-maximal edge.
For such a stacking and such edges $e_1,\ldots,e_k$, we have $\chi(\hp)\le -\sum_{i=1}^k \hp(\{e_i\})$ for every
edge-tight $\eta$-polymatroid $\hp$.
Indeed, let $C$ be a connected component of $\Gamma\setminus\{e_1,\ldots,e_k\}$, and let $f\in E(C)$
be $\prec$-maximal. As $\hp^{\eta(f)}$ is tight, $\partial\hp(f)=0$, so, by
Proposition~\ref{prop_injective_morphism_ineq_delta}, $\partial\hp(\src(f))\le\partial\hp(f)=0$.
Thus every connected component of $\Gamma\setminus\{e_1,\ldots,e_k\}$ has a vertex $v$ with
$\partial\hp(v)=0$, so, by Corollary~\ref{corollary_bound_no_tree_components} for
$E_0=\{e_1,\ldots,e_k\}$, we have $\chi(\hp)\le-\sum_{i=1}^k\hp(\{e_i\})$.
\end{remark}

\subsection*{Reduction to subgraphs}
The following lemma is a 
polymatroid version of
\textbf{Shearer's inequality} 
\cite[Section V]{shearerchung1986};
it follows from \cite[Corollary I]{madiman2010information} and submodularity:

\begin{lemma}
    \label{lemma_Shearer}
    Let $\hp, \lambda\colon 2^V\to \R_{\ge 0}$ for a finite set $V$, 
    where $\hp$ is a polymatroid and 
    $\lambda$ is a fractional cover, 
    that is, for every $v\in V$ we have
    $\sum_{U\ni v} \lambda(U)\ge 1$.
    Then
    $\inner{\lambda, \hp} \defeq \sum_{U\subseteq V} \lambda(U)\hp(U) \ge \hp(V)$. 
\end{lemma}

If $\hp$ is a polymatroid on $V$ and $T\subseteq V$, we denote the $T$\textbf{-contraction} of $\hp$ by
\[ \hp(U\,|\,T)\defeq \hp(U\cup T)-\hp(T)
\quad\quad(\textup{which is }\le \hp(U)-\hp(U\cap T)\textup{ by submodularity}). \]
The $T$-contraction $\hp(\cdot\,|\,T)$ is a polymatroid 
(see \cite[Section 2]{bonin2021excluded}).

\begin{proposition}
\label{prop_ec_h_is_monotone}
Let $\iota_\Gamma\colon\Gamma\immerse\Omega$
and $\iota_{\Delta}\colon\Delta\immerse\Omega$
be immersions of \textbf{core} graphs, 
let $\eta\colon
\iota_\Gamma\to\iota_{\Delta}$ be a morphism,
and let $\hp$ be an $\iota_{\Delta}$-polymatroid.
Then $\chi(\eta^*\hp)\ge \chi(\hp)$, with
equality if $\eta$ is surjective.
\end{proposition}

\begin{proof}
    If $\eta$ is surjective the claim is obvious.
Otherwise, write 
$\eta = \eta_{\textup{inj}}
\circ \eta_{\textup{sur}}$ 
where 
$\eta_{\textup{inj}}\colon \textup{Im}(\eta)
\hookrightarrow \Delta$ is the inclusion map. 
Then $\chi(\eta^*\hp)
= \chi(\eta_{\textup{sur}}^*\eta_{\textup{inj}}^*\hp) 
= \chi(\eta_{\textup{inj}}^*\hp)$, 
so we may assume 
$\Gamma\subseteq \Delta$, 
$\iota_\Gamma=\iota_\Delta\restriction_\Gamma$ and 
$\eta^*\hp = \hp\restriction_\Gamma$. 
Let $e\in E(\Omega)$.
To ease notation, denote 
$\Gamma_b\defeq \iota_{\Gamma}^{-1}(b)$
and 
$\Delta_b\defeq \iota_{\Delta}^{-1}(b)$
for every $b\in V(\Omega)\cup E(\Omega)$. 
Since 
$\src(\Gamma_e)\subseteq 
\src(\Delta_e)\cap 
\Gamma_{\src(e)}$,
by monotonicity,
$\hp^{\src(e)}(\src(\Gamma_e))\le 
\hp^{\src(e)}(\src(\Delta_e)\cap \Gamma_{\src(e)})$,
so
\begin{equation*}
    \begin{split}
        \hp^{\src(e)}(\src(\Delta_e)) - 
        \hp^{\src(e)}(\src(\Gamma_e)) 
        &\ge 
        \hp^{\src(e)}(\src(\Delta_e)) -
        \hp^{\src(e)}(\src(\Delta_e)
        \cap \Gamma_{\src(e)})
        \\&\ge \hp^{\src(e)}(\src(\Delta_e) \,|\, 
        \Gamma_{\src(e)}).
    \end{split}
\end{equation*}
The same holds for $\tar$.
Since 
$\src\colon
\Delta_e\hookrightarrow
\Delta_{\src(e)},
\tar\colon
\Delta_e\hookrightarrow
\Delta_{\tar(e)}
$
are morphisms,
\begin{equation}
    \label{eq_diff_Delta_Gamma_bound}
    \begin{split}
        \hp^e(\Delta_e) - 
        \hp^e(\Gamma_e)
        &\ge 
        \max\braces*{
            \hp^{\src(e)}(\src(\Delta_e)) - 
            \hp^{\src(e)}(\src(\Gamma_e)),\,\,
            \hp^{\tar(e)}(\tar(\Delta_e)) - 
            \hp^{\tar(e)}(\tar(\Gamma_e))
        }
        \\&\ge 
        \max\braces*{
            \hp^{\src(e)}(\src(\Delta_e)\,|\,
            \Gamma_{\src(e)}),\,\,
            \hp^{\tar(e)}(\tar(\Delta_e)\,|\,
            \Gamma_{\tar(e)})
        }
        \\&\ge
        \frac{1}{2}\prn*{
            \hp^{\src(e)}(\src(\Delta_e)\,|\,
            \Gamma_{\src(e)}) +
            \hp^{\tar(e)}(\tar(\Delta_e)\,|\,
            \Gamma_{\tar(e)})        
        }.
    \end{split}
\end{equation}

For every $v\in V(\Omega)$, define 
$\lambda_v\colon 2^{\Delta_v}\to \R$
by 
\[ \lambda_v(U)\defeq 
\frac{1}{2}
\prn*{
    \sum_{e\in \src^{-1}(v)} 
    \one_{U=\src(\Delta_e)} + 
    \sum_{e\in \tar^{-1}(v)}
    \one_{U=\tar(\Delta_e)} 
}. \]
Since $\Delta$ is a \textbf{core} graph,
for every $u\in \Delta_v$ we have 
\[ \sum_{U\ni u} \lambda_v(U) 
= 
\prn*{
    \abs*{\src^{-1}(u)} + \abs*{\tar^{-1}(u)}
    }/2
=\deg(u)/2
\ge 1.\]
Therefore $\lambda_v$ is a fractional cover
of $\Delta_v$,
so by applying
Shearer's lemma (Lemma~\ref{lemma_Shearer}) 
to the polymatroid
$\hp^{v}(\,\cdot\,\,|\,\Gamma_v)$
restricted to $\Delta_v$, and 
summing the inequality \eqref{eq_diff_Delta_Gamma_bound}
over all $e\in E(\Omega)$,
\begin{align*}
        \sum_{e\in E(\Omega)} (\hp^e(\Delta_e) -
    \hp^e(\Gamma_e))
    &\ge \frac{1}{2}\sum_{e\in E(\Omega)}\prn*{
        \hp^{\src(e)}(\src(\Delta_e)\,|\,
            \Gamma_{\src(e)}) +
            \hp^{\tar(e)}(\tar(\Delta_e)\,|\,
            \Gamma_{\tar(e)})
    }
    \\&= \sum_{v\in V(\Omega)}
    \sum_{U\subseteq \Delta_v}
    \lambda_v(U)\cdot \hp^{v}(U\,|\,\Gamma_v)
    \\&\ge \sum_{v\in V(\Omega)} \hp^{v}(\Delta_v\,|\,\Gamma_v)
    = \sum_{v\in V(\Omega)} \prn*{
        \hp^v(\Delta_v)-\hp^v(\Gamma_v)
    }.\qedhere
\end{align*}
\end{proof}

\begin{corollary}
    \label{cor_chi_nonpositive}
    Let $\iota_\Delta\colon\Delta\immerse\Omega$ be an immersion from a core graph, and let
    $\hp$ be an $\iota_\Delta$-polymatroid. Then $\chi(\hp)\le 0$. In particular,
    $\chi(\hp)\le 0$ for every Stallings core graph $\Gamma$ and every \Binv $\Gamma$-polymatroid
    $\hp$.
\end{corollary}

\begin{proof}
    Apply Proposition~\ref{prop_ec_h_is_monotone} to the empty graph (with no vertices), its
    immersion in $\Omega$, and its morphism $\eta$ to $\iota_\Delta$: then
    $\chi(\hp)\le\chi(\eta^*\hp)=0$. The Stallings core graph $\Gamma$ is a core graph, and, by
    Remark~\ref{remark_B_invariant_determines}, a \Binv $\Gamma$-polymatroid $\hp$ determines a
    $\Gamma$-polymatroid over $\Omega_B$ with the same Euler characteristic.
\end{proof}

Proposition~\ref{prop_ec_h_is_monotone} needs $\Delta$ to be a core graph: its
proof uses that every vertex of $\Delta$ has degree at least $2$.
Removing a leaf reduces the general case to it.

\begin{lemma}
    \label{lemma_leaf_removal}
    Let $\iota_{\Delta}\colon\Delta\immerse\Omega$ be an immersion of finite
    graphs, let $u\in V(\Delta)$ be a vertex of degree $1$, let $e$ be the edge of $\Delta$ at $u$,
    let $\Delta'\defeq\Delta\setminus\{u,e\}$, and let $\hp$ be an $\iota_{\Delta}$-polymatroid.
    Then $\chi(\hp\restriction_{\Delta'})\ge \chi(\hp)$.
\end{lemma}

\begin{proof}
    Assume that $u=\src(e)$; the case $u=\tar(e)$ is the same.
    Denote $u_0\defeq \iota_{\Delta}(u)$, $e_0\defeq \iota_{\Delta}(e)$, and
    $\Delta_c\defeq \iota_{\Delta}^{-1}(c)$ for
    $c\in V(\Omega)\cup E(\Omega)$.
    Then $\Delta'_{u_0}=\Delta_{u_0}\setminus\{u\}$,
    $\Delta'_{e_0}=\Delta_{e_0}\setminus\{e\}$, and $\Delta'_c=\Delta_c$ for
    every other cell $c$, so
    \[ \chi(\hp)-\chi(\hp\restriction_{\Delta'})
    =\prn[\big]{\hp^{u_0}(\Delta_{u_0})-\hp^{u_0}(\Delta_{u_0}\setminus\{u\})}
    -\prn[\big]{\hp^{e_0}(\Delta_{e_0})-\hp^{e_0}(\Delta_{e_0}\setminus\{e\})}. \]
    As $u$ has degree $1$, we have
    $\src(\Delta_{e_0}\setminus\{e\})\subseteq\Delta_{u_0}\setminus\{u\}$, and
    as $\src$ is injective on $\Delta_{e_0}$, we have
    $\src(\Delta_{e_0}\setminus\{e\})=\src(\Delta_{e_0})\setminus\{u\}$.
    By submodularity and by
    Definition~\ref{def_eta_polymatroid},
    \[ \hp^{u_0}(\Delta_{u_0})-\hp^{u_0}(\Delta_{u_0}\setminus\{u\})
    \le \hp^{u_0}(\src(\Delta_{e_0}))-\hp^{u_0}(\src(\Delta_{e_0}\setminus\{e\}))
    \le \hp^{e_0}(\Delta_{e_0})-\hp^{e_0}(\Delta_{e_0}\setminus\{e\}). \qedhere \]
\end{proof}

\begin{corollary}
    \label{cor_reduction_to_subgraphs}
    Let $\iota_\Gamma\colon\Gamma\immerse\Omega$ and
    $\iota_{\Delta}\colon\Delta\immerse\Omega$ be immersions of finite graphs,
    with $\Gamma$ a core graph and with no connected component of $\Delta$ a
    tree, let $\eta\colon \iota_\Gamma\to\iota_{\Delta}$ be a morphism, and let
    $\hp$ be an $\iota_{\Delta}$-polymatroid.
    Then $\chi(\eta^*\hp)\ge \chi(\hp)$.
\end{corollary}

\begin{proof}
    Let $\Delta_0\subseteq\Delta$ be the largest subgraph all of whose degrees
    are at least $2$, obtained from $\Delta$ by removing leaves one at a time.
    As no connected component of $\Delta$ is a tree, $\Delta_0$ meets every
    connected component of $\Delta$, so $\Delta_0$ is a core graph, and
    $\chi(\hp\restriction_{\Delta_0})\ge\chi(\hp)$ by
    Lemma~\ref{lemma_leaf_removal}.
    Since $\eta$ is an immersion, it maps the edge-ends at a vertex
    $v\in V(\Gamma)$ to distinct edge-ends at $\eta(v)$, so every vertex of
    the subgraph $\eta(\Gamma)\subseteq\Delta$ has degree at least $2$ in
    $\eta(\Gamma)$; hence $\eta(\Gamma)\subseteq \Delta_0$ and
    $\eta^*\hp=\eta^*(\hp\restriction_{\Delta_0})$.
    Now, by Proposition~\ref{prop_ec_h_is_monotone},
    \[\chi(\eta^*\hp)=\chi(\eta^*(\hp\restriction_{\Delta_0}))
    \ge\chi(\hp\restriction_{\Delta_0})\ge\chi(\hp). \qedhere\]
\end{proof}

We are ready to prove the first part
of the $\Gamma$-polymatroid theorem, 
assuming Lemma~\ref{lemma_stackable_subgroup}.

\begin{theorem}
    \label{thm_Gamma_poly_subgrp}
    Let $\iota_{\Gamma}\colon\Gamma\immerse\Omega$ 
    be an immersion of core graphs with $\chi(\Gamma)<0$,
    and let $\hp$ be an $\iota_{\Gamma}$-polymatroid.
    Then $\chi(\hp)\le -\hp(\{e\})$ 
    for some $e\in E(\Gamma)$.
\end{theorem}

\begin{proof}
    [Proof assuming Lemma~\ref{lemma_stackable_subgroup}]
    Let $C$ be a connected component of $\Gamma$ with $\chi(C)<0$. By
    Lemma~\ref{lemma_stackable_subgroup} applied to $C\immerse\iota_{\Gamma}(C)$, there is a
    connected graph $\Sigma$ with $\chi(\Sigma)<0$ and an immersion $\eta\colon\Sigma\immerse\Gamma$,
    such that $\iota_{\Sigma}\defeq \iota_{\Gamma}\circ\eta$ has a stacking $\prec$; replacing
    $\Sigma$ by the core of $\Sigma$, we may assume that $\Sigma$ is a core graph. Let $e_0\in E(\Sigma)$ be a $\prec$-minimal edge.
    Since $\Sigma$ is a core graph and $\chi(\Sigma)<0$, no connected component
    of $\Sigma\setminus\{e_0\}$ is a tree, so by
    Corollary~\ref{corollary_bound_no_tree_components},
    $\chi(\eta^*\hp)\le-\eta^*\hp(\{e_0\})=-\hp(\{\eta(e_0)\})$.
    By Proposition~\ref{prop_ec_h_is_monotone}, 
    $\chi(\hp)\le \chi(\eta^*\hp)\le-\hp(\{\eta(e_0)\})$.
\end{proof}

\subsection*{Existence of stackings: proof of Lemma~\ref{lemma_stackable_subgroup}}

To prove Lemma~\ref{lemma_stackable_subgroup}, we choose $\nu\colon\Sigma\immerse\Gamma$ such that
$\eta\circ\nu$ decomposes into embeddings and covering maps of normal $\Z$-covers; by an argument due
to Louder and Wilton \cite{louder2014stacking}, $\eta\circ\nu$ is then stackable. In this subsection, we generalize $\Omega_B$ to an arbitrary finite connected
graph, and keep the notation $\FF$, now for the fundamental group of this graph, with a basis $B$.

\begin{definition}
    [{\cite[Definition 3.1]{louder2014stacking}}]
    A $\Z$\textbf{-tower of graphs} of length $k$ is a sequence 
    \[
    \Sigma=
    \Gamma_k\overset{\eta_k}{\immerse}
    \Gamma_{k-1}\overset{\eta_{k-1}}{\immerse}\cdots\overset{\eta_{2}}{\immerse}
    \Gamma_1\overset{\eta_1}{\immerse} \Gamma_0=\Gamma
    \]
    where each $\Gamma_i$ is a finite connected graph and each
    $\eta_i$ is either an embedding in $\Gamma_{i-1}$
    or the composition of an embedding in a normal $\Z$-cover of $\Gamma_{i-1}$ with the
    covering map.
\end{definition}

\begin{proposition}
    \label{prop_Z_tower}
    Let $\eta\colon\Gamma\immerse\Omega$ be a surjective, non-bijective immersion of
    finite connected pointed graphs. 
    Denote $\FF\defeq\pi_1(\Omega)$ and $H\defeq\eta_*\pi_1(\Gamma)\le\FF$.
    Every claim implies the next one:
    \begin{enumerate}[label=\textbf{\textup{(\arabic*)}}, nosep]
        \item \label{item_Z_tower_free_factor} $H\le\FF$ is a free factor.
        \item \label{item_Z_tower_rank} $\rk(\eta_*(H_1(\Gamma)))
        =\rk(H/(H\cap[\FF,\FF]))
        <\rk(\FF)$.
        \item \label{item_Z_tower_factors} $\eta\colon\Gamma\immerse\Omega$ factors as 
        $\Gamma\overset{\eta^1}{\immerse} \Delta \overset{\eta^2}{\hookrightarrow} \tilde{\Omega}
        \overset{p}{\twoheadrightarrow} \Omega,$
        where 
        $\eta^1$ is surjective,
        $\eta^2$ is injective,
        $p$ is a normal $\Z$-cover, 
        and $|V(\Delta)|>|V(\Omega)|$.
    \end{enumerate}
\end{proposition}

\begin{proof}
    \ref{item_Z_tower_free_factor}$\Rightarrow$\ref{item_Z_tower_rank}. Since $\eta$ is not
    bijective, $H\neq\FF$, so
    $\rk(H/(H\cap [\FF,\FF]))\le \rk(H) < \rk(\FF)$.\\
    \ref{item_Z_tower_rank}$\Rightarrow$\ref{item_Z_tower_factors}.
    There is a surjection $\varepsilon\in\Hom(\FF,\Z)$ with $H\le\ker(\varepsilon)$.
    Let $p\colon \tilde{\Omega}\twoheadrightarrow \Omega$ be the normal $\Z$-cover
    corresponding to $\ker(\varepsilon)$.
    Since $H\le \ker(\varepsilon)$, the map $\eta$ lifts to $\tilde{\Omega}$; 
    let $\Delta$ denote its image. Since $p(\Delta)=\eta(\Gamma)=\Omega$, the surjective
    restriction $p\restriction_{\Delta}\colon\Delta\immerse\Omega$ cannot be injective 
    (otherwise we would get $\FF=p_*\pi_1(\Delta)\le \ker(\varepsilon)$),
    so $|V(\Delta)|>|V(\Omega)|$.

    The following diagram illustrates this process.
\[\begin{tikzcd}[ampersand replacement=\&]
	\otimes \& \bullet \&\& {\blue{ \Delta}\subseteq\cdots} \& \otimes \& \bullet \& \bullet \& {\cdots=\tilde\Omega} \\
	{\Gamma=} \& \bullet \& \bullet \& \bullet \&\& \otimes \& {=\Omega}
	\arrow["a", from=1-1, to=1-1, loop, in=60, out=120, distance=5mm]
	\arrow["c", from=1-1, to=1-2]
	\arrow["c", from=1-2, to=2-2]
	\arrow["a", from=1-2, to=2-4]
	\arrow["c", from=1-4, to=1-5]
	\arrow["{\blue a}", color={rgb,255:red,92;green,92;blue,214}, from=1-5, to=1-5, loop, in=60, out=120, distance=5mm]
	\arrow["b", from=1-5, to=1-5, loop, in=240, out=300, distance=5mm]
	\arrow["{\blue c}", color={rgb,255:red,92;green,92;blue,214}, from=1-5, to=1-6]
	\arrow["{\blue a}", color={rgb,255:red,92;green,92;blue,214}, from=1-6, to=1-6, loop, in=60, out=120, distance=5mm]
	\arrow["b", from=1-6, to=1-6, loop, in=240, out=300, distance=5mm]
	\arrow["{\blue c}", color={rgb,255:red,92;green,92;blue,214}, from=1-6, to=1-7]
	\arrow["a", from=1-7, to=1-7, loop, in=60, out=120, distance=5mm]
	\arrow["{\blue b}", color={rgb,255:red,92;green,92;blue,214}, from=1-7, to=1-7, loop, in=240, out=300, distance=5mm]
	\arrow["c", from=1-7, to=1-8]
	\arrow["p", curve={height=-12pt}, dotted, from=1-8, to=2-7]
	\arrow["b", from=2-3, to=2-2]
	\arrow["c", from=2-4, to=2-3]
	\arrow["a", from=2-6, to=2-6, loop, in=195, out=255, distance=5mm]
	\arrow["b", from=2-6, to=2-6, loop, in=240, out=300, distance=5mm]
	\arrow["c", from=2-6, to=2-6, loop, in=285, out=345, distance=5mm]
\end{tikzcd}\]
\end{proof}

\begin{corollary}
    \label{corollary_ff_Z_tower}
    Let $\eta\colon\Gamma\immerse\Omega$ be an immersion of finite connected graphs.
    Suppose that $H\defeq\eta_*\pi_1(\Gamma)\le\FF\defeq\pi_1(\Omega)$ is either a
    free factor of $\FF$,
    or a free factor of a normal subgroup $N\trianglelefteq\FF$ with $\FF/N\cong \Z$.
    Then $\eta$ decomposes as a $\Z$-tower of graphs of length $k\le |V(\Gamma)|$.
\end{corollary}

\begin{proof}
    We may assume that $\eta$ is surjective (otherwise,
    let $\Omega_1=\Img(\eta)$ and apply the proof
    to $\eta\colon\Gamma\to\Omega_1$), 
    so that $|V(\Gamma)|\ge |V(\Omega)|$.
    By \cite[Claim 3.9(3)]{PP15}, if $H\le M\le J\le\FF$ and $H$ is a free factor of $J$, then $H$
    is a free factor of $M$. Hence, if $\eta$ is not surjective, then $H$ is a free factor of
    $\pi_1(\Omega_1)$ or of $N\cap\pi_1(\Omega_1)$, the kernel of a homomorphism
    $\pi_1(\Omega_1)\to\FF/N\cong\Z$. A $\Z$-tower for
    $\eta\colon\Gamma\to\Omega_1$ of length $k\le|V(\Gamma)|-|V(\Omega_1)|$, followed by the
    inclusion $\Omega_1\subseteq\Omega$, is a $\Z$-tower for $\eta\colon\Gamma\to\Omega$ of length
    $k+1\le|V(\Gamma)|$.
    We proceed by induction on $n_{\Omega}\defeq |V(\Gamma)| - |V(\Omega)|$,
    and show $k\le n_{\Omega}$ (for surjective $\eta$).
    If 
    $n_{\Omega}=0$,
    then $H=\FF$ has an empty $\Z$-tower.
    Assume 
    $n_{\Omega}>0$.
    If $H$ is either a free factor of $\FF$
    or a free factor of $\ker(\varepsilon)$ for some surjective
    $\varepsilon\in\Hom(\FF,\Z)$,
    condition~\ref{item_Z_tower_rank} of Proposition~\ref{prop_Z_tower} holds, so
    $\eta$ decomposes as
    $\Gamma\overset{\eta^1}{\immerse} \Delta \overset{\eta^2}{\hookrightarrow} \tilde{\Omega}
        \overset{p}{\twoheadrightarrow} \Omega$,
    with $n_{\Delta}\defeq |V(\Gamma)|-|V(\Delta)| < n_{\Omega}$.
    If $H$ is a free factor of $\ker(\varepsilon)$, we may take $p$ in the proof of
    Proposition~\ref{prop_Z_tower} to correspond to $\ker(\varepsilon)$, so that
    $p_*\pi_1(\Delta)\le\ker(\varepsilon)$.
    By
    \cite[Claim 3.9(3)]{PP15}, $H$ is also a free factor in $p_*\pi_1(\Delta)$.
    By the induction hypothesis, $\eta^1$ is a $\Z$-tower of length at most 
    $n_{\Delta}$, and so $\eta$ is a $\Z$-tower of length at most $n_{\Delta}+1\le n_{\Omega}$.     
\end{proof}

\begin{definition}
    A $\Z$\textbf{-tower of groups} of length $k$ is a sequence 
    \[
    H = J_k \le J_{k-1} \le \cdots \le J_1\le J_0=J
    \]
    where each $J_i$ is a finitely generated free group and is either a free 
    factor of $J_{i-1}$ or a free factor of a normal subgroup $N\trianglelefteq J_{i-1}$ with 
    $J_{i-1}/N\cong \Z$.
\end{definition}

Corollary~\ref{corollary_ff_Z_tower} immediately implies:
\begin{corollary}
    \label{corollary_Z_tower}
    Let $\eta\colon\Sigma\immerse\Gamma$ be an immersion of finite connected graphs.
    Then $\eta$ decomposes as a $\Z$-tower of graphs if and only if the inclusion
    $H\defeq\eta_*\pi_1(\Sigma)\le J\defeq\pi_1(\Gamma)$
    decomposes as a $\Z$-tower of groups.
\end{corollary}

\begin{proof}
    If $\eta$ decomposes as a $\Z$-tower of graphs,
    clearly the inclusion $H\le J$
    decomposes as a $\Z$-tower of groups, of the same length.
    In the other direction, let 
    $%
    H = J_k \le J_{k-1} \le \cdots \le J_1\le J_0=J
    $ %
    be a $\Z$-tower of groups. 
    Let $\Gamma_0\defeq\Gamma$. For $1\le i\le k$, in the covering space of $\Gamma$ corresponding to
    $J_i$, let $\Gamma_i$ be the union of the image of the lift of $\eta$ and all reduced closed
    paths at the base point; this union is finite and connected, with $\pi_1(\Gamma_i)=J_i$. Use
    Corollary~\ref{corollary_ff_Z_tower}
    to decompose the lift $\Sigma\immerse\Gamma_k$ of $\eta$ and each immersion
    $\Gamma_i\to \Gamma_{i-1}$ into a $\Z$-tower of graphs;
    then concatenate the towers.
\end{proof}

Recall that a subgroup $H$ of a free group $J$ is a \textbf{retract} of $J$ if there is a
homomorphism $r\colon J\to H$ whose restriction to $H$ is the identity; such
$r$ is a \textbf{retraction}.
Recall also that $H\le J$ is \textbf{strictly compressed} in $J$ if
$\rk(M)>\rk(H)$ whenever $H\lneqq M\le J$.
For $1\ne w\in J$, the subgroup $\inner{w}$ is strictly compressed in $J$ if and
only if $w$ is not a proper power.

\begin{definition}[{\cite[Theorem 3.16]{miasnikov2007algebraic}}]
    \label{def_algebraic_closure}
    Let $H\le N\le\FF$. There is a greatest subgroup of $N$ that is algebraic over $H$
    (Definition~\ref{def_algebraic_extension}), the \textbf{algebraic closure} of $H$ inside $N$, and
    it is a free factor of $N$.
\end{definition}

\begin{proposition}[Miasnikov--Ventura--Weil, {\cite[Proposition 3.7 and the paragraph before Example 2.5]{miasnikov2007algebraic}}]
    \label{prop_algebraic_extensions}
    Let $H\le\FF$ be a finitely generated subgroup. Then only finitely many subgroups of $\FF$ are
    algebraic over $H$, and each of them is finitely generated.
\end{proposition}

\begin{proposition}
    \label{prop_ff_strictly_compressed_Z_tower_stackable}
    Let $\eta\colon\Sigma\immerse\Gamma$ be an immersion of finite connected
    graphs, and denote $H\defeq\eta_*\pi_1(\Sigma)\le J\defeq\pi_1(\Gamma)$.
    Each statement implies the next one:
    \begin{enumerate}[label=\textbf{\textup{(\arabic*)}}, nosep]
        \item \label{item_ff_free_factor} $H$ is a free factor of $J$;
        \item \label{item_ff_retract} $H$ is a retract of $J$;
        \item \label{item_ff_strictly_compressed} $H$ is strictly compressed in $J$;
        \item \label{item_ff_Z_tower} the inclusion $H\le J$ decomposes as a $\Z$-tower of groups;
        \item \label{item_ff_stackable} $\eta$ is stackable.
    \end{enumerate}
\end{proposition}

\begin{proof}
    For \ref{item_ff_free_factor}$\Rightarrow$\ref{item_ff_retract}, if $J=H*H'$, then the
    homomorphism $J\to H$ that is the identity on $H$ and trivial on $H'$ is a retraction.

    For \ref{item_ff_retract}$\Rightarrow$\ref{item_ff_strictly_compressed}, let
    $r\colon J\to H$ be a retraction, and let $H\le M\le J$ with $\rk(M)\le\rk(H)$.
    As $r(M)=H$, we have $\rk(H)\le\rk(M)$, so $\rk(M)=\rk(H)<\infty$. As free groups of finite
    rank are Hopfian (every surjective endomorphism is injective), and $M\cong H$, the surjection
    $r\restriction_M\colon M\to H$ is an isomorphism.
    For every $g\in M$, we have $r(g)\in H$, so $r(r(g))=r(g)$; as $r\restriction_M$ is
    injective, $g=r(g)\in H$. Hence $M=H$.

    For \ref{item_ff_strictly_compressed}$\Rightarrow$\ref{item_ff_Z_tower}, we define a sequence
    of subgroups $J_n$, satisfying
    \[
        H\le\cdots\le J_2\le J_1\le J_0=J,
    \]
    by the following recursive rule: for $n\ge1$, as long as the image of $H$ in
    $J_{n-1}^{\textup{ab}}$ has rank $\rk(H/(H\cap[J_{n-1},J_{n-1}]))<\rk(J_{n-1})$, there is a
    non-zero homomorphism $\varepsilon_n\colon J_{n-1}\to\Z$ with $H\le\ker(\varepsilon_n)$; we
    define $J_n$ as the algebraic closure of $H$ inside $\ker(\varepsilon_n)$ (recall
    Definition~\ref{def_algebraic_closure}).
    As $J_n\lneqq J_{n-1}$, and, by Proposition~\ref{prop_algebraic_extensions}, only finitely many
    subgroups of $J$ are algebraic over $H$, the recursion stops at some $J_k$, with
    $\rk(H/(H\cap[J_k,J_k]))=\rk(J_k)$. Then
    $\rk(H)\ge\rk(H/(H\cap[J_k,J_k]))=\rk(J_k)$, so $H=J_k$, as $H$ is strictly compressed in
    $J$. By Definition~\ref{def_algebraic_closure} and Proposition~\ref{prop_algebraic_extensions},
    for $1\le i\le k$, the subgroup $J_i$ is a finitely generated free factor of $\ker(\varepsilon_i)$, and
    $\ker(\varepsilon_i)$ is a normal subgroup of $J_{i-1}$ with quotient $\Z$. Hence
    $H=J_k\le J_{k-1}\le\cdots\le J_0=J$ is a $\Z$-tower of groups.

    Louder and Wilton proved the implication
    \ref{item_ff_Z_tower}$\Rightarrow$\ref{item_ff_stackable} for immersed cycles, in the
    proof of \cite[Lemma 3.4]{louder2014stacking}. Their proof applies to every $\eta$, and has the
    following structure. By Corollary~\ref{corollary_Z_tower}, $\eta$ decomposes as a $\Z$-tower
    of graphs, so it is a composition of embeddings and covering maps of normal $\Z$-covers of
    connected graphs. The identity map of $\Sigma$ is stackable. Hence it suffices to show that if a
    map
    $\varphi\colon\Sigma\to\Delta$ has a stacking $\hat{\varphi}$, and $\theta\colon\Delta\to\Delta'$
    is such an embedding or covering map, then $\theta\circ\varphi$ has a stacking.
    If $\theta$ is an embedding, then $(\theta\times\mathrm{id}_{\R})\circ\hat{\varphi}$ is a
    stacking of $\theta\circ\varphi$. If $\theta$ is a covering map, then, by
    \cite[Lemma 3.3]{louder2014stacking}, there is an embedding
    $\theta_{\R}\colon\Delta\times\R\hookrightarrow\Delta'\times\R$ with
    $p_{\Delta'}\circ\theta_{\R}=\theta\circ p_{\Delta}$, and $\theta_{\R}\circ\hat{\varphi}$ is a
    stacking of $\theta\circ\varphi$.
\end{proof}

\begin{observation}
    \label{observe_basis_image_implies_retract}
    Let $\rho\colon G\to L$ be a homomorphism of free groups.
    If $w_i\in G$ for $i\in I$, and the elements $\rho(w_i)$ form a basis of $L$, then
    $\{w_i\}_{i\in I}$ is a basis of a retract of $G$.
\end{observation}

\begin{proof}
    Let $W\defeq\inner{w_i\mid i\in I}$, and let $\sigma\colon L\to W$ be the homomorphism with
    $\sigma(\rho(w_i))=w_i$ for every $i\in I$. Then $\rho\circ\sigma$ is the identity of $L$, so
    $\sigma$ is injective. Moreover, $\sigma$ is surjective, so $\{w_i\}_{i\in I}$ is a basis of $W$.
    Finally, $\sigma\circ\rho\colon G\to W$ is a retraction, as it fixes every $w_i$.
\end{proof}

\begin{proposition}
    \label{prop_lin_indep_implies_retract}
    Let $u,v\in\FF$. If their images in the abelianization
    $\FF/[\FF,\FF]$ are linearly independent, then for every large enough
    $k\in\N$, the set
    $\braces*{[u^{kn},v^{km}]}_{n,m\in\Z\setminus\{0\}}$
    is a basis of a retract of $[\FF,\FF]$.
\end{proposition}
\begin{proof}
    During the proof, we freely use arithmetic on sets: for $S,T\subseteq\R^2$, $\vec{a}\in\R^2$,
    $c\in\R$ and $I\subseteq\R$, denote $S+T\defeq\braces{s+t\mid s\in S,\, t\in T}$,
    $S+\vec{a}\defeq S+\braces{\vec{a}}$, $cS\defeq\braces{cs\mid s\in S}$ and
    $I\vec{a}\defeq\braces{c\vec{a}\mid c\in I}$. By convention, a homotopy between two paths must
    fix their endpoints.
\begin{figure}[t]
    \centering
    \captionsetup{format=plain}
    \begin{minipage}[c]{0.48\textwidth}
        \centering
        \includegraphics[width=\linewidth, trim=45 440 26 212, clip]{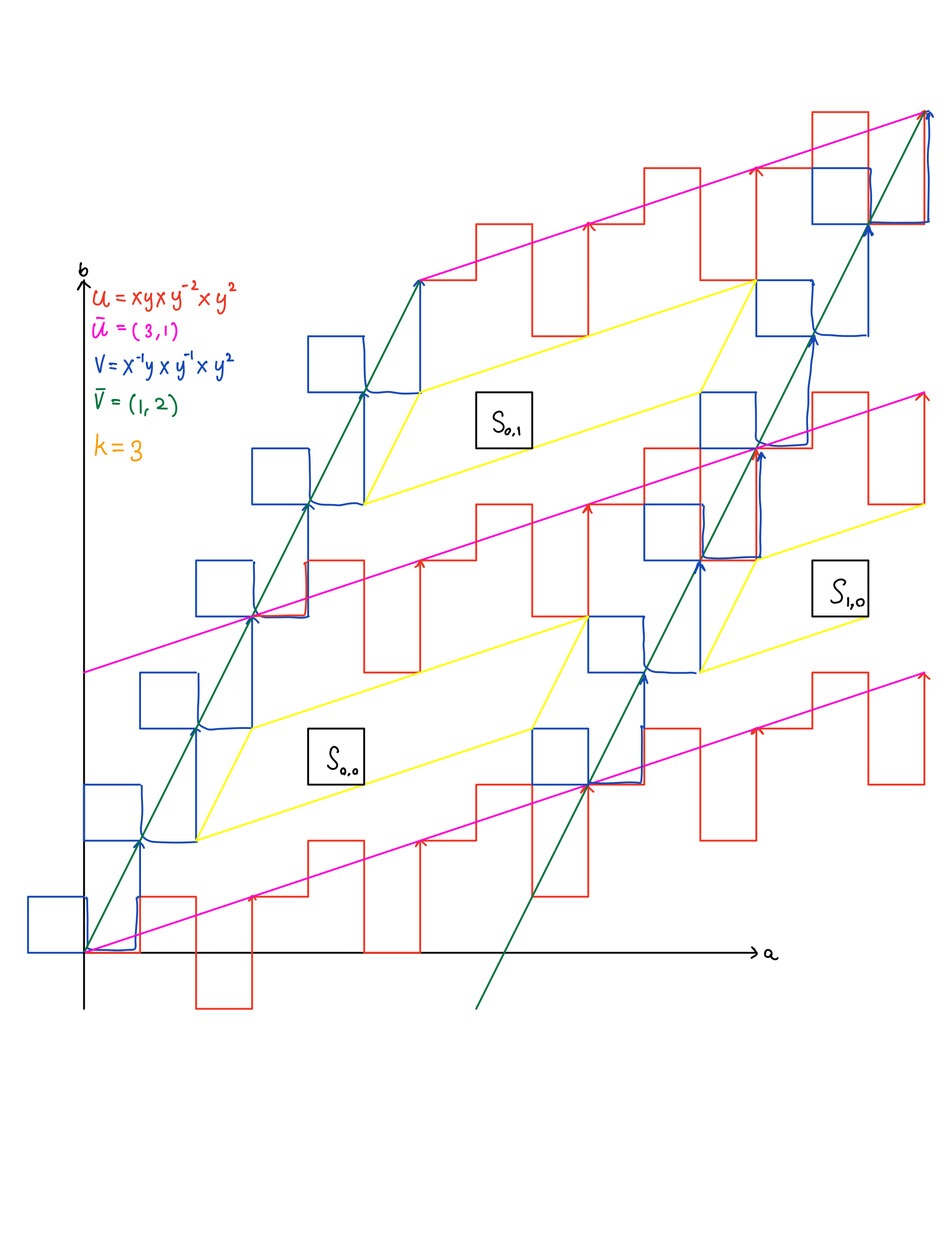}
            \caption{An example for the proof of Proposition~\ref{prop_lin_indep_implies_retract}:
            translates of $P(\gamma_u)$ (red) and of $P(\gamma_v)$ (blue), the graph $\Cay(\Lambda_k)$
            (magenta and green), and the grid cell $S_{0,0}$ containing $z$, with its translates
            $S_{i,j}\defeq S_{0,0}+ik\vec{u}+jk\vec{v}$. Each $S_{i,j}$ is in a yellow parallelogram,
            which the paths do not enter.}
            \label{fig_retract_proof}
    \end{minipage}\hfill
    \begin{minipage}[c]{0.48\textwidth}
        \centering
            \begin{tikzpicture}[x={(2cm,0.54cm)}, y={(0.61cm,1.6cm)}]
                \begin{scope}
                    \clip (-0.4,-0.4) -- (2.4,-0.4) -- (2.4,2.4) -- (-0.4,2.4) -- cycle;
                    \fill[gray!20] (0,0) -- (1,0) -- (1,1) -- (0,1) -- cycle;
                    \foreach \i in {-1,...,3} {
                        \draw[gray] (\i,-1) -- (\i,3);
                        \draw[gray] (-1,\i) -- (3,\i);
                    }
                    \foreach \i in {-1,...,2} {
                        \foreach \j in {-1,...,2} {
                            \fill (\i+0.47,\j+0.47) circle (1pt);
                        }
                    }
                \end{scope}
                \draw[blue, thick, -latex] (0,0) -- (2,0);
                \draw[blue, thick, -latex] (2,0) -- (2,2);
                \draw[blue, thick, -latex] (2,2) -- (0,2);
                \draw[blue, thick, -latex] (0,2) -- (0,0);
                \node[below left, inner sep=1.5pt, fill=white] at (0,0) {$\vec{0}$};
                \fill (0,0) circle (1.4pt);
                \node[below, inner sep=1.5pt] at (0.5,0) {$\zeta$};
                \node[left, inner sep=1.5pt] at (0,0.5) {$\xi$};
                \node[above right, inner sep=0.5pt] at (0.47,0.47) {$z$};
            \end{tikzpicture}
        \medskip
            \caption{The graph $\Cay(\Lambda_k)$, with its edges $[\vec{p},\vec{p}+k\vec{u}]$ and
            $[\vec{p},\vec{p}+k\vec{v}]$ labeled $\zeta$ and $\xi$, the points of $z+\Lambda_k$, and
            $kQ$ (shaded). The blue loop is homotopic in $Y$ to $P(\gamma_{[u^{2k},v^{2k}]})$, and reads
            $[\zeta^2,\xi^2]$.}
            \label{fig_lattice_graph}
    \end{minipage}
\end{figure}

    Let $(e_b)_{b\in B}$ be the standard basis of $\Z^B$, and for $w\in\FF$ denote by
    $\overline{w}$ the image of $w$ under the abelianization $\FF\to\Z^B$, $b\mapsto e_b$. Since
    $\overline{u}$ and $\overline{v}$ are linearly independent, there are $b_1,b_2\in B$ such that,
    for the projection $P\colon\R^B\to\R^2$ to the coordinates $b_1$ and $b_2$, the vectors
    $\vec{u}\defeq P(\overline{u})$ and $\vec{v}\defeq P(\overline{v})$ are linearly independent.

    Denote $\Cay(\Z^B)\defeq\bigcup_{\vec{a}\in\Z^B,\,b\in B}[\vec{a},\vec{a}+e_b]$. It is the
    Cayley graph of $\Z^B$ with respect to the basis $(e_b)_{b\in B}$: its vertex set is
    $\Z^B$, and the segment $[\vec{a},\vec{a}+e_b]$ is an edge from $\vec{a}$ to $\vec{a}+e_b$ with label $b$.
    Every $w\in\FF$ can be read from $\vec{0}$ along a unique path $\gamma_w$ in $\Cay(\Z^B)$, which
    ends at $\overline{w}$. Thus, as for Stallings core graphs, we identify
    $\pi_1(\Cay(\Z^B),\vec{0})$ with $\braces{w\in\FF\mid\overline{w}=\vec{0}}=[\FF,\FF]$. The
    image $P(\Cay(\Z^B))$ is the grid $(\R\times\Z)\cup(\Z\times\R)$.

    Choose homotopies in $\R^2$ from $P(\gamma_u)$ to the segment $[\vec{0},\vec{u}]$ and from
    $P(\gamma_v)$ to $[\vec{0},\vec{v}]$. Choose $C>0$ such that the image of each homotopy is
    contained in the $C$-neighborhood of its target segment.
    For $j\in\Z$, a concatenation of $|j|$ translates of $\gamma_u$ (of its inverse if $j<0$) is
    homotopic in $\Cay(\Z^B)$ to $\gamma_{u^j}$, by a homotopy that removes backtracks and stays
    inside the image of the concatenation. Hence, by translates of the first homotopy,
    $P(\gamma_{u^j})$ is homotopic to $[\vec{0},j\vec{u}]$ inside the $C$-neighborhood of
    $[\vec{0},j\vec{u}]$. The same holds for $v$, with the second homotopy.

    Let $Q\defeq[0,1]\vec{u}+[0,1]\vec{v}$ denote the parallelogram spanned by $\vec{u}$ and
    $\vec{v}$, and denote its boundary by $\partial Q$. As the center
    $\frac{k}{2}(\vec{u}+\vec{v})$ of $kQ$ satisfies
    $\operatorname{dist}(\frac{k}{2}(\vec{u}+\vec{v}),\partial(kQ))
    =k\operatorname{dist}(\frac{1}{2}(\vec{u}+\vec{v}),\partial Q)$, and $P(\Cay(\Z^B))$ has empty
    interior (as $\Cay(\Z^B)$ is a graph), for every large enough $k\in\N$ there is
    $z\in kQ\setminus P(\Cay(\Z^B))$ with $\operatorname{dist}(z,\partial(kQ))>C$. Fix such $k$
    and $z$, and let $\Lambda_k\defeq\Z k\vec{u}+\Z k\vec{v}$ and
    $Y\defeq\R^2\setminus(z+\Lambda_k)$. Since $\Lambda_k\subseteq\Z^2$ and
    $P(\Cay(\Z^B))+\Z^2=P(\Cay(\Z^B))$, the set $z+\Lambda_k$ is disjoint from $P(\Cay(\Z^B))$.
    Therefore $P$ induces a homomorphism
    $\rho\colon[\FF,\FF]=\pi_1(\Cay(\Z^B),\vec{0})\to\pi_1(Y,\vec{0})$.

    Denote $\Cay(\Lambda_k)\defeq\Lambda_k+\prn[\big]{[0,k]\vec{u}\cup[0,k]\vec{v}}$. It is the
    Cayley graph of $\Lambda_k$ with respect to the basis $(k\vec{u},k\vec{v})$. The parallelograms
    $kQ+\vec{p}$, for $\vec{p}\in\Lambda_k$, tile $\R^2$, and $\Cay(\Lambda_k)=\Lambda_k+\partial(kQ)$
    (see Figures~\ref{fig_retract_proof} and~\ref{fig_lattice_graph}). Each point $z+\vec{p}$ is in the interior of
    $kQ+\vec{p}$, and $\operatorname{dist}(z+\vec{p},\partial(kQ+\vec{p}))>C$. As $\Cay(\Lambda_k)$
    does not meet this interior, the $C$-neighborhood of $\Cay(\Lambda_k)$ is contained in $Y$.
    Let $n,m\in\Z\setminus\{0\}$. A concatenation of translates of $\gamma_{u^{kn}}$,
    $\gamma_{v^{km}}$, $\gamma_{u^{-kn}}$ and $\gamma_{v^{-km}}$, in this order, is homotopic in
    $\Cay(\Z^B)$ to $\gamma_{[u^{kn},v^{km}]}$, by a homotopy that stays inside the image of the
    concatenation. Hence, by translates of the homotopies of $P(\gamma_{u^{\pm kn}})$ and
    $P(\gamma_{v^{\pm km}})$ above, $P(\gamma_{[u^{kn},v^{km}]})$ is homotopic in $Y$ to the loop in
    $\Cay(\Lambda_k)$ around the parallelogram with vertices $\vec{0}$, $kn\vec{u}$,
    $kn\vec{u}+km\vec{v}$ and $km\vec{v}$, in this order.

    The radial projections from the points $z+\vec{p}$ onto $\partial(kQ+\vec{p})$, for
    $\vec{p}\in\Lambda_k$, form a deformation retraction of $Y$ onto $\Cay(\Lambda_k)$. Let $F_2$ be
    the free group with basis $\{\zeta,\xi\}$, and label the edges $[\vec{p},\vec{p}+k\vec{u}]$ and
    $[\vec{p},\vec{p}+k\vec{v}]$ of $\Cay(\Lambda_k)$ by $\zeta$ and $\xi$. As for $\Cay(\Z^B)$, we
    identify $\pi_1(\Cay(\Lambda_k),\vec{0})$, and hence $\pi_1(Y,\vec{0})$, with $[F_2,F_2]$ by
    reading labels. The loop around the parallelogram above reads $[\zeta^n,\xi^m]$, so
    \[
        \rho([u^{kn},v^{km}])=[\zeta^n,\xi^m]
        \qquad
        \text{for all }n,m\in\Z\setminus\{0\}.
    \]
    By \cite[Section~1, Rank~2, item (iii)]{putman2022commutator}, the elements
    $[\zeta^n,\xi^m]$, for $n,m\in\Z\setminus\{0\}$, form a basis of $[F_2,F_2]$.\footnote{
        In \cite{putman2022commutator}, the commutator of $g$ and $h$ is
        $g^{-1}h^{-1}gh=[g^{-1},h^{-1}]$, so the basis there is
        $\braces*{[\zeta^{-n},\xi^{-m}]}_{n,m\in\Z\setminus\{0\}}$.
    }
    The proposition now follows from Observation~\ref{observe_basis_image_implies_retract}.
\end{proof}

\begin{observation}
    \label{obs_retract_of_commutator}
    Let $K$ be a finitely generated retract of $[\FF,\FF]$. Then the inclusion $K\le\FF$
    decomposes as a $\Z$-tower of groups.
    Indeed, let $\varepsilon_1,\ldots,\varepsilon_\ell$ be a basis of $\Hom(\FF,\Z)$, and let
    $J_0\defeq\FF$. For $1\le i\le\ell$, choose a finitely generated free factor $J_i$ of
    $J_{i-1}\cap\ker(\varepsilon_i)$ that contains $K$. Then $J_\ell\le\cdots\le J_0=\FF$ is a
    $\Z$-tower of groups, and $K\le J_\ell\le\bigcap_{i=1}^{\ell}\ker(\varepsilon_i)=[\FF,\FF]$, so
    $K$ is a retract of $J_\ell$. By the implications
    \ref{item_ff_retract}$\Rightarrow$\ref{item_ff_strictly_compressed}$\Rightarrow$\ref{item_ff_Z_tower}
    of Proposition~\ref{prop_ff_strictly_compressed_Z_tower_stackable}, the inclusion
    $K\le J_\ell$ decomposes as a $\Z$-tower of groups, and therefore so does $K\le\FF$.
\end{observation}

\begin{theorem}
\label{thm_exist_subgrp_Z_tower}
    For every \noncyc $H\le\FF$ and $s\in\N$, there is a subgroup $K\le H$ of rank $s$ such that
    the inclusion $K\le\FF$ decomposes as a $\Z$-tower of groups.
\end{theorem}

\begin{proof}
    We may assume that $H$ is finitely generated. Then $H=\eta_*\pi_1(\Gamma)$ for an immersion
    $\eta\colon\Gamma\immerse\Omega$ of finite connected graphs with $\pi_1(\Omega)=\FF$. We induct on $|V(\Gamma)|-|V(\eta(\Gamma))|$.
    As $\eta(\Gamma)$ is a connected subgraph of $\Omega$, the group $\pi_1(\eta(\Gamma))$ is a free
    factor of $\FF$. Hence, for every $K\le H$, if the inclusion $K\le\pi_1(\eta(\Gamma))$
    decomposes as a $\Z$-tower of groups, then so does the inclusion $K\le\FF$.
    Therefore we may assume that $\eta$ is surjective.

    Assume first that some $u,v\in H$ have linearly independent images in the abelianization
    $\FF/[\FF,\FF]$. By Proposition~\ref{prop_lin_indep_implies_retract}, for some $k\in\N$ the set
    $\braces{[u^{kn},v^{km}]}_{n,m\in\Z\setminus\{0\}}$ is a basis of a retract of $[\FF,\FF]$.
    Choose any $s$ of these elements, and let $K\le H$ be the subgroup they generate. Then $K$ is a
    free factor of that retract, and is therefore a retract of $[\FF,\FF]$. By
    Observation~\ref{obs_retract_of_commutator}, the inclusion $K\le\FF$ decomposes as a
    $\Z$-tower of groups.

    Otherwise, $\rk(H/(H\cap[\FF,\FF]))\le1$, and $\rk(\FF)\ge2$ as $H\le\FF$ is \noncyc. Hence
    $\eta$ is not bijective, and, by the implication
    \ref{item_Z_tower_rank}$\Rightarrow$\ref{item_Z_tower_factors} of Proposition~\ref{prop_Z_tower},
    $\eta$ factors as
    $\Gamma\overset{\eta^1}{\immerse} \Delta \overset{\eta^2}{\hookrightarrow} \tilde{\Omega}
    \overset{p}{\twoheadrightarrow} \Omega$
    with $\eta^1$ surjective, $p$ a normal $\Z$-cover and $|V(\Delta)|>|V(\Omega)|$.
    By the induction hypothesis applied to $\eta^1$, there is $K'\le\eta^1_*\pi_1(\Gamma)$ of rank
    $s$ such that the inclusion $K'\le\pi_1(\Delta)$ decomposes as a $\Z$-tower of groups. As
    $(p\circ\eta^2)_*$ is injective, $K\defeq(p\circ\eta^2)_*(K')$ is a subgroup of $H$ of rank $s$,
    and the inclusion $K\le(p\circ\eta^2)_*\pi_1(\Delta)$ decomposes as a $\Z$-tower of groups. As
    $(p\circ\eta^2)_*\pi_1(\Delta)$ is a free factor of $p_*\pi_1(\tilde\Omega)$, a normal subgroup of
    $\FF$ with quotient $\Z$, the inclusion $K\le\FF$ decomposes as a $\Z$-tower of groups too.
\end{proof}

\begin{proof}[Proof of Lemma~\ref{lemma_stackable_subgroup}]
    The subgroup $H\defeq\eta_*\pi_1(\Gamma)\le\FF\defeq\pi_1(\Omega)$ is \noncyc, as $\eta_*$ is
    injective and $\chi(\Gamma)<0$. By Theorem~\ref{thm_exist_subgrp_Z_tower} with $s=2$, there is
    $K\le H$ of rank $2$ such that the inclusion $K\le\FF$ decomposes as a $\Z$-tower of groups.
    Let $\Sigma$ be the union of all reduced closed paths at the base point of the covering space
    of $\Gamma$ corresponding to $\eta_*^{-1}(K)$, with the induced immersion
    $\nu\colon\Sigma\immerse\Gamma$. Then $\Sigma$ is finite and connected, $\chi(\Sigma)=-1$, and
    $(\eta\circ\nu)_*\pi_1(\Sigma)=K$. Hence, by the implication
    \ref{item_ff_Z_tower}$\Rightarrow$\ref{item_ff_stackable} of
    Proposition~\ref{prop_ff_strictly_compressed_Z_tower_stackable}, $\eta\circ\nu$ is stackable.
\end{proof}

%% file: fixed_points.tex
\section{Counting Fixed Points}
\label{section_counting_fixed_sub_spaces}

In this section we prove Theorems~\ref{thm_Sn_spibard} and~\ref{thm_spibarqd} on exact decay
rates.

By the classification of finite simple groups, every finite simple group of large rank
is an alternating group $A_n$, a projective special linear group $\PSL_n(\F_q)$,
or a classical group of automorphisms that preserve a bilinear, sesquilinear or quadratic form.
Our goal here is to count, given a random homomorphism $\alpha$ from $\FF$ to such a finite simple group $G$,
the expected number of common fixed points of $\alpha(H)$, for $H\le \FF$, when $G$ acts on a relatively small set.
We treat $S_n$ and $\GLnFq$ in full
(Theorems~\ref{thm_Sn_spibard} and~\ref{thm_spibarqd}), and indicate how one could extend the
proofs to the remaining families.
The general part of the proofs is written for a category $\catC$ with the following property, shared by finite
sets, by finite $\F_q$-spaces and, by Witt's theorem,\footnote{
    \cite[Satz 4]{Witt1937QuadratischeFormen}; see
    \cite[Theorem 7.4]{Taylor1992GeometryClassicalGroups}, and
    \cite[Theorem 3.4]{SprehnWahl2020FormsOverFieldsWitt} for the
    characteristic $2$ case.
} by finite $\F_q$-spaces equipped with such a non-degenerate form:
\begin{equation}\label{stmt_transitive_aut}%
\begin{minipage}{0.9\linewidth}
\sffamily\small
For every two objects $X, Y\in \catC$, the group $\Aut_{\catC}(Y)$ acts
transitively by composition on the set 
$\HomCinj(X, Y)$
of injective morphisms.
Equivalently, $\Aut_{\catC}(Y)$ acts transitively on isomorphic sub-objects of $Y$.
\end{minipage}
\end{equation}

We therefore fix a category $\catC$ with this property, which covers
all of these families at once.
The corresponding passage from the linear groups to the form-preserving
ones is worked out in detail by Eberhard and Jezernik
\cite[Sections 2.3 and 2.4]{eberhard2022babai}; here, it would also involve the values of the form
(see also Remark~\ref{remark_why_q_analog}).
Descending from such a group to the simple group it determines, by passing to
the derived subgroup and dividing by its center, as in
$S_n\rightsquigarrow A_n$ or $\GLnFq\rightsquigarrow\PSL_n(\F_q)$, is a
technical matter of a different kind, which we do not treat here.

\begin{definition}
\label{def_inter_alpha_beta}
    Let $H$ be a group, let $X, Y$ be objects of $\catC$, and let
    $\alpha\in \Hom(H, \Aut(Y))$ and $\beta\in \Hom(H, \Aut(X))$.
    Let $\Inter(\alpha, \beta)$ be the set of morphisms
    $X\to Y$ intertwining $\alpha$ and $\beta$:
    \begin{equation*}
        \Inter(\alpha, \beta) \defeq \braces*{\iota\in \Hom_{\catC}(X, Y)
        \,\middle|\,
        \alpha(h)\circ\iota = \iota\circ\beta(h)\ \ \forall h\in H},
    \end{equation*}
    that is, the morphisms $\iota$ making the square
    \begin{equation*}
        \begin{tikzcd}[ampersand replacement=\&]
	{X} \& {X} \\
	{Y} \& {Y}
	\arrow["\iota", from=1-1, to=2-1]
	\arrow["{\beta(h)}"', from=1-1, to=1-2]
	\arrow["\iota"', from=1-2, to=2-2]
	\arrow["{\alpha(h)}", from=2-1, to=2-2]
\end{tikzcd}
    \end{equation*}
    commute for every $h\in H$.
    We also define 
    $\Interinj(\alpha, \beta)\defeq 
    \HomCinj(X, Y)
    \cap \Inter(\alpha, \beta).$
\end{definition} 

We denote the image of a function $f$ by $\Img(f)$; if $f$ is a morphism in $\catC$, then
$\Img(f)\in \textup{Obj}(\catC)$.

\begin{observation}
\label{observe_common_fixed_mod_G_and_Inter}
    Let $H$ be a group, 
    $X, Y\in \textup{Obj}(\catC)$,
    $\alpha\in \Hom(H, \Aut(Y))$ and 
    $\iota\in \Hom^{\textup{inj}}_{\catC}(X, Y)$.
    Then $\Img(\iota)\subseteq Y$ is $\alpha(H)$-invariant if and only if there exists 
    $\beta\in \Hom(H, \Aut(X))$
    such that 
    $\iota\in \textup{Inter}(\alpha, \beta)$.
    If such $\beta$ exists, it is unique: $\beta(h) = \iota^{-1}\restriction_{\Img(\iota)} \circ \alpha(h) \circ \iota $ for every $h\in H$.
\end{observation}

\begin{proposition}
\label{prop_common_fixed_mod_G_and_Com}
    For a group $H$,
    objects $X, Y$ of $\catC$, a subgroup
    $G\le \Aut(X)$ and a homomorphism $\alpha\in \Hom(H, \Aut(Y))$, we have
    \[ \abs*{\braces*{\begin{array}{cc}
         \textup{common fixed points} \\
         \textup{of }\alpha(H)\acts \HomCinj(X, Y) /G
    \end{array}}} = \frac{1}{|G|} \sum_{\beta\in \Hom(H, G)} \abs*{\Interinj(\alpha, \beta)}. \]
\end{proposition}

For example, if $\catC$ is the category of finite sets, $X = [d]$ and $Y = [n]$
(so $\Aut(X) = S_d$, $\Aut(Y) = S_n$, and $\HomCinj(X, Y) = [n]_d$),
then Proposition~\ref{prop_common_fixed_mod_G_and_Com} gives
\begin{align*}
    \abs*{\braces*{\begin{array}{cc}
         \textup{common fixed points} \\
         \textup{of }\alpha(H)\acts [n]_d
    \end{array}}} &= \abs*{\Interinj(\alpha, \mathbf{1})}, \\[0.5em]
    \abs*{\braces*{\begin{array}{cc}
         \textup{common fixed points} \\
         \textup{of }\alpha(H)\acts \binom{[n]}{d}
    \end{array}}} &= \frac{1}{d!} \sum_{\beta\in \Hom(H, S_d)} \abs*{\Interinj(\alpha, \beta)}.
\end{align*}

\begin{proof}
    Define $\mathfrak{I}\defeq \biguplus_{\beta\in \Hom(H, G)} \Interinj(\alpha, \beta) \times \{\beta\}$ and denote by $\mathcal{CFP}$ the set of common fixed points of $\alpha(H)\acts \HomCinj(X, Y)/G$.
    The group $G$ acts freely on $\mathfrak{I}$: for every $g\in G$ and $(\iota, \beta)\in \mathfrak{I}$, let $(\iota, \beta).g \defeq (\iota\circ g^{-1}, \,\, h\mapsto g^{-1}\beta(h) g)$.
    Indeed, $\iota\circ g^{-1} = \iota$ implies $g = \textup{id}_X$ since $\iota$ is injective.
    For every $h\in H$ and $(\iota, \beta)\in \mathfrak{I}$, we have equality of $G$-orbits 
    $\alpha(h)\circ\iota\circ G = \iota\circ\beta(h)\circ G = \iota\circ G$
    so $\iota\circ G\in \mathcal{CFP}$.
    With $\iota.g\defeq\iota\circ g^{-1}$, the projection map $\mathfrak{I}\to \Hom^{\textup{inj}}_{\catC}(X, Y)$ defined by $(\iota, \beta)\mapsto \iota$ intertwines the $G$-actions on $\mathfrak{I}$ and on $\Hom^{\textup{inj}}_{\catC}(X, Y)$, so it induces a map on the orbits $\phi\colon \mathfrak{I}/G\to \mathcal{CFP}$.    
    Now by Observation~\ref{observe_common_fixed_mod_G_and_Inter}, $\phi$ is bijective. 
    Since the action $G\acts \mathfrak{I}$ is free we get $|\mathcal{CFP}| = |\mathfrak{I}| / |G|$ as needed.
\end{proof}

\subsection*{$S_n$ and Covering Spaces}

The following proposition, which we use only for graphs, is a known fact from algebraic topology:
degree-$d$ coverings of a pointed CW-complex correspond to actions of $\pi_1$ on the
fiber over the basepoint, see e.g.\ \cite[\S1.3]{hatcher2002algebraic}.
\begin{proposition}
\label{prop_alg_top}
    Let $(X, x_0)$ be a pointed CW-complex, and let $d\in \N$.
    There is a bijective correspondence between numbered coverings and actions on $[d]$:
    \begin{equation*}
        \braces*{ (\Tilde{X}, p, c): \begin{array}{cc}
         & p\colon \Tilde{X} \twoheadrightarrow X \textup{ is a covering map of degree }d, \\
         & \textup{ and } c\colon [d]\overset{\cong}{\to}  p^{-1}(x_0) \textup{ is a bijective numbering. }
    \end{array} } \Longleftrightarrow \Hom(\pi_1(X, x_0), S_d).
    \end{equation*}
Moreover, $S_d$ acts on both sets: every $\tau\in S_d$ acts on numbered coverings by 
$(\Tilde{X}, p, c).\tau \defeq (\Tilde{X}, p, c\circ\tau^{-1})$ and on homomorphisms
$\beta\in \Hom(\pi_1(X, x_0), S_d)$ by $(\beta.\tau)(\gamma) \defeq \tau^{-1}\beta(\gamma)\tau$, 
and the correspondence intertwines the actions.
Given a numbered covering $(\Tilde{X}, p, c)$ that corresponds to $\beta\in \Hom(\pi_1(X, x_0), S_d)$, the group $\pi_1(X, x_0)$ acts on $p^{-1}(x_0)$: $\gamma\in \pi_1(X, x_0)$ maps $\Tilde{x}\in p^{-1}(x_0)$ to the end point $\Tilde{x}.\gamma$ of the unique lift of $\gamma$ to $\Tilde{X}$ starting at $\Tilde{x}$, and the following diagram commutes:
\[\begin{tikzcd}
	{[d]} & {p^{-1}(x_0)} \\
	{[d]} & {p^{-1}(x_0)}
	\arrow["c"', from=1-1, to=1-2]
	\arrow["\gamma"', from=1-2, to=2-2]
	\arrow["{\beta(\gamma)}"', from=1-1, to=2-1]
	\arrow["c", from=2-1, to=2-2]
\end{tikzcd}\]
\end{proposition}

\begin{example}
Let $B=\{x,y\}$ and $H = \inner*{xyx, yx^2}\le \FF$. There are $3$ transitive actions $H\acts \{1, 2\}$, corresponding to the following covering spaces of the core graph $\Gamma_B(H)$:
\[\begin{tikzcd}[row sep=small]
	\otimes & \bullet & \bullet & \bullet &&& \bullet & \bullet \\
	\bullet & \bullet & \otimes & \bullet & \bullet & \otimes & \bullet & \otimes \\
	\bullet & \bullet &&& \bullet & \bullet & \otimes & \bullet \\
	\bullet & \otimes && \otimes & \bullet && \bullet & \bullet \\
	&&& \bullet & \bullet
	\arrow["x", from=4-4, to=4-5]
	\arrow["y", from=4-5, to=5-5]
	\arrow["x", from=5-4, to=5-5]
	\arrow["y"', from=4-4, to=5-4]
	\arrow["x"{description}, from=5-5, to=4-4]
	\arrow["x"', from=1-1, to=1-2]
	\arrow["y", from=1-2, to=2-2]
	\arrow["x"{description}, from=2-2, to=1-1]
	\arrow["y", from=1-1, to=2-1]
	\arrow["x"{description}, from=1-7, to=2-8]
	\arrow["y", from=2-7, to=1-7]
	\arrow["y", from=2-8, to=1-8]
	\arrow["x", from=1-8, to=1-7]
	\arrow["y"', from=4-2, to=3-2]
	\arrow["x"{description}, from=3-1, to=4-2]
	\arrow["x", from=4-2, to=4-1]
	\arrow["y"', from=4-1, to=3-1]
	\arrow["x", from=2-1, to=3-1]
	\arrow["x"', from=3-2, to=2-2]
	\arrow["y"', from=3-8, to=4-8]
	\arrow["x"{description}, from=4-8, to=3-7]
	\arrow["y"', from=3-7, to=4-7]
	\arrow["x"', from=4-7, to=4-8]
	\arrow["x"', from=2-8, to=3-8]
	\arrow["x", from=3-7, to=2-7]
	\arrow["x"', from=1-3, to=1-4]
	\arrow["y", from=2-3, to=1-3]
	\arrow["x", from=1-4, to=2-6]
	\arrow["y", from=2-6, to=3-6]
	\arrow["x", from=3-6, to=3-5]
	\arrow["x", from=3-5, to=2-3]
	\arrow["x"{description}, from=2-3, to=2-4]
	\arrow["y"', from=2-4, to=1-4]
	\arrow["x"{description}, from=2-6, to=2-5]
	\arrow["y"', from=2-5, to=3-5]
\end{tikzcd}\]
\end{example}

Recall that for a $B$-graph $\Gamma$ and $\alpha\in\Hom(\FF,S_n)$, a function
$f\colon V(\Gamma)\to[n]$ is $\alpha$-valid if it satisfies the validity constraints
\eqref{eq_valid_func} (see also \cite[Definition A.7]{Sho23I});
equivalently, for every $v, u\in V(\Gamma)$ and every path
$\gamma\colon v\rightsquigarrow u$ that \enquote{reads} a word $w\in \FF$, we
have $f(u) = f(v).\alpha(w)$.

\begin{figure}[ht]
    \centering

\begin{tikzcd}[ampersand replacement=\&]
	{\tilde{\Gamma}=} \& 3 \& 5 \& 6 \& 4 \&\& {\{2,4\}} \& {\{1,3\}} \& {=\Gamma} \\
	\& 2 \& 2 \& 3 \& 1 \&\& {\{2,6\}} \& {\{3,5\}} \\
	\&\& 1 \\
	\& 2 \& 3 \& 4 \\
	{\Delta=} \&\& 5 \& 6 \&\&\&\& {\{1,\ldots,6\}} \& {=\Omega_B}
	\arrow[curve={height=-18pt}, dotted, from=1-1, to=1-9]
	\arrow[curve={height=6pt}, dotted, from=1-1, to=5-1]
	\arrow["y", color={rgb,255:red,51;green,61;blue,255}, from=1-2, to=1-3]
	\arrow["x"{description}, color={rgb,255:red,255;green,51;blue,54}, from=1-3, to=2-2]
	\arrow["x"', color={rgb,255:red,255;green,51;blue,54}, from=1-4, to=1-3]
	\arrow["y"', color={rgb,255:red,51;green,61;blue,255}, from=1-5, to=1-4]
	\arrow["x", color={rgb,255:red,255;green,51;blue,54}, from=1-5, to=2-5]
	\arrow["x", color={rgb,255:red,255;green,51;blue,54}, from=1-7, to=1-8]
	\arrow["y"', color={rgb,255:red,51;green,61;blue,255}, from=1-7, to=2-7]
	\arrow["y", color={rgb,255:red,51;green,61;blue,255}, from=1-8, to=2-8]
	\arrow[curve={height=-6pt}, dotted, from=1-9, to=5-9]
	\arrow["x"', color={rgb,255:red,255;green,51;blue,54}, from=2-2, to=1-2]
	\arrow["y", color={rgb,255:red,51;green,61;blue,255}, from=2-2, to=2-3]
	\arrow["x", color={rgb,255:red,255;green,51;blue,54}, from=2-3, to=2-4]
	\arrow["x"{description}, color={rgb,255:red,255;green,51;blue,54}, from=2-4, to=1-5]
	\arrow["y"', color={rgb,255:red,51;green,61;blue,255}, from=2-5, to=2-4]
	\arrow["x", color={rgb,255:red,255;green,51;blue,54}, from=2-7, to=2-8]
	\arrow["x"{description}, color={rgb,255:red,255;green,51;blue,54}, from=2-8, to=1-7]
	\arrow["y"', color={rgb,255:red,51;green,61;blue,255}, from=3-3, to=4-3]
	\arrow["y", color={rgb,255:red,51;green,61;blue,255}, from=4-2, to=4-2, loop, in=55, out=125, distance=10mm]
	\arrow["x", color={rgb,255:red,255;green,51;blue,54}, from=4-2, to=4-3]
	\arrow["x"', color={rgb,255:red,255;green,51;blue,54}, from=4-3, to=4-4]
	\arrow["y", color={rgb,255:red,51;green,61;blue,255}, from=4-3, to=5-3]
	\arrow["x"', color={rgb,255:red,255;green,51;blue,54}, from=4-4, to=3-3]
	\arrow["y", color={rgb,255:red,51;green,61;blue,255}, from=4-4, to=5-4]
	\arrow[curve={height=18pt}, dotted, from=5-1, to=5-9]
	\arrow["x", color={rgb,255:red,255;green,51;blue,54}, from=5-3, to=4-2]
	\arrow["x", color={rgb,255:red,255;green,51;blue,54}, from=5-4, to=5-3]
	\arrow["x", color={rgb,255:red,255;green,51;blue,54}, from=5-8, to=5-8, loop, in=55, out=125, distance=10mm]
	\arrow["y", color={rgb,255:red,51;green,61;blue,255}, from=5-8, to=5-8, loop, in=100, out=170, distance=10mm]
\end{tikzcd}
    \caption{A system of equations
    on $\Gamma$ over the action 
    $S_6\acts \binom{\{1\ldots 6\}}{2}$.}
    \label{fig_Sn_pairs}
\end{figure}

Recall that we fix a basis $B\subseteq \FF$,
and recall the notation $(n)_t\defeq n\cdot (n-1)\cdots (n-t+1)$. 

\begin{proposition}
    [{\enquote{Basis dependent Möbius inversions}, \cite[Proposition 6.8]{HP22}}]
\label{prop_LB}
Let $\eta\colon \Gamma\to \Delta$ 
be a surjective morphism
of core graphs.
For every large enough $n$, 
let $L_{\eta}^B(n)$ 
be the average number of injective lifts 
of $\eta$ from $\Gamma$ to a random $n$-cover of $\Delta$. 
Then 
\[ L_{\eta}^B(n) = \frac{\prod_{v\in V(\Delta)}(n)_{|\eta^{-1}(v)|}}
{\prod_{e\in E(\Delta)}(n)_{|\eta^{-1}(e)|}} = n^{\chi(\Gamma)} \cdot (1 + O\prn*{n^{-1}}). \]
 Moreover, $L_{\eta}^B(n)$ is multiplicative with respect to the connected components of $\Delta$.
\end{proposition}

For a $B$-graph $\Gamma$, we denote by $\Gamma\to \Omega_B$ the unique
morphism from $\Gamma$ to the bouquet.
One can easily show (see e.g.\ \cite[Appendix A]{Sho23I}) that 
$L_{\Gamma\to\Omega_B}^B(n)$ is the average number of injective 
$\alpha$-valid functions $V(\Gamma)\to [n]$ where $\alphaSimUHomFSn$.
\begin{definition}[{Surjective decompositions, \cite[Definition 6.4]{HP22}}]
\label{def_decompos_B}
Let $\eta\colon\Gamma\to\Delta$ be a surjective morphism of $B$-graphs. Define $\DecompB^2(\eta)$ as
the set of pairs $(\eta_1,\eta_2)$ of surjective morphisms with $\eta_1\colon\Gamma\twoheadrightarrow\textup{Im}(\eta_1)$
and $\eta_2\circ\eta_1=\eta$, modulo the
following equivalence relation: $(\eta_1, \eta_2) \sim (\eta_1', \eta_2') $ whenever there is an isomorphism $\theta \colon \textup{Im}(\eta_1) \to \textup{Im}(\eta_1')$ such that the following diagram commutes:
\[\begin{tikzcd}
	\Gamma & \textup{Im}(\eta_1) \\
	& {\textup{Im}(\eta_1')} & \Delta
	\arrow["{\eta_1}", two heads, from=1-1, to=1-2]
	\arrow["{\eta_1'}"', two heads, from=1-1, to=2-2]
	\arrow["\cong", from=1-2, to=2-2]
	\arrow["{\eta_2}", two heads, from=1-2, to=2-3]
	\arrow["{\eta_2'}", two heads, from=2-2, to=2-3]
\end{tikzcd}\]
\end{definition}

\begin{definition}
\label{def_p_efficient}
    Let $\Gamma$ be a connected core graph with a basepoint $v_0\in V(\Gamma)$,
    let $p\colon \Tilde{\Gamma}\to \Gamma$ be a covering,
    and let $\eta\colon \Tilde{\Gamma}\to \Delta$ be a morphism of $B$-graphs.
    We say that $\eta$ is $p$\textbf{-efficient} if $\eta\restriction_{p^{-1}(v_0)}$ is injective.
\end{definition}

\begin{remark}
\label{rem_p_efficient_is_fiberwise}
    A $p$-efficient $\eta$ is injective on \emph{every} fiber $p^{-1}(v)$, not
    only over $v_0$: since $\Gamma$ is connected, there is a path $v\rightsquigarrow v_0$; if
    $\eta(u_1)=\eta(u_2)$ for distinct $u_1,u_2\in p^{-1}(v)$, the lifts of this path at $u_1$ and
    $u_2$ would end at distinct vertices over $v_0$ with a common $\eta$-image, since a starting
    vertex and a word determine at most one path in the $B$-graph $\Delta$.
\end{remark}

For $H\le\FF$, both $\pibar(H)$ and $\EHtoFVarActsVar{S_n}{X}$ are invariant under automorphisms
of $\FF$, and do not change when $\FF$ is replaced by a free factor that contains $H$. In
particular, we may conjugate $H$ to assume that $H=\pi_1(\Gamma_B(H),v_0)$ for some vertex $v_0$,
and we may assume that $\Gamma_B(H)\to\Omega_B$ is surjective (that is, every letter of $B$ is the
label of some edge of $\Gamma_B(H)$).

\begin{lemma}[{for cyclic subgroups, cf.\ the proof of \cite[Proposition 5.3]{puder2025stableII}}]
\label{lemma_induction_convolution_Com}
    Let $H\le\FF$ be non-trivial and finitely generated, let $d\le n\in \N$, and let $\beta\in \Hom(H, S_d)$.
    Denote $\Gamma\defeq \Gamma_B(H)$, so that $H = \pi_1\prn*{\Gamma, v_0}$ for some vertex $v_0\in V(\Gamma)$, and let $(\Gamma^{\beta}, p, c)$ be the numbered covering of $(\Gamma, v_0)$ corresponding to $\beta$. Then 
    \[ \EX_{\alphaSimUHomFSn}\brackets*{\abs*{\Interinj(\alpha\restriction_H, \beta)}} = \sum_{(\eta_1, \eta_2)\in \DecompB^2\prn*{\Gamma^{\beta}\to \Omega_B}} \one\braces*{\eta_1 \textup{ is }p\textup{-efficient}} \cdot L^B_{\eta_2}(n). \]
\end{lemma}

\begin{proof}
    Let $\alpha\in \Hom(\FF, S_n)$. 
    Define $\textup{Val}_{\alpha}(\Gamma^{\beta}, [n])$ as the set of $\alpha$-valid functions $V\prn*{\Gamma^{\beta}}\to [n]$ whose restriction to $p^{-1}(v_0)$ is injective.
    We start by proving that the map $c^*\colon \textup{Val}_{\alpha}(\Gamma^{\beta}, [n])\to \Interinj(\alpha\restriction_H, \beta)$ defined by $c^* f \defeq f\restriction_{p^{-1}\prn*{v_0}}\circ c$ is a well-defined bijection.
    Let $h\in H, f\in \textup{Val}_{\alpha}(\Gamma^{\beta}, [n])$. By Proposition~\ref{prop_alg_top}, and as $f$ is $\alpha$-valid, the following diagram commutes:
\[\begin{tikzcd}
	{[d]} & {p^{-1}(v_0)} & {[n]} \\
	{[d]} & {p^{-1}(v_0)} & {[n]}
	\arrow["c"', from=1-1, to=1-2]
	\arrow["h"', from=1-2, to=2-2]
	\arrow["{\beta(h)}"', from=1-1, to=2-1]
	\arrow["c", from=2-1, to=2-2]
	\arrow["f"', hook, from=1-2, to=1-3]
	\arrow["f"', hook, from=2-2, to=2-3]
	\arrow["{\alpha(h)}", from=1-3, to=2-3]
\end{tikzcd}\]
This shows that $c^*$ is well-defined. 
Since $\Gamma$ is connected, for every $u\in V\prn*{\Gamma^{\beta}}$ there are $v\in p^{-1}(v_0)$ and $w\in \FF$ such that a path from $v$ to $u$ reads $w$, and by the validity of $f$, we have $f(u) = f(v).\alpha(w)$; hence $f$ is determined by $c^*f$, and $c^*$ is injective. Let $\iota\in \Interinj(\alpha\restriction_H, \beta)$. If paths from $v_1,v_2\in p^{-1}(v_0)$ to $u$ read $w_1,w_2$, then the path from $v_1$ to $u$ and back to $v_2$ reads $h\defeq w_1w_2^{-1}$ and projects to a loop at $v_0$, so $h\in H$ and $v_1.h=v_2$, and, by Proposition~\ref{prop_alg_top}, $c^{-1}(v_2)=c^{-1}(v_1).\beta(h)$. As $\iota$ intertwines $\alpha\restriction_H$ and $\beta$,
\[\iota(c^{-1}(v_2)).\alpha(w_2)=\iota(c^{-1}(v_1)).\alpha(h).\alpha(w_2)=\iota(c^{-1}(v_1)).\alpha(w_1).\]
Therefore $f_\iota(u)\defeq\iota(c^{-1}(v_1)).\alpha(w_1)$ is well defined; it is $\alpha$-valid, and $f_\iota\circ c=\iota$, so $f_\iota\in \textup{Val}_{\alpha}(\Gamma^{\beta}, [n])$ and $c^*f_\iota=\iota$, showing that $c^*$ is bijective.

For every $ f\in \textup{Val}_{\alpha}(\Gamma^{\beta}, [n])$, define a $B$-graph $\Gamma^{\beta}/f $ as follows. 
The vertices $V\prn*{\Gamma^{\beta}/f}$ are $\braces*{f^{-1}(i)}_{i\in \Img(f)}$, and there is a $b$-labeled edge $f^{-1}(i)\overset{b}{\to} f^{-1}(j)$ whenever there are $v\in f^{-1}(i), u\in f^{-1}(j)$ with $\prn*{v\overset{b}{\to} u} \in E(\Gamma^{\beta})$. By the validity constraints \eqref{eq_valid_func}, $\Gamma^{\beta}/f$ is indeed a $B$-graph, and it is a core graph, as the quotient map from the core graph $\Gamma^{\beta}$ is an immersion, which cannot decrease degrees. There is a natural decomposition 
\[\begin{tikzcd}
	{\Gamma^{\beta}} \\
	{\Gamma^{\beta}/f} & {[n]}
	\arrow["{\eta_f}"', from=1-1, to=2-1]
	\arrow["{\bar{f}}", from=2-1, to=2-2]
	\arrow["f", from=1-1, to=2-2]
\end{tikzcd}\]
where $\eta_f$ is a surjective morphism of core graphs, and $\bar{f}$ is injective and $\alpha$-valid. 
Since $f\restriction_{p^{-1}(v_0)}$ is injective, $\eta_f$ is $p$-efficient.
Clearly, such pairs $(\eta_f, \bar{f})$ where $\eta_f$ is a $p$-efficient surjective morphism and $\bar{f}$ is injective and $\alpha$-valid, are in bijective correspondence with $\textup{Val}_{\alpha}(\Gamma^{\beta}, [n])$, so $\abs*{\Interinj(\alpha\restriction_H, \beta)}$ equals
\[\abs*{\textup{Val}_{\alpha}(\Gamma^{\beta}, [n])}
= \sum_{(\eta_1, \eta_2)\in \DecompB^2\prn*{\Gamma^{\beta}\to \Omega_B}} \one\braces*{\eta_1 \textup{ is }p\textup{-efficient}} \cdot \abs*{\braces*{\begin{array}{cc}
     \alpha\textup{-valid injective} \\
     \textup{functions Im}(\eta_1)\to [n]
\end{array}}}.\]
Now take expectation \wrt $\alphaSimUHomFSn$ and apply Proposition~\ref{prop_LB}.
\end{proof}

From Proposition~\ref{prop_common_fixed_mod_G_and_Com} and Lemma~\ref{lemma_induction_convolution_Com} we immediately get:

\begin{corollary}
\label{cor_EHtoFSnActsNDG_sum}
Let $H\le \FF$ be a non-trivial finitely generated subgroup, let $d\le n\in \N$ and $G\le S_d$. Denote $\Gamma\defeq \Gamma_B(H)$. Then
\[ \EHtoFSnActsNDG = \frac{1}{|G|} \sum_{\beta\in \Hom(H, G)} \sum_{(\eta_1, \eta_2)\in \DecompB^2\prn*{\Gamma^{\beta}\to \Omega_B}} \one\braces*{\eta_1 \textup{ is }p\textup{-efficient}} \cdot L^B_{\eta_2}(n).  \]
\end{corollary}

Fix now $H\le \FF$ non-trivial and finitely generated, let $\Gamma\defeq \Gamma_B(H)$ with basepoint $v_0$,
and for $\beta\in \Hom(H, S_d)$ let $p_\beta\colon \Gamma^{\beta}\to \Gamma$ be the corresponding
$d$-covering given by Proposition~\ref{prop_alg_top}.
Denote
\[ \mathcal{P}_d(H) \defeq \braces*{ (\beta, \eta_1) : \begin{array}{cc}
     & \beta\in \Hom(H, S_d),\ \eta_1\colon \Gamma^{\beta}\twoheadrightarrow \Img(\eta_1) \\
     & \textup{ is a }p_\beta\textup{-efficient surjective morphism}
\end{array}},\]
so that the sum in Lemma~\ref{lemma_induction_convolution_Com} ranges over $\mathcal{P}_d(H)$.

\begin{corollary}
\label{lemma_p_efficient_vs_spibard}
    For every $d\in \N$,
    \[ \min_{(\beta,\eta_1)\in \mathcal{P}_d(H)} \prn*{-\chi\prn*{\Img(\eta_1)}}
    = d\cdot \spivarbar{d}(H) = d\cdot\prn*{\pibar(H)-1}, \]
    and the minimum is attained where $\beta$ is trivial.
\end{corollary}
\begin{proof}
For $(\beta,\eta_1)\in\mathcal{P}_d(H)$, by Remark~\ref{rem_p_efficient_is_fiberwise}, the
pullback $\Gamma\times_{\Omega_B}\Img(\eta_1)$ contains the $d$-covering $\Gamma^{\beta}$. Conversely, assume that for some core graph $\Delta$
there is a $d$-covering of $\Gamma$ inside the pullback $\Gamma\times_{\Omega_B}\Delta$ (which
must be of the form $\Gamma^{\beta}$ by Proposition~\ref{prop_alg_top}); then the projection
$\eta_1\colon\Gamma^{\beta}\to\Delta$ is $p_\beta$-efficient and
$\chi(\Img(\eta_1))\ge\chi(\Delta)$. Therefore
$\min_{(\beta,\eta_1)\in \mathcal{P}_d(H)} \prn*{-\chi\prn*{\Img(\eta_1)}} = d\cdot\spivarbar{d}(H)$.
The second equality is Theorem~\ref{thm_jz_stability}. For
$J\in\overline{\Crit}(H)$, the disjoint union $\Delta$ of $d$ copies of $\Gamma_B(J)$
has $-\chi(\Delta)=d\cdot(\pibar(H)-1)$ and, as $H\le J$, the pullback $\Gamma\times_{\Omega_B}\Delta$
contains the trivial $d$-covering of $\Gamma$.
\end{proof}

\begin{proposition}
\label{prop_Sn_quotients_of_tuples}
    For every $d\in \Z_{\ge0}$ and $G\le S_d$, as $n\to \infty$,
    \[ \EHtoFSnActsNDG = \Theta\prn*{n^{-d\cdot (\pibar(H)-1)}}. \]
    For $G=S_d$, as $[n]_d/S_d\cong\binom{[n]}{d}$ and $\binom{n}{d}=\Theta(n^d)$, we get the first
    part of Theorem~\ref{thm_Sn_spibard}; for $G=\{1\}$, we get
    equation~\eqref{eq_EHtoFSnActsTuples}.
\end{proposition}

\begin{proof}
    For $d=0$, the set $[n]_0/G$ has one point, so $\EHtoFSnActsNDG=1$. Let $d\ge1$. By
    Corollary~\ref{cor_EHtoFSnActsNDG_sum} and Proposition~\ref{prop_LB},
    \[ \EHtoFSnActsNDG = \frac{1}{\abs*{G}}
    \sum_{\substack{(\beta,\eta_1)\in\mathcal{P}_d(H)\\ \beta\in\Hom(H,G)}} L^B_{\eta_2}(n)
    = \frac{1}{\abs*{G}}
    \sum_{\substack{(\beta,\eta_1)\in\mathcal{P}_d(H)\\ \beta\in\Hom(H,G)}}
    n^{\chi(\Img(\eta_1))}\prn*{1 + O(n^{-1})}, \]
    where $\eta_2\colon \Img(\eta_1)\to \Omega_B$ is the unique morphism.
    The index set of the sum is finite and does not depend on $n$.
    Since all the summands are non-negative, the sum is $\Theta$ of its largest term. As
    $\Hom(H,G)$ contains the trivial homomorphism $\beta$, by
    Corollary~\ref{lemma_p_efficient_vs_spibard} the
    maximum of $\chi(\Img(\eta_1))$ over the index set is $-d\cdot(\pibar(H)-1)$.
\end{proof}

We now prove the second part of Theorem~\ref{thm_Sn_spibard}, about a sequence
of transitive actions $S_n\acts X_n$ with $\abs*{X_n}=n^{O(1)}$.
Two steps of the proof hold for every finite group:

\begin{observation}
\label{obs_index_sandwich}
    Let $G_1\le G_2\le G$ be finite groups. Then for every finitely generated $H\le\FF$,
    \[ \EHtoFVarActsVar{G}{G_1\backslash G}
    \le [G_2:G_1]\cdot\EHtoFVarActsVar{G}{G_2\backslash G}. \]
    Indeed, the map $G_1\backslash G\to G_2\backslash G$ given by $G_1g\mapsto G_2g$ commutes with
    the action of $G$, so it maps fixed points of $\alpha(H)$ to fixed points, and each of its fibers has
    $[G_2:G_1]$ points.
\end{observation}

Recall that $H\le\FF=\inner{B}$ and $\Gamma=\Gamma_B(H)$ are fixed.

\begin{proposition}
\label{prop_product_lower_bound}
    Let a finite group $G$ act on finite sets $X$ and $Y$, such that for every $U\subseteq X$ with
    $\abs*{U}\le\abs*{V(\Gamma)}$, the pointwise stabilizer $G_{(U)}$ acts transitively on $Y$.
    Let $G$ act on $X\times Y$ diagonally. Then
    \[ \EHtoFVarActsVar{G}{X\times Y}\ge\abs*{Y}^{1-\abs*{B}}\,\EHtoFVarActsVar{G}{X}. \]
\end{proposition}
\begin{proof}
    As $(X\times Y)^{\alpha(H)}=X^{\alpha(H)}\times Y^{\alpha(H)}$ and
    $Y^{\alpha(\FF)}\subseteq Y^{\alpha(H)}$,
    \begin{align*}
        \EHtoFVarActsVar{G}{X\times Y}
        &= \EX_{\alpha}\brackets*{\abs*{X^{\alpha(H)}}\cdot\abs*{Y^{\alpha(H)}}}\\
        &\ge \EX_{\alpha}\brackets*{\abs*{X^{\alpha(H)}}\cdot\abs*{Y^{\alpha(\FF)}}}
        \overset{(\star)}{=} \abs*{Y}^{1-\abs*{B}}\,\EHtoFVarActsVar{G}{X}.
    \end{align*}
    Let us now explain the last equality $(\star)$. By Reiter's observation (see
    \eqref{eq_valid_func}),
    $\abs*{X^{\alpha(H)}}=\sum_{f\colon V(\Gamma)\to X}\one\braces*{f\textup{ is }\alpha\textup{-valid}}$.
    For a function $f$ and $b\in B$, the set
    \[ C_b\defeq\braces*{g\in G : f(\src(e)).g=f(\tar(e))\ \ \forall e\in E_b(\Gamma)} \]
    is either empty or a coset of the pointwise stabilizer $G_{(f(\src(E_b(\Gamma))))}$, which acts
    transitively on $Y$. Hence, for every $y\in Y$, the stabilizer $G_y$ of $y$ satisfies
    $\abs*{C_b\cap G_y}=\abs*{C_b}/\abs*{Y}$. As the elements $\alpha(b)$ are independent and
    uniform, for every $y\in Y$,
    \[ \PR\prn*{f\textup{ is }\alpha\textup{-valid and }\alpha(\FF)\le G_y}
    =\prod_{b\in B}\frac{\abs*{C_b\cap G_y}}{\abs*{G}}
    =\abs*{Y}^{-\abs*{B}}\,\PR\prn*{f\textup{ is }\alpha\textup{-valid}}. \]
    The sum of these equalities over $y\in Y$ and $f\colon V(\Gamma)\to X$ is $(\star)$.
\end{proof}

The following classification of the subgroups of small index in $S_n$ goes back to Jordan; the
modern formulation is due to Liebeck.

\begin{theorem}[Jordan--Liebeck, {\cite[Theorem 5.2B]{DixonMortimer1996}}]
\label{thm_jordan_liebeck}
    Let $r\in\Z_{\ge1}$. For every large enough $n$, if $P_n\le S_n$ satisfies
    $[S_n : P_n] < \binom{n}{r}$, then, up to conjugation of $P_n$, we have
    $\{1\}\times A_{n-d}\le P_n\le S_d\times S_{n-d}$ for some $d<r$, where
    $S_d\times S_{n-d}\le S_n$ is the setwise stabilizer of $[d]$.\footnote{
        In the other cases of \cite[Theorem 5.2B]{DixonMortimer1996}, $n\le 8$, or $n$ is even and
        $[S_n:P_n]=\frac12\binom{n}{n/2}$, which is larger than $\binom{n}{r}$ for large $n$.}
\end{theorem}

\begin{proof}[Proof of the second part of Theorem~\ref{thm_Sn_spibard}]
    Let $P_n\le S_n$ be the stabilizer of a point of $X_n$, so that $X_n\cong P_n\backslash S_n$.
    As $\abs*{X_n}=n^{O(1)}$, there is $r\in\Z_{\ge1}$ such that $[S_n:P_n]=\abs*{X_n}<\binom{n}{r}$
    for large $n$. By Theorem~\ref{thm_jordan_liebeck}, for a suitable choice of the point,
    $\{1\}\times A_{n-d}\le P_n\le S_d\times S_{n-d}$ for some $d<r$; below, we denote
    $\{1\}\times A_{n-d}$ by $A_{n-d}$. Denote by $G\le S_d$ the image of $P_n$ under the projection
    to $S_d$. As $d<r$, the pair $(d,G)$ takes finitely many values, so we may assume that it does
    not depend on $n$. As $P_n$ maps onto $G$ and contains $A_{n-d}$, we have
    $[G\times S_{n-d}:P_n]\le 2$ and $[P_n:A_{n-d}]\le 2\abs*{G}$.
    By Observation~\ref{obs_index_sandwich}, applied to $A_{n-d}\le P_n\le S_n$ and to
    $P_n\le G\times S_{n-d}\le S_n$, and by Proposition~\ref{prop_Sn_quotients_of_tuples}, as
    $(G\times S_{n-d})\backslash S_n\cong[n]_d/G$ via $g\mapsto(1.g,\ldots,d.g)$,
    \[ \frac{1}{2\abs*{G}}\EHtoFVarActsVar{S_n}{A_{n-d}\backslash S_n}
    \le\EHtoFVarActsVar{S_n}{X_n}\le 2\,\EHtoFSnActsNDG=\Theta\prn*{n^{-d\cdot(\pibar(H)-1)}}. \]

    For the lower bound, let $n\ge d\abs*{V(\Gamma)}+2$, and let $U\subseteq[n]_d$ of size
    $\abs*{U}\le\abs*{V(\Gamma)}$. At most $d\abs*{V(\Gamma)}\le n-2$ points of $[n]$ appear in
    the tuples in $U$, so $(S_n)_{(U)}$ contains a transposition, and therefore acts transitively
    on $\{\pm1\}$ through $\sgn$. Hence, by Proposition~\ref{prop_product_lower_bound} for
    $X=[n]_d$ and $Y=\{\pm1\}$, and by Proposition~\ref{prop_Sn_quotients_of_tuples} for
    $\{1\}\le S_d$, as
    $A_{n-d}\backslash S_n\cong[n]_d\times\{\pm1\}$ via
    $A_{n-d}g\mapsto((1.g,\ldots,d.g),\sgn(g))$,
    \[ \EHtoFVarActsVar{S_n}{A_{n-d}\backslash S_n}\ge 2^{1-\abs*{B}}\,\EHtoFVarActsVar{S_n}{[n]_d}
    =\Theta\prn*{n^{-d\cdot(\pibar(H)-1)}}. \]

    Finally, $\EHtoFVarActsVar{S_n}{X_n}=\Theta\prn*{n^{-d\cdot(\pibar(H)-1)}}$ and
    $\abs*{X_n}=[G\times S_{n-d}:P_n]\cdot(n)_d/\abs*{G}=\Theta(n^d)$, so
    \[ \EHtoFVarActsVar{S_n}{X_n}=\Theta\prn*{\abs*{X_n}^{1-\pibar(H)}}. \qedhere \]
\end{proof}

\subsection*{Counting fixed subspaces of $\F_q^n$}

In \cite[Definition 1.10]{ernst2024word}
(see also \cite[Definition 1.2]{ernst2024ring}),
a function $\tilde{A}\colon \GL_n(\F_q)\to \N$ (denoted $\tilde{B}$ there) was defined
for every $A\in \GL_d(\F_q)$
by
\[\tilde{A}(g)\defeq \abs*{\braces*{\iota\colon \F_q^d\to \F_q^n:
\iota g = A\iota}}.\]
In \cite[Theorem 1.11]{ernst2024word}, it was shown that for
every $w\in \FF$ and $A\in \GL_d(\F_q)$, if $\alpha$ is uniformly random in
$\Hom(\FF,\GL_n(\F_q))$, then
$\E_{\alpha}[\tilde{A}(\alpha(w))]$ coincides with a rational function in $q^n$
for every large enough $n$.
Later, another version of $\tilde{A}$, restricted to full-rank $\iota$, was presented in
\cite[equation (1.5)]{ernst2024ring}:
\[\tilde{A}^{\textup{f.r.}}(g)\defeq \abs*{\braces*{\iota\colon \F_q^d\to \F_q^n:
\iota\textup{ is injective and }\iota g = A\iota}}.\]

Lemma~\ref{lemma_GLn_inter_count} below generalizes the leading term of the rational function in
$q^n$ from \cite[Theorem 1.11]{ernst2024word} in two directions:
The first direction is the generalization from a word
$w\in \FF$ to a subgroup $H\le \FF$, which requires
us to generalize $\tilde{A}(g)\in \N$
(which is defined for $A\in \GL_d(\F_q)$ and $g\in\GL_n(\F_q)$)
to $\abs*{\Inter(\alpha\restriction_H, \beta)}$
(where $\alpha\in \Hom(\FF, \GL_n(\F_q))$
generalizes $g$ and $\beta\in\Hom(H, \GL_d(\F_q))$
generalizes $A$).
The second direction is that we work both with
$\Inter(\alpha\restriction_H, \beta)$ (that generalizes $\tilde{A}(g)$)
and $\Interinj(\alpha\restriction_H, \beta)$ (that generalizes
$\tilde{A}^{\textup{f.r.}}(g)$).

Fix $d\in\N$, a prime power $q$, and a finitely generated subgroup $H\le\FF$ with a basis
$B_H$. For $m\in\N$, denote by $\mathcal{E}_m=\{e_i\}_{i=1}^m$ the standard $\F_q[\FF]$-basis of
$\F_q[\FF]^m$.

\begin{definition}
    \label{def_M_beta}
    For $\beta\in \Hom(H, \GLvarFq{d})$, denote by $M_{\beta}$ the right $\F_q[H]$-submodule of
    $\F_q[H]^d$ generated by the rows of the matrices
    $\braces{h\cdot I_{d\times d}-\beta(h)}_{h\in H}$. That is,
    \[ M_{\beta}=\inner[\Big]{e_ih-\textstyle\sum_{j=1}^d\beta(h)_{i,j}e_j}_{h\in H,\,i\le d}. \]
\end{definition}

\begin{observation}
    \label{obs_M_beta_generators}
    For $\beta\in \Hom(H, \GLvarFq{d})$, the rows of the matrices
    $\braces{h\cdot I_{d\times d}-\beta(h)}_{h\in B_H}$ generate $M_{\beta}$. Indeed, for
    $h,h_1,h_2\in H$,
    \begin{align*}
        h_1h_2\cdot I_{d\times d}-\beta(h_1h_2)
        &=(h_1\cdot I_{d\times d}-\beta(h_1))h_2+\beta(h_1)(h_2\cdot I_{d\times d}-\beta(h_2)),\\
        h^{-1}\cdot I_{d\times d}-\beta(h^{-1})
        &=-\beta(h^{-1})(h\cdot I_{d\times d}-\beta(h))h^{-1}.
    \end{align*}
\end{observation}

\begin{proposition}
    \label{prop_normal_form_for_bases}
    Let $M\le\F_q[H]^d$ be a submodule. Then $\mathcal{E}_d+M$ is an $\F_q$-basis of
    $M\backslash\F_q[H]^d$ if and only if $M=M_{\beta}$ for some $\beta\in\Hom(H,\GLvarFq{d})$.
    This $\beta$ is unique: for $h\in H$ and $i\in[d]$, we have $(e_i+M)h=\sum_{j=1}^d\beta(h)_{i,j}(e_j+M)$.
\end{proposition}

\begin{proof}
    Let $\beta\in\Hom(H,\GLvarFq{d})$. By Definition~\ref{def_M_beta},
    $e_ih\equiv\sum_j\beta(h)_{i,j}e_j$ modulo $M_{\beta}$, so $\mathcal{E}_d+M_{\beta}$ spans
    $M_{\beta}\backslash\F_q[H]^d$ over $\F_q$. To show that $\mathcal{E}_d+M_{\beta}$ is linearly independent,
    consider $\F_q^d$ as a right $\F_q[H]$-module by $u\cdot h\defeq u\beta(h)$ for $u\in\F_q^d$ and
    $h\in H$, and let
    $\varphi\colon\F_q[H]^d\to\F_q^d$ be the $\F_q[H]$-linear map that maps $\mathcal{E}_d$ to the
    standard basis of $\F_q^d$. Then $\varphi(e_ih)=\varphi(e_i)\beta(h)$, so $\varphi$ vanishes on
    $M_{\beta}$, so it induces a map $M_{\beta}\backslash\F_q[H]^d\to\F_q^d$, which maps
    $\mathcal{E}_d+M_{\beta}$ to the standard basis of $\F_q^d$.

    Conversely, assume that $\mathcal{E}_d+M$ is an $\F_q$-basis of $M\backslash\F_q[H]^d$. For
    $h\in H$, let $\beta(h)$ be the $d\times d$ matrix over $\F_q$ with
    $(e_i+M)h=\sum_{j=1}^d\beta(h)_{i,j}(e_j+M)$ for every $i\in[d]$. Then $\beta(1)=I_{d\times d}$,
    and $\beta(h_1h_2)=\beta(h_1)\beta(h_2)$ for $h_1,h_2\in H$, as
    $(e_i+M)h_1h_2=\sum_{j,k}\beta(h_1)_{i,j}\beta(h_2)_{j,k}(e_k+M)$. Hence
    $\beta\in\Hom(H,\GLvarFq{d})$. The generators $e_ih-\sum_j\beta(h)_{i,j}e_j$ of $M_{\beta}$ are in
    $M$, so $M_{\beta}\le M$. Since both have codimension $d$, we have $M=M_{\beta}$.
\end{proof}

Following \cite[Section 1D]{ernst2024word}, we say that a module
$N\le \F_q[\FF]^m$ is \textbf{supported} on a set
$S\subseteq \mathcal{E}_m\times \FF$
if $N$ is generated by the intersection $N\cap \F_q[S]$, where $\F_q[S]$ is the $\F_q$-span of the
monomials $e_ig$ with $(e_i,g)\in S$.

For $\beta\in\Hom(H,\GLvarFq{d})$, denote
$M_{\beta}^\FF \defeq M_{\beta}\otimes_{\F_q[H]} \F_q[\FF]$. Since $\F_q[\FF]$ is a free left
$\F_q[H]$-module on a right transversal of $H$ containing $1$, we identify $M_{\beta}^\FF$ with the
submodule $M_{\beta}\cdot\F_q[\FF]$ of $\F_q[\FF]^d$, and $M_{\beta}^\FF\cap\F_q[H]^d=M_{\beta}$.

Denote by $\T_B(H)$ the minimal subtree of $\textup{Cay}(\FF, B)$ that contains $B_H$ and the
identity vertex $1$, and denote $\T_B(H)^d\defeq \mathcal{E}_d\times \T_B(H)$. There are finitely
many submodules of $\F_q[\FF]^d$ supported on $\T_B(H)^d$, as each is generated by its intersection
with the finite space $\F_q[\T_B(H)^d]$. For $\beta\in\Hom(H,\GLvarFq{d})$, denote
\begin{align*}
\mathcal{N}_{\beta}&\defeq\braces*{N\le\F_q[\FF]^d\,\middle|\,M_{\beta}^\FF\le N\textup{ and }
N\textup{ is supported on }\T_B(H)^d},\\
\mathcal{N}_{\beta}^{\textup{inj}}&\defeq\braces*{N\in\mathcal{N}_{\beta}\,\middle|\,
N\cap\F_q[H]^d=M_{\beta}}.
\end{align*}
By Observation~\ref{obs_M_beta_generators}, $M_{\beta}^\FF$ is generated by the rows of the matrices
$\braces{h\cdot I_{d\times d}-\beta(h)}_{h\in B_H}$, and these rows are in $\F_q[\T_B(H)^d]$. Hence
$M_{\beta}^\FF\in\mathcal{N}_{\beta}^{\textup{inj}}$, as $M_{\beta}^\FF\cap\F_q[H]^d=M_{\beta}$.

\begin{lemma}
\label{lemma_GLn_inter_count}
    Let $\beta\in \Hom(H, \GLvarFq{d})$, and let $\alpha$ be uniformly random in
    $\Hom(\FF, \GLnFq)$. Then, as $n\to\infty$,
    \[ \E_{\alpha}\brackets*{\abs*{\Interinj(\alpha\restriction_H, \beta)}}
    =\sum_{N\in\mathcal{N}_{\beta}^{\textup{inj}}}q^{n(d-\rk(N))}\prn*{1+O(q^{-n})}. \]
    For completeness, although it is not used below, the same holds for
    $\E_{\alpha}\brackets*{\abs*{\Inter(\alpha\restriction_H, \beta)}}$, with
    $\mathcal{N}_{\beta}^{\textup{inj}}$ replaced by $\mathcal{N}_{\beta}$.
\end{lemma}

\begin{proof}
The proof follows \cite[Section 2]{ernst2024word}; there, $H=\inner{w}$, $B_H=\{w\}$, and $\T_B(H)$
is the path from $1$ to $w$. By definition,
\[ \E_{\alpha}\brackets*{\abs*{\Inter(\alpha\restriction_H,\beta)}}
=\frac{\abs*{\braces*{(\alpha,\iota)\in\Hom(\FF,\GLnFq)\times\Hom(\F_q^d,\F_q^n)\,\middle|\,
\iota\alpha(h)=\beta(h)\iota\ \forall h\in H}}}{\abs*{\GLnFq}^{\rk(\FF)}}. \]
We group the pairs $(\alpha,\iota)$ according to the linear equations over $\F_q$ satisfied by the
rows of the matrices $\iota\alpha(v)$, for $v\in\T_B(H)$.
For $\alpha\in\Hom(\FF,\GLnFq)$, consider $\F_q^n$ as a right $\F_q[\FF]$-module by
$u\cdot z\defeq u\alpha(z)$ for $u\in\F_q^n$ and $z\in\FF$. For $\iota\in\Hom(\F_q^d,\F_q^n)$, let
$f_{\alpha,\iota}\colon\F_q[\FF]^d\to\F_q^n$ be the $\F_q[\FF]$-linear map that maps each $e_j$ to
the $j$-th row of $\iota$, and denote by $N_{\alpha,\iota}$ the submodule generated by
$\ker f_{\alpha,\iota}\cap\F_q[\T_B(H)^d]$. Since $f_{\alpha,\iota}(e_jv)$ is the $j$-th row of
$\iota\alpha(v)$, the linear equations satisfied by the rows of the matrices $\iota\alpha(v)$, for
$v\in\T_B(H)$, are the elements of $\ker f_{\alpha,\iota}\cap\F_q[\T_B(H)^d]$. The submodule
$N_{\alpha,\iota}$ is supported on $\T_B(H)^d$, and
$N_{\alpha,\iota}\cap\F_q[\T_B(H)^d]=\ker f_{\alpha,\iota}\cap\F_q[\T_B(H)^d]$, as
$N_{\alpha,\iota}\le\ker f_{\alpha,\iota}$.

Because $f_{\alpha,\iota}\prn[\big]{e_ih-\sum_j\beta(h)_{i,j}e_j}$ is the $i$-th row of
$\iota\alpha(h)-\beta(h)\iota$, for $h\in H$ and $i\le d$, the map $\iota$ is in $\Inter(\alpha\restriction_H,\beta)$ if and only
if $\ker f_{\alpha,\iota}$ contains the rows of the matrices
$\braces{h\cdot I_{d\times d}-\beta(h)}_{h\in H}$, that is, if and only if
$M_{\beta}^\FF\le\ker f_{\alpha,\iota}$. By Observation~\ref{obs_M_beta_generators},
$M_{\beta}^\FF$ is generated by elements of $\F_q[\T_B(H)^d]$, so
$M_{\beta}^\FF\le\ker f_{\alpha,\iota}$ if and only if $M_{\beta}^\FF\le N_{\alpha,\iota}$, that is, if
and only if $N_{\alpha,\iota}\in\mathcal{N}_{\beta}$. Moreover, $\iota$ is injective if and only if
its rows $f_{\alpha,\iota}(e_1),\ldots,f_{\alpha,\iota}(e_d)$ are linearly independent, that is, if
and only if $\mathcal{E}_d$ is linearly independent modulo $N_{\alpha,\iota}$ (as
$\mathcal{E}_d\subseteq\F_q[\T_B(H)^d]$), or equivalently modulo $N_{\alpha,\iota}\cap\F_q[H]^d$ (as
$\mathcal{E}_d\subseteq\F_q[H]^d$). By Proposition~\ref{prop_normal_form_for_bases},
$\mathcal{E}_d+M_{\beta}$ is a basis of $M_{\beta}\backslash\F_q[H]^d$. Hence, if
$N_{\alpha,\iota}\in\mathcal{N}_{\beta}$, then the surjection
$M_{\beta}\backslash\F_q[H]^d\to(N_{\alpha,\iota}\cap\F_q[H]^d)\backslash\F_q[H]^d$ maps
$\mathcal{E}_d+M_{\beta}$ to $d$ linearly independent elements if and only if that surjection is
injective, that is, if and only if $N_{\alpha,\iota}\in\mathcal{N}_{\beta}^{\textup{inj}}$.

Hence the pairs $(\alpha,\iota)$ with $\iota\in\Inter(\alpha\restriction_H,\beta)$ are partitioned
according to $N_{\alpha,\iota}\in\mathcal{N}_{\beta}$, and those with
$\iota\in\Interinj(\alpha\restriction_H,\beta)$ are the pairs with
$N_{\alpha,\iota}\in\mathcal{N}_{\beta}^{\textup{inj}}$. For $H=\inner{w}$, $N\in\mathcal{N}_{\beta}$
and large $n$, the ratio
$\abs{\braces{(\alpha,\iota)\mid N_{\alpha,\iota}=N}}/\abs{\GLnFq}^{\rk(\FF)}$ is the summand of $N$
in \cite[equation (2-7)]{ernst2024word} (there, $n$, $d$ and $N$ are denoted $N$, $m$ and $M$),
and, by \cite[Corollary 3.9]{ernst2024word}, this summand is $q^{n(d-\rk(N))}\prn*{1+O(q^{-n})}$.
The proofs of \cite[equation (2-7) and Corollary 3.9]{ernst2024word} use only that the path from $1$
to $w$ is a finite subtree that contains $1$ and $w$, so they apply to $\T_B(H)$, a finite subtree
that contains $1$ and $B_H$. Hence, as $n\to\infty$,
\begin{alignat}{2}
\E_{\alpha}\brackets[\big]{\abs[\big]{\Inter(\alpha\restriction_H, \beta)}}
&=\sum_{N\in\mathcal{N}_{\beta}}\frac{\abs*{\braces*{(\alpha,\iota)\mid N_{\alpha,\iota}=N}}}
{\abs*{\GLnFq}^{\rk(\FF)}}&&=\sum_{N\in\mathcal{N}_{\beta}}q^{n(d-\rk(N))}\prn*{1+O(q^{-n})},
\label{eq_Inter_as_sum}\\
\E_{\alpha}\brackets[\big]{\abs[\big]{\Interinj(\alpha\restriction_H, \beta)}}
&=\sum_{N\in\mathcal{N}_{\beta}^{\textup{inj}}}\frac{\abs*{\braces*{(\alpha,\iota)\mid N_{\alpha,\iota}=N}}}
{\abs*{\GLnFq}^{\rk(\FF)}}&&=\sum_{N\in\mathcal{N}_{\beta}^{\textup{inj}}}q^{n(d-\rk(N))}\prn*{1+O(q^{-n})}.
\label{eq_Interinj_as_sum}
\end{alignat}
\end{proof}

\begin{remark}
    \label{remark_why_q_analog}
    By \cite[equation (2-7)]{ernst2024word}, for large $n$, the term of $N$ in the middle sum in
    \eqref{eq_Inter_as_sum} is a product (and quotient) of expressions of the form
    $(q^n - 1)(q^n - q)\cdots (q^n - q^{m-1})$,
    which is the $q$-analog of the falling factorial
    $(n)_m = n(n-1)\cdots (n-(m-1))$.
    Similarly one can think of $\GL_n(\F_q)$ (or $\textup{SL}_n(\F_q)$) as a
    $q$-analog of the symmetric group $S_n$.
    For the form-preserving groups, the count in the proof of Lemma~\ref{lemma_GLn_inter_count}
    would also involve the values of the form on the rows of the matrices $\iota\alpha(v)$, as in
    \cite[Lemma 3.3]{eberhard2022babai}.
\end{remark}

Recall from Definition~\ref{def_spibarqd} that
\begin{equation}
\label{eq_spibarqd_recall}
s\pibar_{q,d}(H)=\inf\braces*{\frac{\rk(N)}{d}-1\,\middle|\,\begin{array}{l}
N\le\F_q[\FF]^d\textup{ is finitely generated, and}\\
\dim_{\F_q}\prn*{(N\cap\F_q[H]^d)\backslash\F_q[H]^d}=d\end{array}}.
\end{equation}

\begin{proposition}
    \label{prop_spibarqd_as_min}
    The $d$-stable $\F_q$-compressed rank of $H$ is a minimum over a finite, non-empty set:
    \[ s\pibar_{q,d}(H)=\min\braces*{\frac{\rk(N)}{d}-1\,\middle|\,\beta\in\Hom(H,\GLvarFq{d}),\
    N\in\mathcal{N}_{\beta}^{\textup{inj}}}. \]
\end{proposition}

\begin{proof}
    The set of pairs $(\beta,N)$ with $N\in\mathcal{N}_{\beta}^{\textup{inj}}$ is finite, since
    $\Hom(H,\GLvarFq{d})$ and every $\mathcal{N}_{\beta}^{\textup{inj}}$ are finite, and non-empty,
    since $M_{\beta}^\FF\in\mathcal{N}_{\beta}^{\textup{inj}}$.
    For every $\beta\in\Hom(H,\GLvarFq{d})$, every $N\in\mathcal{N}_{\beta}^{\textup{inj}}$ is finitely
    generated, as $\T_B(H)^d$ is finite, and, by Proposition~\ref{prop_normal_form_for_bases},
    $\dim_{\F_q}\prn*{(N\cap\F_q[H]^d)\backslash\F_q[H]^d}=\dim_{\F_q}(M_{\beta}\backslash\F_q[H]^d)=d$.
    Hence $N$ is one of the submodules in \eqref{eq_spibarqd_recall}, and
    $\rk(N)/d-1\ge s\pibar_{q,d}(H)$.

    Conversely, let $N$ be one of the submodules in \eqref{eq_spibarqd_recall}. Let
    $f_1,\ldots,f_d\in\F_q[H]^d$ map to an $\F_q$-basis of
    $(N\cap\F_q[H]^d)\backslash\F_q[H]^d$, and let $N'$ be the kernel of the $\F_q[\FF]$-linear map
    $\psi\colon\F_q[\FF]^d\to N\backslash\F_q[\FF]^d$ defined by $e_i\mapsto f_i+N$. Every $e_j$ is
    congruent modulo $N\cap\F_q[H]^d$ to an $\F_q$-linear combination of $f_1,\ldots,f_d$, so $\psi$
    is surjective, and $N'\backslash\F_q[\FF]^d\cong N\backslash\F_q[\FF]^d$. As $N$ is finitely
    generated, the module $N\backslash\F_q[\FF]^d$ is finitely presented, so $N'$ is finitely
    generated. Since the Euler characteristic $\chi$ is well defined on finitely presented modules
    (see Section~\ref{subsec_spiK}), $\rk(N')=d-\chi(N\backslash\F_q[\FF]^d)=\rk(N)$. Moreover, the $\F_q[H]$-linear
    map $\F_q[H]^d\to(N\cap\F_q[H]^d)\backslash\F_q[H]^d$ defined by
    $e_i\mapsto f_i+(N\cap\F_q[H]^d)$ is surjective and maps $\mathcal{E}_d$ to a basis, and its kernel
    is $N'\cap\F_q[H]^d$. Hence,
    by Proposition~\ref{prop_normal_form_for_bases}, $N'\cap\F_q[H]^d=M_{\beta}$ for some
    $\beta\in\Hom(H,\GLvarFq{d})$, so $M_{\beta}^\FF\le N'$.

    Let $N''$ be a submodule $M_{\beta}^\FF\le N''\le N'$ of minimal rank, so
    $\rk(N'')\le\rk(N')=\rk(N)$, and $N''$ is finitely generated. Since a proper direct
    summand of $N''$ has smaller rank, $M_{\beta}^\FF$ is not contained in any proper direct summand
    of $N''$. By \cite[Corollary 3.10]{ernst2024word}, if $L\le L'\le\F_q[\FF]^d$ are finitely
    generated submodules, $L$ is supported on $\T_B(H)^d$, and $L$ is not contained in any proper
    direct summand of $L'$ (called a proper free factor there), then $L'$ is supported on
    $\T_B(H)^d$. Hence $N''$ is supported on $\T_B(H)^d$, as $M_{\beta}^\FF$ is. As
    $M_{\beta}\le N''\cap\F_q[H]^d\le N'\cap\F_q[H]^d=M_{\beta}$, the submodule $N''$ is in
    $\mathcal{N}_{\beta}^{\textup{inj}}$.
\end{proof}

By applying
Proposition~\ref{prop_common_fixed_mod_G_and_Com}
for the category $\catC = \textbf{FinVect}_{\F_q}$
of finite dimensional vector spaces over $\F_q$,
and using the identification
$\Grasm_d(\F^n_q) \cong
\HomCinj(\F_q^d, \F_q^n)/\GLvarFq{d}$,
we get that for a uniformly random homomorphism
$\alphaSimUHomFVar{\GLnFq}$, as $n\to\infty$,
\begin{equation*}
    \begin{split}
        \EHtoFVarActsVar{\GLnFq}{\Grasm_d(\F^n_q)}
        &= \E_{\alpha}\abs*{\braces*{\begin{array}{cc}
            \textup{common fixed points of} \\
            \alpha(H)\acts
            \HomCinj(\F_q^d, \F_q^n)/\GLvarFq{d}
       \end{array}}} \\
       &= \frac{1}{|\GLvarFq{d}|}
       \sum_{\beta\in \Hom(H, \GLvarFq{d})}
       \E_{\alpha}
       \abs*{\Interinj(\alpha\restriction_H, \beta)} \\
       \textup{(by Lemma~\ref{lemma_GLn_inter_count})}\qquad\qquad
       &= \frac{1}{|\GLvarFq{d}|}
        \sum_{\beta\in \Hom(H, \GLvarFq{d})}
        \sum_{N\in\mathcal{N}_{\beta}^{\textup{inj}}}q^{n(d-\rk(N))}\prn*{1+O(q^{-n})} \\
        \textup{(by Proposition~\ref{prop_spibarqd_as_min})}\qquad\qquad
        &= \Theta \prn*{q^{-d\cdot s\pibar_{q, d}(H)n}}\\
        \textup{(as $\abs*{\Grasm_d(\F^n_q)}=\Theta(q^{dn})$)}\qquad\qquad
        &= \Theta \prn*{\abs*{\Grasm_d(\F^n_q)}
        ^{-s\pibar_{q, d}(H)}}.
    \end{split}
\end{equation*}
This finishes the proof of Theorem~\ref{thm_spibarqd}.

%% file: puder_van_handel.tex
\section{Random words generate \texorpdfstring{$A_n$}{An}}
\label{appendix_random_words_gen_A_n}

The content of this appendix was observed by Doron Puder. 

\begin{theorem}
    Let $H\le \FF$ be a \noncyc subgroup, and 
    $\alphaSimUHomFSn$ a random homomorphism. Then
    $\Pr(\alpha(H)\supseteq A_n)\to_{n\to\infty} 1$.
\end{theorem}

\begin{proof}
Let $u, v\in H$ be non-commuting words.
Let $\rho'\colon S_n\to \textup{GL}_{(n)_6}(\Z)$ be the permutation representation given by the action of $S_n$ on tuples $(x_1,\ldots, x_6)\in [n]^6$ with distinct numbers.
Let $\rho=\rho'\restriction_{(1,\ldots,1)^{\bot}}$ be the sub-representation of vectors with sum $0$.

By \cite[Theorem 3.14]{chen2024new}, 
the random matrix
$M_n\defeq \rho(\alpha(u)) + \rho(\alpha(u^{-1})) + \rho(\alpha(v)) + \rho(\alpha(v^{-1}))$
strongly converges to the operator
$M_{\infty}\defeq 
 u + u^{-1} + v + v^{-1}\in \Z[\FF]$, and in particular,
the spectral radius of $M_n$ is, with probability $1-o(1)$ as $n\to\infty$, 
at most $2\sqrt{3}+o(1)<4$, where $2\sqrt{3}$ is the spectral radius of $M_{\infty}$
(by \cite[Theorem 3]{Kesten1959SymmetricRandomWalks}).
It follows that the Schreier graph of the action of
$S_n$ on $[n]_6$ with edges given by
$\{\Vec{x},\sigma(\Vec{x})\}$ for $\Vec{x}\in [n]^6$ 
with distinct numbers and 
$\sigma\in \{\alpha(u),\alpha(v)\}$ 
is an expander graph (as $n\to\infty$) 
with probability $1-o(1)$, and in particular connected. 
Therefore the subgroup 
$\alpha(H)\le S_n$ acts 6-transitively, 
thus by \cite[Section 7.4, page 229]{DixonMortimer1996} it contains $A_n$.
\end{proof}

If one wishes to avoid the use of the classification of finite simple groups,
which is needed for \cite[Section 7.4, page 229]{DixonMortimer1996},
replace $[n]^6$ by $[n]^{\ceil{3\log_2(n)}}$; now strong convergence 
is given by \cite[Theorem 1.10]{cassidy2024random} and containment of $A_n$ 
follows from \cite[Theorem 5.7.2]{Hall1959TheoryOfGroups}.

%% file: ta_shma_app.tex
\section{The dependence theorem (joint with Noam Ta Shma)}
\label{appendix_ta_shma}

Louder and Wilton proved the dependence theorem \cite[Theorem 1.16]{louder2022negative}, a
further strengthening of Wise's strengthened $w$-cycle conjecture
(Theorem~\ref{thm_Wilton_Louder_Helfer_Wise}).
We generalize the dependence theorem from words to finitely generated subgroups of $\FF$.

For a connected graph $\Gamma$ with an immersion $\iota_{\Gamma}\colon\Gamma\immerse\Omega$, we denote
by $\pi(\Gamma)$ the primitivity rank \eqref{eq_primitivity_rank} of $(\iota_{\Gamma})_*\pi_1(\Gamma)$
in $\pi_1(\Omega)$; the immersion will always be clear from the context.
Recall the notions of a free and of an algebraic morphism of core graphs
from the paragraph after Definition~\ref{def_algebraic_extension}.
In this language,
\begin{equation}
    \label{eqn_pi_of_immersion}
    \pi(\Gamma)=\min\braces*{1-\chi(\Xi)\,\middle|\,
    \Gamma\immerse\Xi\immerse\Omega\textup{ is a decomposition
    of }\iota_{\Gamma}\textup{, and }\Gamma\immerse\Xi\textup{ is not free}}.
\end{equation}

Recall from Definition~\ref{def_pullback} that the pullback $\Gamma\times_{\Omega}\Delta$ of two
immersions $\Gamma\immerse\Omega$ and $\Delta\immerse\Omega$ is the core of their fiber product, and
from Definition~\ref{def_stable_primitivity_rank} that $s\pi_d(H)$ is the minimum of
$-\chi(\Delta)/d$ over the connected core graphs $\Delta$ such that $\Gamma_B(H)\times_{\Omega_B}\Delta$
contains a $d$-covering of $\Gamma_B(H)$ whose projection to $\Delta$ is algebraic and not an
isomorphism.

\begin{theorem}[The dependence theorem, generalized]
    \label{thm_gen_dependence_theorem}
    For every finitely generated subgroup $H\le \FF$ with $\pi(H)\ge 2$, and every
    $d\in\N$,
    \begin{equation}
        \label{generalized_dependence_thm_inequality}
        s\pi_d(H)\ge 1+\frac{\pi(H)-2}{d}.
    \end{equation}

    More explicitly, let
    $\iota_{\Gamma}\colon\Gamma\immerse\Omega$ and
    $\iota_{\Delta}\colon\Delta\immerse\Omega$
    be immersions of core graphs, with $\Gamma$ connected,
    let $P\subseteq \Gamma\times_{\Omega}\Delta$ be a subgraph, and
    assume that
    the projection $p_{\Gamma}\colon P\to \Gamma$
    is a $d$-covering
    and the projection
    $p_{\Delta}\colon P\to \Delta$
    is algebraic and is not an isomorphism over any connected component of $\Delta$ (see the
    diagram \eqref{eqn_dependence_thm_setting}).
    If $\pi(\Gamma)\ge 2$, then
    \[ -\chi(\Delta)\ge d + \pi(\Gamma) - 2. \]
\end{theorem}

This strengthens the conclusion of
Theorem~\ref{thm_Wilton_Louder_Helfer_Wise}: the inequality
$-\chi(\Delta)\ge d$ is never an equality when $\pi(\Gamma)>2$.\footnote{Some
applications of the dependence theorem use its bound itself, and not only that $-\chi(\Delta)>d$. For example, Louder and Wilton deduced from it
\cite[Proposition 6.19]{louder2022negative} that if a word $1\ne w\in\FF$, which is not a proper
power, is contained in two subgroups $H_1,H_2\le\FF$ of rank $2$ and is primitive in neither, then
$\rk\inner{H_1,H_2}=2$.}
The case where $\Gamma$ is a cycle was proved by Louder and Wilton;
it is the main technical theorem of their paper, \cite[Theorems 1.16 and 2.21]{louder2022negative},
from which their applications follow. They stated \cite[Theorem 1.16]{louder2022negative} in the
language of graphs of groups, and Hanany and Puder, in \cite[Remark 7.10]{HP22},\footnote{There it
is Theorem 1.11, its number in the first arXiv version.} translated the theorem into the language of
core graphs, which we use. The proof of Louder and Wilton, however, is written in a completely different language,
using homological algebra and ideas reminiscent of Morse theory. We found that the essence of their proof can be extracted into a
definition of a $\Gamma$-polymatroid $\hp$; then Theorem~\ref{thm_Gamma_poly_subgrp} implies
Theorem~\ref{thm_gen_dependence_theorem} when $\chi(\Gamma)<0$. The
$\Gamma$-polymatroid $\hp$ we obtain is not necessarily vertex-tight, so to prove the case where
$\Gamma$ is a cycle, one needs also \cite[Lemma 4.8]{louder2022negative}, the \enquote{up-down lemma}.

We denote by $C(\Gamma)\defeq V(\Gamma)\uplus E(\Gamma)$ the set of \textbf{cells} of a graph
$\Gamma$. In this appendix, we use at most one morphism from any graph to another, so we often
denote it by an arrow between them, such as $P\to\Gamma$ for $p_{\Gamma}$.

\begin{definition}
    \label{def_core_graph_pushout}
    Let the morphisms $p_{i}\colon P\to \Gamma_i$ and the immersions
    $\iota_i\colon\Gamma_i\immerse\Omega$, for $i=1,2$, satisfy $\iota_1\circ p_1=\iota_2\circ p_2$.
    The \textbf{pushout} $\Gamma_1\sqcup_P \Gamma_2$ of $p_1,p_2$ \emph{in the category of
    graphs} is obtained from the disjoint union $\Gamma_1\sqcup\Gamma_2$ by gluing $p_1(r)$
    to $p_2(r)$ for every cell $r$ of $P$; that is,
    \begin{equation}
        \label{eqn_pushout_as_quotient}
        C\prn*{\Gamma_1\sqcup_P \Gamma_2}
        = \prn*{C(\Gamma_1)\uplus C(\Gamma_2)}/\!\sim,
        \quad
        \sim\textup{ generated by } p_{1}(r)\sim p_{2}(r)\textup{ for } r\in C(P),
    \end{equation}
    and the morphisms $\Gamma_i\to\Gamma_1\sqcup_P \Gamma_2$ are the restrictions of the
    quotient map. As $\iota_1\circ p_1=\iota_2\circ p_2$, the quotient is a graph
    (Definition~\ref{def_serre_graph}) with a morphism to $\Omega$.

    The \textbf{pushout over $\Omega$}
     is the pushout $\Gamma_1\sqcup^{\Omega}_P \Gamma_2$
     of $p_1,p_2$ \emph{in the category of
    graphs immersing into $\Omega$}, whose morphisms are the morphisms over $\Omega$
    (Definition~\ref{def_morphism_of_immersions}), that is, the morphisms that commute with the
    immersions into $\Omega$. It is obtained from $\Gamma_1\sqcup_P \Gamma_2$
    by folding with respect to $\Omega$: repeatedly identifying two distinct edges
    that have a common image in $\Omega$ and share their source or their target, until the
    morphism to $\Omega$ becomes an immersion.
\end{definition}

Each of the two pushouts has the universal property of a pushout in its own category. In the
category of graphs immersing into $\Omega$, this property is the following: for every graph $\Pi$ immersing into
$\Omega$, and every two morphisms $\Gamma_1\to\Pi$ and $\Gamma_2\to\Pi$ over $\Omega$ whose
compositions with $p_1$ and with $p_2$ are equal, there is a unique morphism
$\Gamma_1\sqcup^{\Omega}_P\Gamma_2\to\Pi$ (the dashed arrow) such that the diagram
\eqref{eqn_pushout_universal_property} commutes. The diagram \eqref{eqn_dependence_thm_setting}
illustrates the setting of Theorem~\ref{thm_gen_dependence_theorem}.
\noindent
\begin{minipage}{0.36\linewidth}
\begin{equation}
    \label{eqn_dependence_thm_setting}
    \begin{tikzcd}
	P & \Gamma \\
	\Delta & \Omega
	\arrow["{p_\Gamma}", from=1-1, to=1-2]
	\arrow["{p_\Delta}"', from=1-1, to=2-1]
	\arrow["{\iota_\Gamma}", from=1-2, to=2-2]
	\arrow["{\iota_\Delta}"', from=2-1, to=2-2]
    \end{tikzcd}
\end{equation}
\end{minipage}\hfill
\begin{minipage}{0.6\linewidth}
\begin{equation}
    \label{eqn_pushout_universal_property}
    \begin{tikzcd}
	P && {\Gamma_2} & \\
	{\Gamma_1} && {\Gamma_1\underset{P}{\sqcup^{\Omega}} \Gamma_2} \\
	&&& \Pi
	\arrow["{p_2}", from=1-1, to=1-3]
	\arrow["{p_1}"', from=1-1, to=2-1]
	\arrow[from=1-3, to=2-3]
	\arrow[from=1-3, to=3-4]
	\arrow[from=2-1, to=2-3]
	\arrow[from=2-1, to=3-4]
	\arrow["\lrcorner"{anchor=center, pos=0.125, rotate=180}, draw=none, from=2-3, to=1-1]
	\arrow[dashed, from=2-3, to=3-4]
    \end{tikzcd}
\end{equation}
\end{minipage}

Every fold removes one edge and at most one vertex, so
\begin{equation}
    \label{eqn_folding_increases_ec}
    \chi\prn*{\Gamma_1\sqcup_P \Gamma_2}\le\chi\prn*{\Gamma_1\sqcup^{\Omega}_P \Gamma_2}.
\end{equation}

For $i=1,2$, the morphisms $\Gamma_i\to\Gamma_1\sqcup_P\Gamma_2$ and
$\Gamma_i\to\Gamma_1\sqcup^{\Omega}_P\Gamma_2$ are immersions, as $\iota_i$ factors through them.
Hence, if $\Gamma_1$ and $\Gamma_2$ are core graphs, then so are $\Gamma_1\sqcup_P\Gamma_2$ and
$\Gamma_1\sqcup^{\Omega}_P\Gamma_2$, as each of them is the union of the images of $\Gamma_1$ and
$\Gamma_2$.

The results in \cite[Section 4]{HP22}, stated there for graphs immersing into a bouquet, hold for
graphs immersing into any graph $\Omega$ and the morphisms over $\Omega$ between them, with the same
proofs. In particular, by \cite[Theorem 4.9]{HP22},
every such morphism $\eta$ between core graphs has a decomposition
$\eta=\eta_2\circ\eta_1$ over $\Omega$, unique up to isomorphism, with $\eta_1$ algebraic and
$\eta_2$ free: the \textbf{algebraic--free decomposition} of $\eta$. By \cite[Lemma~4.5]{HP22}, the
pullback of a free morphism is free; the dual property, Proposition~\ref{prop_pushout_preserves_alg},
is that the pushout of an algebraic morphism is algebraic. It is a technical step, which we use only
in Lemma~\ref{lemma_pushout_pi_bound}:

\begin{proposition}
\label{prop_pushout_preserves_alg}
    Let $P$, $\Gamma_1$ and $\Gamma_2$ be core graphs immersing into a graph $\Omega$, and let
    $p_1\colon P\to\Gamma_1$ and $p_2\colon P\to\Gamma_2$ be morphisms over $\Omega$. If $p_2$ is
    algebraic, then the morphism $\Gamma_1\to\Gamma_1\sqcup^{\Omega}_P\Gamma_2$ is algebraic.
\end{proposition}

\begin{proof}
    Denote $\Lambda\defeq\Gamma_1\sqcup^{\Omega}_P\Gamma_2$, let $\Gamma_1\to\Gamma_1'\to\Lambda$ be the
    algebraic--free decomposition of $\Gamma_1\to\Lambda$, and let $p_1'$ and $p_2'$ be the projections
    to $\Gamma_1'$ and to $\Gamma_2$ from the pullback $\Gamma_1'\times_{\Lambda}\Gamma_2$:
    \begin{equation}
        \label{eqn_pushout_alg_proof}
    \begin{tikzcd}
	P & {\Gamma_1'\times_{\Lambda}\Gamma_2} & {\Gamma_2} & \\
	{\Gamma_1} & {\Gamma_1'} & \Lambda & \Omega
	\arrow[from=1-1, to=1-2]
	\arrow["{p_1}", from=1-1, to=2-1]
	\arrow["{p_2'}", "{\textrm{free}}"', from=1-2, to=1-3]
	\arrow["{p_1'}", from=1-2, to=2-2]
	\arrow[curve={height=18pt}, dashed, from=1-3, to=1-2]
	\arrow[from=1-3, to=2-3]
	\arrow["{\textrm{alg}}", from=2-1, to=2-2]
	\arrow["{\textrm{free}}", from=2-2, to=2-3]
	\arrow[curve={height=-12pt}, dashed, from=2-3, to=2-2]
	\arrow[from=2-3, to=2-4]
    \end{tikzcd}
    \end{equation}
    By \cite[Lemma~4.5]{HP22}, $p_2'$ is free, as $\Gamma_1'\to\Lambda$ is free. As $P$ is a core
    graph, by the universal property of the pullback \cite[Section 4.1]{HP22}, there is a morphism
    $P\to\Gamma_1'\times_{\Lambda}\Gamma_2$ whose compositions with $p_1'$ and $p_2'$ are
    $P\to\Gamma_1\to\Gamma_1'$ and $p_2$. As $p_2$ is algebraic and factors through $p_2'$, the
    projection $p_2'$ is an isomorphism; its inverse is the upper dashed arrow in
    \eqref{eqn_pushout_alg_proof}. Hence, by the universal property of the pushout $\Lambda$ over
    $\Omega$ for $\Pi=\Gamma_1'$, there is a morphism $\Lambda\to\Gamma_1'$, the lower dashed arrow
    in \eqref{eqn_pushout_alg_proof}, whose compositions with $\Gamma_1\to\Lambda$ and with
    $\Gamma_2\to\Lambda$ are $\Gamma_1\to\Gamma_1'$ and $p_1'\circ\prn*{p_2'}^{-1}$. By the
    uniqueness in this universal property for $\Pi=\Lambda$, the composition
    $\Lambda\to\Gamma_1'\to\Lambda$ is the identity, so $\Lambda\to\Gamma_1'$ is injective, and hence,
    by \cite[Proposition 4.3(1)]{HP22}, free. As the algebraic morphism $\Gamma_1\to\Gamma_1'$
    factors through $\Lambda\to\Gamma_1'$, the morphism $\Lambda\to\Gamma_1'$ is an isomorphism, so
    $\Gamma_1\to\Lambda$ is algebraic.
\end{proof}

For the rest of the appendix, fix immersions of core graphs $\iota_{\Gamma}\colon\Gamma\immerse\Omega$
and $\iota_{\Delta}\colon\Delta\immerse\Omega$, with $\Gamma$ connected. Fix a subgraph $P$ of $\Gamma\times_{\Omega}\Delta$, with the
projections $p_{\Gamma}\colon P\to\Gamma$ and $p_{\Delta}\colon P\to\Delta$, and an integer $d\ge1$
such that $p_{\Gamma}$ is a $d$-covering (in particular, $P$ is a core graph) and $p_{\Delta}$ is
algebraic.
Let $\Sigma\defeq\Gamma\sqcup_{P}\Delta$ and $\Sigma^{\Omega}\defeq\Gamma\sqcup^{\Omega}_{P}\Delta$
be the pushouts of $p_{\Gamma}$ and $p_{\Delta}$ in the category of graphs and over $\Omega$; thus
$\Sigma^{\Omega}$ is obtained from $\Sigma$ by folding:
\begin{equation}
    \label{eqn_appB_morphisms}
\begin{tikzcd}[ampersand replacement=\&]
	P \& \Gamma \\
	\Delta \& \Sigma \& {\Sigma^{\Omega}} \& \Omega
	\arrow["{p_{\Gamma}}", from=1-1, to=1-2]
	\arrow["{p_{\Delta}}"', from=1-1, to=2-1]
	\arrow[from=1-2, to=2-2]
	\arrow["{\iota_{\Gamma}}", curve={height=-18pt}, from=1-2, to=2-4]
	\arrow[from=2-1, to=2-2]
	\arrow["{\iota_{\Delta}}"', curve={height=18pt}, from=2-1, to=2-4]
	\arrow[from=2-2, to=2-3]
	\arrow[from=2-3, to=2-4]
\end{tikzcd}
\end{equation}

We prove Theorem~\ref{thm_gen_dependence_theorem} in two parts: the following lemma, about
$\Sigma^{\Omega}$, and the essence of the proof, about polymatroids, which uses only $\Sigma$.

\begin{lemma}
    \label{lemma_pushout_pi_bound}
    If $p_{\Delta}$ is not an isomorphism, then
    $\chi(\Sigma)\le\chi\prn*{\Sigma^{\Omega}}\le 1-\pi(\Gamma)$.
\end{lemma}

\begin{proof}
    The first inequality is \eqref{eqn_folding_increases_ec}.
    By \cite[Proposition 4.3(1)]{HP22}, $p_{\Delta}$ is not injective, as it is algebraic and not an
    isomorphism. Hence, as a cell of $P\subseteq\Gamma\times_{\Omega}\Delta$ is determined by its
    images in $\Gamma$ and in $\Delta$, two distinct cells of $\Gamma$ are glued in $\Sigma$, so
    $\Gamma\to\Sigma^{\Omega}$ is not injective. By Proposition~\ref{prop_pushout_preserves_alg} for
    $p_1=p_{\Gamma}$ and $p_2=p_{\Delta}$, the morphism $\Gamma\to\Sigma^{\Omega}$ is algebraic and not
    an isomorphism, so it is not free. Therefore, by \eqref{eqn_pi_of_immersion},
    $\pi(\Gamma)\le 1-\chi\prn*{\Sigma^{\Omega}}$.
\end{proof}

We now turn to the second part, about polymatroids, which is the main and new part of the proof.
For a cell $b$ of $\Sigma$, we denote by $\Gamma_b$, $\Delta_b$ and $P_b$ the sets of cells of
$\Gamma$, of $\Delta$ and of $P$, respectively, that are mapped to $b$.
For every
$b\in C(\Sigma)$, we define a graph $W^b$, and a polymatroid $\hp^b$ on $\Gamma_b$ that essentially
counts connected components of subgraphs of $W^b$ (see
Definition~\ref{def_polymatroid_dependence_thm}); together, these graphs $W^b$ form, essentially,
the \enquote{horizontal} graph of graphs defined in
\cite[Section 2.3]{louder2022negative}. Then we
prove in Propositions~\ref{prop_hp_is_Gamma_polymatroid}, \ref{prop_euler_char_of_hp}
and~\ref{prop_hp_of_single_edge} below that $\hp$ is a $\Gamma$-polymatroid over $\Sigma$, with
$\chi(\hp)=\chi(\Delta)-\chi(\Sigma)$ and singleton values $\hp(\{e\})=d-1$, and apply
Theorem~\ref{thm_Gamma_poly_subgrp}.

\begin{definition}
    \label{def_polymatroid_dependence_thm}
    For $b\in C(\Sigma)$, let $W^b$ be the bipartite graph with vertex set
    $\Gamma_b\uplus\Delta_b$, and with an edge $\prn*{p_{\Gamma}(r), p_{\Delta}(r)}$ for every
    $r\in P_b$.
    For $U\subseteq\Gamma_b$, let $W^b_U$ be the subgraph of $W^b$ induced on
    $U\uplus\Delta_b$, and define
    \[ \hp^b(U)\defeq |\Delta_b| - |\{\textup{connected components of }W^b_U\}|. \]
    We denote by $\hp$ the collection of the functions $\hp^b$ for $b\in C(\Sigma)$.
\end{definition}

\begin{lemma}
    \label{lemma_W_b_connected}
    For every $b\in C(\Sigma)$, the graph $W^b$ is connected. In particular,
    $\hp^b(\Gamma_b)=|\Delta_b|-1$.
\end{lemma}

\begin{proof}
    Let $W$ be the bipartite graph with vertex set $C(\Gamma)\uplus C(\Delta)$ and with an edge
    $\prn*{p_{\Gamma}(r),p_{\Delta}(r)}$ for every $r\in C(P)$.
    By \eqref{eqn_pushout_as_quotient}, the relation $\sim$ that defines the cells of $\Sigma$ is
    generated by the edges of $W$, so two cells of $\Gamma$ or of $\Delta$ are mapped to the same
    cell of $\Sigma$ if and only if they are in the same connected component of $W$. Hence the
    vertex sets of the connected components of $W$ are the sets $\Gamma_b\uplus\Delta_b$ for
    $b\in C(\Sigma)$.
    An edge $\prn*{p_{\Gamma}(r),p_{\Delta}(r)}$ of $W$ has its endpoints in
    $\Gamma_b\uplus\Delta_b$ if and only if $r\in P_b$, so the connected component of $W$ with
    vertex set $\Gamma_b\uplus\Delta_b$ is $W^b$.
    Hence $W^b_{\Gamma_b}=W^b$ is connected, so $\hp^b(\Gamma_b)=|\Delta_b|-1$.
\end{proof}

\newcommand{\partition}[2]{\text{Part}_{#1, #2}}
\newcommand{\srcp}[1]{\src \prn*{ #1 }}

For $b\in C(\Sigma)$, $u\in\Gamma_b$ and $U\subseteq\Gamma_b$, we denote by
$N(u)\subseteq\Delta_b$ the set of neighbors of $u$ in $W^b$, and by $\partition{u}{U}$ the
partition of $N(u)$ induced by the connected components of $W^b_U$. As $p_{\Gamma}$ is
surjective, $N(u)\ne\emptyset$. If $u\notin U$, then
$W^b_{U\cup\{u\}}$ is obtained from $W^b_U$ by adding the vertex $u$ with its edges to $N(u)$, so
in $W^b_{U\cup\{u\}}$, the vertex $u$ and the connected components of $W^b_U$ that meet $N(u)$
form a single connected component. These components of $W^b_U$ correspond to the blocks of
$\partition{u}{U}$, so
\begin{equation}
    \label{eqn_marginal_gain_of_hp}
    \hp^b\prn*{U\cup\braces*{u}}-\hp^b(U)=\abs*{\partition{u}{U}}-1.
\end{equation}
If $u\in U$, both sides of \eqref{eqn_marginal_gain_of_hp} are $0$.

\begin{proposition}
    \label{prop_hp_b_is_polymatroid}
    For every $b \in C(\Sigma)$, the function $\hp^b$ is a polymatroid.
\end{proposition}
\begin{proof}
    By Observation~\ref{obs_diminishing_returns}, it suffices to show that $\hp^b(\emptyset)=0$
    and the two inequalities of \eqref{eqn_diminishing_returns}, which say that the increments
    $\hp^b(U\cup\{u\})-\hp^b(U)$ are non-negative and do not increase when $U$ grows.
    The graph $W^b_{\emptyset}$ consists of the $|\Delta_b|$ isolated vertices of $\Delta_b$,
    so $\hp^b(\emptyset)=0$.
    Let $u\in \Gamma_b$ and $U\subseteq U'\subseteq\Gamma_b$. As $N(u)\neq\emptyset$, by
    \eqref{eqn_marginal_gain_of_hp}, $\hp^b(U\cup\{u\})-\hp^b(U)=\abs*{\partition{u}{U}}-1\ge0$. As $W^b_{U}\subseteq W^b_{U'}$, the partition
    $\partition{u}{U}$ refines $\partition{u}{U'}$, so
    $\abs*{\partition{u}{U}}\ge\abs*{\partition{u}{U'}}$, and, by \eqref{eqn_marginal_gain_of_hp},
    $\hp^b(U\cup\{u\})-\hp^b(U)\ge\hp^b(U'\cup\{u\})-\hp^b(U')$.
\end{proof}

\begin{proposition}
    \label{prop_hp_is_Gamma_polymatroid}
    The collection $\hp$ is a $\Gamma$-polymatroid over $\Sigma$.
\end{proposition}
\begin{proof}
    By Proposition~\ref{prop_hp_b_is_polymatroid}, the functions $\hp^b$ are polymatroids.
    By Definition~\ref{def_eta_polymatroid}, it remains to prove, for every $b\in E(\Sigma)$,
    that $\src\colon \Gamma_b\to \Gamma_{\srcp{b}}$ is a morphism of polymatroids, that is,
    \[ \hp^b(U') - \hp^b(U) \ge \hp^{\srcp{b}}(\src(U')) - \hp^{\srcp{b}}(\src(U))
        \qquad\textup{for all } U\subseteq U'\subseteq \Gamma_b, \]
    and similarly for $\tar$; we treat $\src$ only.
    It suffices to prove the case $U'=U\cup\braces*{e}$; namely, for $e\in\Gamma_b$ and
    $U\subseteq\Gamma_b$,
    \begin{equation}
        \label{eqn_src_morphism_single}
        \hp^b \prn*{ U \cup \braces*{ e }} - \hp^b \prn*{ U } \ge
        \hp^{\srcp{b}} \prn*{ \srcp{ U } \cup \braces*{ \srcp{ e } }} - \hp^{\srcp{b}} \prn*{ \srcp { U } }.
    \end{equation}
    By \eqref{eqn_marginal_gain_of_hp}, the inequality \eqref{eqn_src_morphism_single} is equivalent
    to
    \[ \abs*{\partition{e}{U}}\ge\abs*{\partition{\srcp{ e }}{\srcp { U }}}. \]
    The map $\src$ maps the vertices of $W^b$ to vertices of $W^{\srcp{b}}$, and the edge
    $\prn*{p_{\Gamma}(r),p_{\Delta}(r)}$ of $W^b$ to the edge
    $\prn*{\srcp{p_{\Gamma}(r)},\srcp{p_{\Delta}(r)}}=\prn*{p_{\Gamma}(\srcp{r}),p_{\Delta}(\srcp{r})}$
    of $W^{\srcp{b}}$. Hence it maps $N\prn*{e}$ into $N\prn*{\srcp{e}}$, and as it maps $U$ into
    $\srcp{U}$, it maps every connected component of $W^b_U$ into a connected component of
    $W^{\srcp{b}}_{\srcp{U}}$. We claim that the map $\src\colon N\prn*{e}\to N\prn*{\srcp{e}}$ is
    surjective; then every block of $\partition{\srcp{ e }}{\srcp { U }}$ contains the
    image of a block of $\partition{e}{U}$, so
    $\abs*{\partition{e}{U}}\ge\abs*{\partition{\srcp{ e }}{\srcp { U }}}$.
    Indeed, let $v \in N \prn*{ \srcp { e } }$. Then there is a vertex $v' \in P_{\srcp{b}}$
    with $p_{\Delta} \prn*{ v' } = v$ and $p_{\Gamma} \prn*{ v' } = \srcp{ e }$, and since
    $p_{\Gamma}$ is a covering, there is an edge $e' \in P_b$ with $p_{\Gamma}\prn*{e'}=e$ and
    $\srcp{ e' } = v'$. Then $p_{\Delta}\prn*{e'}\in N\prn*{e}$ and
    $\srcp{p_{\Delta}\prn*{e'}}=p_{\Delta}\prn*{v'}=v$.
\end{proof}

The following proposition is essentially \cite[Lemmas 4.3 and 4.4]{louder2022negative}, which
Louder and Wilton also derived from the connectedness of the graphs $W^b$.

\begin{proposition}
    \label{prop_euler_char_of_hp}
    The Euler characteristic of $\hp$ is $\chi(\hp) = \chi(\Delta) - \chi(\Sigma)$.
\end{proposition}
\begin{proof}
    By Definition~\ref{def_eta_polymatroid_euler_char} of $\chi(\hp)$ and by
    Lemma~\ref{lemma_W_b_connected},
    \[ \chi(\hp)
        = \sum_{b\in V(\Sigma)} \hp^b(\Gamma_b) - \sum_{b\in E(\Sigma)} \hp^b(\Gamma_b)
        = \sum_{b\in V(\Sigma)} \prn*{|\Delta_b| - 1} - \sum_{b\in E(\Sigma)} \prn*{|\Delta_b| - 1}. \]
    Hence, as $V(\Delta)=\biguplus_{b\in V(\Sigma)}\Delta_b$ and
    $E(\Delta)=\biguplus_{b\in E(\Sigma)}\Delta_b$,
    \[ \chi(\hp) = \prn*{|V(\Delta)| - |V(\Sigma)|} - \prn*{|E(\Delta)| - |E(\Sigma)|}
        = \chi(\Delta) - \chi(\Sigma). \qedhere \]
\end{proof}

\begin{proposition}
    \label{prop_hp_of_single_edge}
    Let $e$ be an edge of $\Gamma$, with image $b$ in $\Sigma$. Then $\hp^{b}(\{e\}) = d-1$.
\end{proposition}

\begin{proof}
    Every cell $r$ of $P$ with $p_{\Gamma}(r)=e$ is mapped to $b$, so
    $N(e) = p_{\Delta}\prn*{p_{\Gamma}^{-1}(e)}$.
    The graph $W^b_{\emptyset}$ consists of the $|\Delta_b|$ isolated vertices of $\Delta_b$, so
    $\hp^b(\emptyset)=0$ and the blocks of $\partition{e}{\emptyset}$ are singletons. Hence, by
    \eqref{eqn_marginal_gain_of_hp} for $U=\emptyset$, we have $\hp^b(\{e\})=|N(e)|-1$.
    Finally, a cell $r$ of $P\subseteq \Gamma\times_{\Omega}\Delta$ is determined by the pair
    $\prn*{p_{\Gamma}(r),p_{\Delta}(r)}$, so $p_{\Delta}$ is injective on
    $p_{\Gamma}^{-1}(e)$, and as $p_{\Gamma}$ is a $d$-covering,
    \[ |N(e)| = \abs*{p_{\Gamma}^{-1}(e)} = d. \qedhere \]
\end{proof}

\begin{proposition}
    \label{prop_pushout_euler_char_bound}
    Assume that $\chi(\Gamma)<0$. Then $-\chi(\Delta)\ge -\chi(\Sigma) + d - 1$.
\end{proposition}

\begin{proof}
    By Proposition~\ref{prop_hp_is_Gamma_polymatroid}, $\hp$ is a $\Gamma$-polymatroid over
    $\Sigma$, and $\Gamma\to\Sigma$ is an immersion between core graphs. Hence, by
    Theorem~\ref{thm_Gamma_poly_subgrp}, there is an edge $e$ of $\Gamma$ with
    $\chi(\hp)\le -\hp(\{e\})$.
    By Propositions~\ref{prop_euler_char_of_hp} and~\ref{prop_hp_of_single_edge},
    $\chi(\hp)=\chi(\Delta)-\chi(\Sigma)$ and $\hp(\{e\})=d-1$, so
    $\chi(\Delta)-\chi(\Sigma)\le 1-d$.
\end{proof}

\begin{proof}[Proof of Theorem~\ref{thm_gen_dependence_theorem} when $\chi(\Gamma)<0$]
    By Proposition~\ref{prop_pushout_euler_char_bound} and Lemma~\ref{lemma_pushout_pi_bound},
    \[ -\chi(\Delta)\ge -\chi(\Sigma) + d - 1\ge \pi(\Gamma) - 1 + d - 1
        = d+\pi(\Gamma)-2. \qedhere \]
\end{proof}

If $p_{\Delta}$ is an isomorphism, then $-\chi(\Delta)$ may be smaller than
$d+\pi(\Gamma)-2$. For example, assume that $B=\{x,y,z\}$, let $\Omega=\Omega_B$, and let
$P=\Gamma=\Delta=\Gamma_B(H)$ for $H\defeq\inner{[x,y],z}$, where $P$ is the diagonal of
$\Gamma\times_{\Omega}\Delta$. Then $d=1$ and $-\chi(\Delta)=1$, while
$\pi(\Gamma)=\pi(H)=\pi([x,y])+1=3$.

\begin{remark}
    \label{remark_cycle_case}
    If $\Gamma$ is a cycle, then the proof of Theorem~\ref{thm_gen_dependence_theorem} above applies
    with \cite[Theorem 2.21]{louder2022negative} in place of
    Proposition~\ref{prop_pushout_euler_char_bound}. Indeed, by
    \cite[Theorem 2.21, Lemma 2.5, Definitions 2.16 and 2.17]{louder2022negative}, in our
    notation,\footnote{In \cite[Sections 2.4 and 4.3]{louder2022negative}, our $\Gamma$, $\Delta$,
    $p_{\Gamma}$, $p_{\Delta}$, $\Sigma$, $\Sigma^{\Omega}$, $d$ and $W^b$ are $S$, $\Gamma$,
    $\sigma$, $\lambda$, $\Gamma_W$, $\Gamma_W^I$, $\deg(\sigma)$ and $W_x$.} the inequality of
    Proposition~\ref{prop_pushout_euler_char_bound} holds for a cycle $\Gamma$ under three
    conditions: $P\subseteq\Gamma\times_{\Omega}\Delta$; $\iota_{\Gamma}$ does not factor through a
    covering of degree at least $2$ (that is, $\pi(\Gamma)>1$); and $p_{\Delta}$ covers every
    edge of $\Delta$ at least twice (which holds by \cite[Remark 7.9]{HP22}, with the same proof over
    any $\Omega$).
\end{remark}

%% file: appendix_jz.tex
\section{Stability of the compressed rank}
\label{appendix_jaikin}

Theorem~\ref{thm_jz_stability} below, about the stability of the compressed rank, is due to Andrei
Jaikin-Zapirain, who kindly allowed us to include it here. It was conjectured in
\cite{shomroniProbabilisticHNC} and stated in \eqref{eq_jz_stability_intro}: $\spivarbar{d}(H)=\pibar(H)-1$ for every
finitely generated $1\ne H\le\FF$ and $d\in\N$.

Following Ivanov \cite[Section 1]{ivanov2018intersection}, a finitely generated $H\le\FF$ is called
\textbf{strongly inert} if, for every non-trivial finitely generated $K\le\FF$,
\begin{equation}\label{eq_strongly_inert}
  \sum_{HgK\,:\,H\cap gKg^{-1}\ne1}\prn*{\rk(H\cap gKg^{-1})-1}\ \le\ \rk(K)-1.
\end{equation}

Let $\calD$ be the universal division ring of fractions of $\Q[\FF]$:\footnote{A division ring, also
called a skew field, is a non-trivial ring in which every non-zero
element has a two-sided inverse. Linear algebra over a division ring is essentially the same as
over a field, just without commutativity. In particular, by
\cite[Theorems IV.2.4 and IV.2.7]{hungerford1974algebra}, every left $\calD$-module $M$ is free, and
all its bases have the same size, called the dimension $\dim_{\calD}M$.} a division ring that contains
$\Q[\FF]$, is generated by it as a division ring, and is universal: for every ring homomorphism from
$\Q[\FF]$ to a division ring $\calD'$ and every finitely presented left $\Q[\FF]$-module $M$, we have
$\dim_{\calD}(\calD\otimes_{\Q[\FF]}M)\le\dim_{\calD'}(\calD'\otimes_{\Q[\FF]}M)$. By
\cite[Section 2.5]{jaikin2024free}, $\calD$ exists and is unique up to an isomorphism fixing
$\Q[\FF]$.
For a finitely generated $H\le\FF$, let $I_H$ denote the augmentation ideal of $\Q[H]$, the
kernel of the map $\Q[H]\to\Q$ with $h\mapsto1$ for $h\in H$, and let
$\kappa_H\colon\calD\otimes_{\Q[H]}I_H\to\calD\otimes_{\Q[\FF]}I_{\FF}$ be the $\calD$-linear map
induced by the inclusion $I_H\subseteq I_{\FF}$. Following \cite[Section 2.8]{jaikin2024free}, $H$
is called \textbf{$L^2$-independent} if $\kappa_H$ is injective.

\begin{theorem}[Jaikin-Zapirain, {\cite[Corollary 3.13]{jaikin2024free}}]
\label{thm_jz_equivalence}
A finitely generated $H\le\FF$ is compressed if and only if it is strongly inert, if and only if
it is $L^2$-independent.
\end{theorem}

By \cite[Theorem 3.12]{jaikin2024free} with the field $\Q$, we have
$\pibar(H)=\rk(H)-\dim_{\calD}\ker(\kappa_H)$.\footnote{In \cite[Theorem 3.12]{jaikin2024free}, $\dim_{\calD}\ker(\kappa_H)$ is
written as an $L^2$-Betti number, defined in \cite[Section 2.5]{jaikin2024free}. Since $I_{\FF}$ is a
free $\Q[\FF]$-module, by the long exact sequence of Tor, this $L^2$-Betti number equals
$\dim_{\calD}\ker(\kappa_H)$.} As $I_H$ is a free
left $\Q[H]$-module of rank $\rk(H)$, we have $\dim_{\calD}(\calD\otimes_{\Q[H]}I_H)=\rk(H)$, so
\begin{equation}\label{eq_dim_image}
  \dim_{\calD} \Img(\kappa_H) = \pibar(H) .
\end{equation}

\begin{lemma}[Jaikin-Zapirain]\label{lemma_jz_claim}
For every finite generating set $S$ of a subgroup $H\le\FF$, the group $H$ contains a compressed subgroup
$H_0$ with $\rk(H_0)=\pibar(H)$, generated by a subset of $S$.
\end{lemma}

\begin{proof}
As $I_H$ is generated as a left $\Q[H]$-module by the
elements $h-1$ for $h\in S$, the image of $\kappa_H$ is spanned over $\calD$ by the elements
$1\otimes(h-1)$ for $h\in S$; let $S_0\subseteq S$ be a subset for which these elements, for
$h\in S_0$, form a $\calD$-basis of $\Img(\kappa_H)$, and let $H_0\defeq\inner{S_0}$. As $S_0$
generates $H_0$, the elements $1\otimes(h-1)$ for $h\in S_0$ span $\Img(\kappa_{H_0})$ too, so they
form a $\calD$-basis of it, and, by \eqref{eq_dim_image} for $H$ and for $H_0$, we have
$\pibar(H)=|S_0|=\pibar(H_0)$. Moreover, by the definition \eqref{eq_compressed_rank},
$\pibar(H_0)\le\rk(H_0)$, and $\rk(H_0)\le|S_0|$. Hence $H_0$ is compressed:
$\rk(H_0)=\pibar(H_0)=\pibar(H)$.
\end{proof}

We do not use the next corollary, which is included for its own interest; compare it
with \eqref{eq_compressed_rank}.

\begin{corollary}\label{cor_pibar_as_max}
For every finitely generated $H \le \FF$,
\[
  \pibar(H)=\max\braces*{\rk(H_0)\,\middle|\,H_0\le H\textup{ and }H_0\textup{ is compressed}}.
\]
\end{corollary}

\begin{proof}
By Lemma~\ref{lemma_jz_claim}, $\le$ holds. Conversely, if $H_0 \le H$ is compressed and
$J\ge H$ has $\rk(J)=\pibar(H)$, then $\rk(H_0)\le\rk(J)=\pibar(H)$, as $H_0\le H\le J$ and $H_0$ is
compressed.
\end{proof}

If $1\ne H\le\FF$ is strongly inert, then
\begin{equation}\label{eq_strongly_inert_graphs}
  -\chi(\Gamma_B(H) \times_{\Omega_B} \Delta) \ \le \ -\chi(\Delta)
  \qquad \text{for every core graph } \Delta .
\end{equation}
Indeed, both sides are additive over the components of $\Delta$. For connected $\Delta$, let
$K\defeq\pi_1(\Delta)$; the components of the pullback correspond to the double cosets $HgK$ with
$H\cap gKg^{-1}\ne1$, and the component corresponding to $HgK$ has fundamental group
$H\cap gKg^{-1}$.

\begin{theorem}[Jaikin-Zapirain]\label{thm_jz_stability}
For every finitely generated $1\ne H\le\FF$ and $d\in\N$, we have
$\spivarbar{d}(H)=\pibar(H)-1$. More explicitly, if $\Delta$ is a core graph such that
$\Gamma_B(H)\times_{\Omega_B}\Delta$ contains a $d$-covering of $\Gamma_B(H)$, then
$-\chi(\Delta)\ge d\cdot(\pibar(H)-1)$, with equality when $\Delta$ consists of $d$ disjoint copies
of $\Gamma_B(J)$ for some $J\ge H$ with $\rk(J)=\pibar(H)$.
\end{theorem}

\begin{proof}
Let $\Gamma\defeq\Gamma_B(H)$. Replacing $H$ by a conjugate, which changes neither $\Gamma$ nor
$\pibar(H)$, we may assume that $H=\pi_1(\Gamma,v)$ for a vertex $v$ of $\Gamma$. By
Lemma~\ref{lemma_jz_claim} for a free basis of $H$ determined by a
spanning tree of $\Gamma$, the group $H$ contains a compressed subgroup $H_0$ with $\rk(H_0)=\pibar(H)$,
generated by a subset of this basis, so $\Gamma_0\defeq\Gamma_B(H_0)$ is a subgraph of $\Gamma$. By
Theorem~\ref{thm_jz_equivalence}, $H_0$ is strongly inert. Let $P\subseteq\Gamma\times_{\Omega_B}\Delta$ be a
subgraph such that the projection $p_\Gamma\colon P\to\Gamma$ is a $d$-covering. Then
$P_0\defeq p_\Gamma^{-1}(\Gamma_0)$ is a $d$-covering of $\Gamma_0$, contained in the fiber product of
$\Gamma_0$ and $\Delta$ over $\Omega_B$. As every vertex of $P_0$ has degree at least $2$, the graph $P_0$ is
contained in the pullback $\Gamma_0\times_{\Omega_B}\Delta$. Hence, by \eqref{eq_strongly_inert_graphs} for
$H_0$,
\[
  -\chi(\Delta)
  \ \ge \ -\chi(\Gamma_0 \times_{\Omega_B} \Delta)
  \ \overset{(\dagger)}{\ge} \ -\chi(P_0)
  \ = \ d \cdot (-\chi(\Gamma_0))
  \ = \ d \cdot (\pibar(H)-1) .
\]
We now explain $(\dagger)$: every vertex of $\Gamma_0\times_{\Omega_B}\Delta$ outside $P_0$ has degree at
least $2$, and all its edges are outside $P_0$, so, counting edge-ends,
$\abs{E(\Gamma_0\times_{\Omega_B}\Delta)\setminus E(P_0)}\ge\abs{V(\Gamma_0\times_{\Omega_B}\Delta)\setminus V(P_0)}$.
If $\Delta$ consists of $d$ disjoint copies of $\Gamma_B(J)$, then $-\chi(\Delta)=d\cdot(\rk(J)-1)$, and
$\Gamma\times_{\Omega_B}\Delta$ contains the trivial $d$-covering of $\Gamma$, given by the morphism
$\Gamma\to\Gamma_B(J)$ induced by $H\le J$.
\end{proof}

Now let us derive some further consequences. The equality $s\pi(w)=\pi(w)-1$ for every
$1\ne w\in\FF$, conjectured by Wilton and disproved by Ta Shma (see Section~\ref{subsec_HNC}, after
\eqref{eq_primitivity_rank}), holds for the subgroups that are not strictly compressed:

\begin{corollary}\label{cor_stable_primitivity_rank}
If $H\le\FF$ is finitely generated and not strictly compressed, $s\pi(H)=\pi(H)-1$.
\end{corollary}

\begin{proof}
As $H$ is not strictly compressed, $\pibar(H)=\pi(H)$; and, as shown right after
\eqref{eq_primitivity_rank}, $s\pi(H)\le s\pi_1(H)=\pi(H)-1$. Conversely, for every $d\in\N$ and every
$\Delta$ in the minimum defining $s\pi_d(H)$ in Definition~\ref{def_stable_primitivity_rank},
$\Gamma_B(H)\times_{\Omega_B}\Delta$ contains a $d$-covering of $\Gamma_B(H)$, so, by
Theorem~\ref{thm_jz_stability}, $s\pi_d(H)\ge\pibar(H)-1$.
\end{proof}

\noindent Wilton conjectured (see \cite[Conjecture 1]{hunt2015hanna}) that, for non-trivial finitely
generated $H,K\le\FF$,
\begin{equation}\label{eq_iehnc}
  (\rk(\inner{H,K})-1) \cdot (\rk(H \cap K)-1)
  \ \le \ (\rk(H)-1) \cdot (\rk(K)-1) ,
\end{equation}
and in \cite[Example 1]{hunt2015hanna}, Hunt disproved \eqref{eq_iehnc}, even up to a constant
factor.\footnote{By Schreier's index formula, \eqref{eq_iehnc} holds when $[\inner{H,K}:K]<\infty$ (also
proved in \cite{hunt2015hanna}).} Theorem~\ref{thm_jz_stability} implies the following variant
of \eqref{eq_iehnc}.

\begin{corollary}\label{cor_iehnc}
For all non-trivial finitely generated $H, K \le \FF$,
\[
  (\pibar(H)-1) \cdot
  \sum_{\substack{HgK \,:\, [H : H \cap gKg^{-1}] < \infty}}
  (\rk(H \cap gKg^{-1})-1)
  \ \le \ (\rk(H)-1) \cdot (\rk(K)-1) .
\]
\end{corollary}

\begin{proof}
Let $d\defeq\sum_{HgK\,:\,[H:H\cap gKg^{-1}]<\infty}[H:H\cap gKg^{-1}]$. By Schreier's index formula, the
sum in Corollary~\ref{cor_iehnc} equals $d\cdot(\rk(H)-1)$. By
\cite[4.1(d) and 7.5(b)]{Stallings1983TopologyOF}, an immersion of connected core graphs that induces an
inclusion of finite index $m$ of fundamental groups is an $m$-covering. Hence, if $d\ge1$, then, by the
explanation after \eqref{eq_strongly_inert_graphs}, ${\Gamma_B(H)\times_{\Omega_B}\Gamma_B(K)}$ contains a $d$-covering of
$\Gamma_B(H)$, so, by Theorem~\ref{thm_jz_stability},
${(\pibar(H)-1)}\cdot d\cdot{(\rk(H)-1)}\le{(\rk(H)-1)}\cdot{(\rk(K)-1)}$.
\end{proof}

In the definition \eqref{eq_Ivanov_stable_compressed_rank}, unlike in Definition~\ref{def_spibard},
$\Gamma\times_{\Omega_B}\Delta$ need not contain a covering of $\Gamma$, and
$s\pibar_{\textup{SHNC}}(H)$ can be much smaller than $\spivarbar{d}(H)=\pibar(H)-1$: in Example~\ref{example_shnpi_gap}, a
variant of \cite[Example 1]{hunt2015hanna}, ${\pibar(H)-1-s\pibar_{\textup{SHNC}}(H)}\ge{\rk(\FF)-3}$.

\begin{example}\label{example_shnpi_gap}
Let $k\ge2$, assume that $B=\{x,y,z_1,\ldots,z_k\}$, let $N_k\le\inner{x,y}$ be the kernel of the
map $\inner{x,y}\to\Z/k$ with $x\mapsto1$ and $y\mapsto0$, and let
$H\defeq N_k*\inner{z_1,\ldots,z_k}$. The images of $y,x^k,z_1,\ldots,z_k\in H$ in the abelianization
$\Z^{k+2}$ of $\FF$ are linearly independent, so $\rk(J)\ge k+2$ for every $J\ge H$, and
$\pibar(H)=k+2=\rk(\FF)$. The graph $\Gamma_B(N_k)$ is the $k$-cycle of $x$-edges with a $y$-loop at
every vertex, and $\Gamma_B(H)$ is $\Gamma_B(N_k)$ with loops labeled $z_1,\ldots,z_k$ at one vertex.
Hence, for $\Delta\defeq\Gamma_B(\inner{x,y})$, we have ${-\chi(\Delta)=1}$, $-\chi(\Gamma_B(H))=2k$ and
$-\chi(\Gamma_B(H)\times_{\Omega_B}\Delta)=-\chi(\Gamma_B(N_k))=k$, so, by the definition
\eqref{eq_Ivanov_stable_compressed_rank}, $s\pibar_{\textup{SHNC}}(H)\le 2k/k=2$, while
$\pibar(H)-1=k+1$.
\end{example}

%% file: polymatroid_table.tex
\phantomsection
\label{appendix_polymatroid_table}
\section*{The \texorpdfstring{$\Gamma$}{Gamma}-polymatroids constructed in this paper}

In every row but the last, $\hp$ is a \Binv $\Gamma$-polymatroid
(Definition~\ref{def_B_invariant_polymatroid}), and $v$ ranges over $V(\Gamma)$.
In the last row, $\hp$ is a $\Gamma$-polymatroid over $\Sigma$
(Definition~\ref{def_eta_polymatroid}), and $v$ ranges over the set $\Gamma_b$ of edges of
$\Gamma$ that are mapped to an edge $b$ of $\Sigma$.

\begingroup
\small
\begin{longtable}{@{}>{\raggedright\arraybackslash}p{0.17\textwidth}%
                    >{\raggedright\arraybackslash}p{0.28\textwidth}%
                    >{\raggedright\arraybackslash}p{0.18\textwidth}%
                    >{\raggedright\arraybackslash}p{0.27\textwidth}@{}}
\toprule
Where & $\hp(U)$ & $\hp(\{v\})$ & $\chi(\hp)$ \\
\midrule
\endfirsthead

\toprule
Where & $\hp(U)$ & $\hp(\{v\})$ & $\chi(\hp)$ \\
\midrule
\endhead

\bottomrule
\endlastfoot

Theorem~\ref{thm_spi_gap_unified} &
$\abs*{p_{\Delta}(p_{\Gamma}^{-1}(U))}$, for a core graph $\Delta$ and
$P\subseteq\Gamma\times_{\Omega_B}\Delta$ with $p_{\Gamma}$ a $d$-covering &
$d$ &
at least $\chi(\Delta)$, with equality when $p_{\Delta}$ is surjective \\
\addlinespace

Lemma~\ref{lemma_module_polymatroid} &
$\dim_K\mathcal{L}(U)$, for $\mathcal{L}$ from Definition~\ref{def_L_of_a_vertex} and
a quotient $\phi\colon W\twoheadrightarrow Q$ of an \Hmod{} &
$d$ &
at least $\chi(Q)$ \\
\addlinespace

Lemma~\ref{lemma_core_upper_bound} &
$\mathcal{H}_\mu\prn[\big]{\bigvee_{v\in U}\alpha_v\mid\mathcal{A}}$, for a generating
partition $\alpha$ fixed by $H\le\FF$ &
$\mathcal{H}_\mu(\alpha\mid\mathcal{A})$ &
at least $f_{\FF}(X,\mu/\mathcal{A})$ \\
\addlinespace

Proposition~\ref{thm_equations_in_group_actions} &
$\log\brackets*{G:\bigcap_{v\in U}G_v}$, for a transitive $G\acts X$ and
$f\colon V(\Gamma)\to X$ &
$\log\abs*{X}$ &
$\log\prn[\big]{\abs*{\mathcal{O}(f)}\cdot\PR_{\alpha}(\Gamma,f)}$ \\
\addlinespace

Definition~\ref{def_polymatroid_dependence_thm} &
$\abs*{\Delta_b}$ minus the number of connected components of $W^b_U$, for
$U\subseteq\Gamma_b$ &
$d-1$ \newline
(Proposition~\ref{prop_hp_of_single_edge}) &
$\chi(\Delta)-\chi(\Sigma)$ \newline
(Proposition~\ref{prop_euler_char_of_hp}) \\

\end{longtable}
\endgroup

%% file: glossary.tex
\phantomsection
\label{appendix_glossary}
\section*{Glossary of notation}
\begingroup
\small
\begin{longtable}{@{}>{\raggedright\arraybackslash}p{0.18\textwidth}%
                    >{\raggedright\arraybackslash}p{0.78\textwidth}@{}}
\toprule
Symbol & Meaning \\
\midrule
\endfirsthead

\toprule
Symbol & Meaning \\
\midrule
\endhead

\midrule
\multicolumn{2}{r}{\small\itshape Continued on next page} \\
\endfoot

\bottomrule
\endlastfoot

$\PR$, $\E$ & probability, expectation. \\

$\FF$, $B$, $H$, $w$ & a fixed free group, a fixed basis of $\FF$, a finitely generated subgroup of
$\FF$, a word in $\FF$. \\
$\F$, $K$ & a finite field, a field. \\
$n$, $q$ & the rank/degree parameter of a family of finite groups (e.g.\ $S_n$ or $\GL_n$), the size
of a finite field. \\

$G$, $X$, $\mathcal{O}$ & a finite group, a finite $G$-set, a $G$-orbit. \\
$Z^{S}$ & Usually, the Cartesian power (as in $\F_q^n$ or $A^{I}$),
unless $S$ is a group acting on $Z$; then it is
the set of common fixed points of $S$ in $Z$ (Def.~\ref{def_EHtoFSnActsND}). \\

$\Gamma,\Delta,\Sigma$ & graphs. \\
$V$, $E$ & the sets of vertices and of edges of a graph. \\
$U$ & a subset of vertices or edges in a graph, or of a $G$-set. \\
$\eta$ & a morphism between graphs. \\
$\src,\tar$ & the source and target maps of a graph (after Def.~\ref{def_stallings_core_graph}). \\
$\mathcal{E}_m$ & the standard $\F_q[\FF]$-basis $\{e_i\}_{i=1}^m$ of $\F_q[\FF]^m$ (before
Def.~\ref{def_M_beta}). \\

$\alpha$, $\beta$ & homomorphisms $\FF\to G$ and $H\to G$. \\
$\EHtoFVarActsVar{G}{X}$ & $\E\abs*{X^{\alpha(H)}}$ for a uniformly random $\alpha\in\Hom(\FF,G)$ (Def.~\ref{def_EHtoFSnActsND}). \\

$\Omega_B$ & the bouquet: one vertex, one $b$-labeled loop for each $b\in B$
(Def.~\ref{def_stallings_core_graph}). \\
$\Gamma_B(H)$ & the Stallings core graph of $H\le\FF$; $\Gamma_w\defeq\Gamma_B(\inner{w})$ (after
Def.~\ref{def_stallings_core_graph}). \\
$\hp$ & a polymatroid (Def.~\ref{def_polymatroid}), usually a $\Gamma$-polymatroid
(Defs.~\ref{def_B_invariant_polymatroid} and~\ref{def_eta_polymatroid}). \\
$\chi$ & Euler characteristic, of a graph or of a $\Gamma$-polymatroid
(Defs.~\ref{def_B_invariant_polymatroid} and~\ref{def_eta_polymatroid_euler_char}). \\

$\pi$, $\pibar$ & primitivity rank \eqref{eq_primitivity_rank}, compressed rank
\eqref{eq_compressed_rank}. \\
$s\pi$, $s\pi_K$ & stable primitivity rank and its $K$-analog
(Defs.~\ref{def_stable_primitivity_rank} and~\ref{def_stable_K_primitivity_rank}). \\
$\spivarbar{d}$, $s\pibar_{K,d}$ & $d$-stable compressed rank and its $K$-analog
(Defs.~\ref{def_spibard} and~\ref{def_spibarqd}). \\

\end{longtable}
\endgroup